\documentclass[leqno]{amsart}
\usepackage{amssymb,amsmath,amsthm}
\usepackage{color}
\usepackage{url}
\usepackage[Symbol]{upgreek}

\usepackage[all,arc,curve,color,frame,line]{xypic}

\usepackage{amssymb}
\usepackage{epic}
\usepackage{mathrsfs}
\usepackage{accents}
\usepackage{stmaryrd}

\usepackage{tikz}

\newcommand{\bbold}{\mathbb}

\def \LL{\mathbb{L}}
\def\H{{\bbold H}}
\def\R { {\bbold R} }
\def\Q { {\bbold Q} }

\def\N { {\bbold N} }
\def\T { {\bbold T} }

\def\c {\operatorname{c}}

\def \ex{\operatorname{e}}

\renewcommand\epsilon{\varepsilon}

\def \<{\langle}
\def \>{\rangle}

\def \supp {\operatorname{supp}}

\def \((  {(\!(}
\def \)) {)\!)}

\def \k {{{\boldsymbol{k}}}}

\DeclareMathSymbol{\precequ}{\mathrel}{symbols}{"16}
\DeclareMathSymbol{\succequ}{\mathrel}{symbols}{"17}

\newtheorem{theorem}{Theorem}[section]
\newtheorem{lemma}[theorem]{Lemma}
\newtheorem{prop}[theorem]{Proposition}
\newtheorem{cor}[theorem]{Corollary}

\numberwithin{theorem}{section}
\numberwithin{equation}{section}

\theoremstyle{definition}

\theoremstyle{remark}

\numberwithin{theorem}{section}
\numberwithin{equation}{section}

\newcommand{\mfL}{\mathfrak{L}}
\def \fM {{\mathfrak M}}
\def \fN{{\mathfrak N}}

\def \fS{{\mathfrak S}}
\def \fd {{\mathfrak d}}
\def \fm {{\mathfrak m}}
\def \fn {{\mathfrak n}}

\def \fv {{\mathfrak v}}

\def \fg{{\mathfrak g}}
\def \fG{{\mathfrak G}}
\def \fp{{\mathfrak p}}
\def \fq{{\mathfrak q}}

\def \inv{\operatorname{inv}}

\let\oldi\i
\let\oldj\j
\renewcommand\i{\relax\ifmmode{\boldsymbol{i}}\else\oldi\fi}
\renewcommand\j{\relax\ifmmode{\boldsymbol{j}}\else\oldj\fi}

\renewcommand\leq{\leqslant}
\renewcommand\geq{\geqslant}
\renewcommand\preceq{\preccurlyeq}
\renewcommand\succeq{\succcurlyeq}
\renewcommand\le{\leq}
\renewcommand\ge{\geq}
\renewcommand\frak{\mathfrak}

\DeclareMathAlphabet{\mathbf}{OML}{cmm}{b}{it}

\DeclareFontFamily{U}{fsy}{}
\DeclareFontShape{U}{fsy}{m}{n}{<->s*[.9]psyr}{}
\DeclareSymbolFont{der@m}{U}{fsy}{m}{n}
\DeclareMathSymbol{\der}{\mathord}{der@m}{182}

\DeclareSymbolFont{der@m}{U}{fsy}{m}{n}
\DeclareMathSymbol{\derdelta}{\mathord}{der@m}{100}

\DeclareSymbolFont{imag@m}{OT1}{cmr}{m}{ui}
\DeclareMathSymbol{\imag}{\mathord}{imag@m}{105}

\DeclareFontFamily{OMS}{smallo}{}
\DeclareFontShape{OMS}{smallo}{m}{n}{<->s*[.65]cmsy10}{}
\DeclareSymbolFont{smallo@m}{OMS}{smallo}{m}{n}
\DeclareMathSymbol{\smallo}{\mathord}{smallo@m}{79}

\DeclareFontFamily{OMS}{largerdot}{}
\DeclareFontShape{OMS}{largerdot}{m}{n}{<->s*[.8]cmsy10}{}
\DeclareSymbolFont{largerdot@m}{OMS}{largerdot}{m}{n}
\DeclareMathSymbol{\largerdot}{\mathord}{largerdot@m}{15}

\DeclareMathSymbol{\llambda}{\mathord}{der@m}{108}
\DeclareMathSymbol{\rrho}{\mathord}{der@m}{114}

\def \upo{\upomega}

\newcommand{\equationqed}[1]{\[\pushQED{\qed}#1 \qedhere\popQED\]\let\qed\relax}
\newcommand{\alignqed}[1]{\begin{align*}\pushQED{\qed} #1 \qedhere\popQED\end{align*}\let\qed\relax}

\def \No{\text{{\bf No}}}

\newcommand{\mathe}{\mathrm{e}}
\newcommand{\tmop}[1]{\ensuremath{\operatorname{#1}}}

\begin{document}

\title{Logarithmic Hyperseries}


\author[van den Dries]{Lou van den Dries}
\address{Department of Mathematics\\
University of Illinois at Urbana-Cham\-paign\\
Urbana, IL 61801\\
U.S.A.}
\email{vddries@illinois.edu}

\author[van der Hoeven]{Joris van der Hoeven}
\address{CNRS\\ LIX\\ \'Ecole Polytechnique\\
91128 Palaiseau Cedex\\
France}
\email{vdhoeven@lix.polytechnique.fr}

\author[Kaplan]{Elliot Kaplan}
\address{Department of Mathematics\\
University of Illinois at Urbana-Cham\-paign\\
Urbana, IL 61801\\
U.S.A.}
\email{eakapla2@illinois.edu}

\begin{abstract}{We define the field $\LL$ of logarithmic hyperseries, construct
on $\LL$ natural operations of differentiation, integration, and composition, establish the basic properties of these operations,
and characterize these operations uniquely by such properties.}
\end{abstract} 

\maketitle

\section{Introduction}\label{intro}  

\noindent
The field of transseries $\mathbb{T}$ was introduced independently by Dahn and
G{\"o}ring~\cite{DG} in model theory and by {\'E}calle~\cite{E} in his proof of the Dulac
conjecture. Roughly speaking, transseries are constructed from
the real numbers and a variable $x>\mathbb{R}$ using the field
operations, exponentiation, and taking logarithms and infinite sums.  Here is an example of
a transseries:
\[ 7 \mathe^{\mathe^x + \mathe^{x / 2} + \mathe^{x / 3} + \cdots} +
   \frac{\mathe^{\mathe^x}}{\log \log x}  + \sqrt{2} + \frac{1}{x}
   + \frac{2}{x^2} + \frac{6}{x^3} + \frac{24}{x^4} + \cdots + 
   \frac{\mathe^{- x}}{x} + \frac{\mathe^{- x}}{x^2} + \cdots . \]
The sign of a transseries is defined to be the sign of its leading
coefficient: $\tmop{sign} 7 >0$ in our
example;  $\mathbb{T}$ is real closed for the
corresponding ordering. See \cite[Appendix A]{ADH1} for a detailed construction.
The field $\mathbb{T}$ can also be equipped with natural `calculus'
operations: differentiation, integration, composition, and
functional inversion. The theory of $\mathbb{T}$ as a
valued differential field was determined in~\cite{ADH1}. In particular, it was shown there
that this theory is model complete. Remarkably,~$\mathbb{T}$ also  satisfies the
intermediate value property for differential polynomials: this was first proven in~\cite{JvdH} for the ordered differential
subfield of $\mathbb{T}$ consisting of the grid-based transseries, and extends to $\mathbb{T}$ itself by model
completeness.


Transseries describe `regular' orders
of growth of real functions. Despite its remarkable closure
properties, however, $\mathbb{T}$ cannot account for all regular orders of growth. For instance, Kneser~\cite{K} constructs a real analytic function $\mathe_{\omega}$ that satisfies the functional equation $\mathe_{\omega}
(x + 1) = \exp \mathe_{\omega} (x)$ and that grows regularly---its germ at $+\infty$ lies in a Hardy field---but faster than any
iterated exponential.  Its functional inverse is infinitely large, but grows slower than any iterated logarithm.

\bigskip\noindent
Accordingly, we wish to enlarge the field $\T$ of transseries to
a  field $\H$ of hyperseries  with transfinite iterates $\mathe_{\alpha}$ and $\ell_{\alpha}$ of $\mathe^x$ and
$\log x$
for all ordinals $\alpha$, and with natural operations of exponentiation, differentiation, integration, and composition. These operations should extend the corresponding
operations on $\T$. In this paper we achieve this for the purely logarithmic part $\LL$ of the intended $\H$ by direct recursive constructions, and with
exponentiation replaced by taking logarithms.  
  We also indicate how the natural embedding of $\T_{\log}$ into the field $\No$ of surreal numbers extends naturally 
to an embedding of $\LL$ into $\No$. As indicated in \cite{ADH2}, this is part
of a plan to eventually construct a canonical exponential field isomorphism $\H\cong \No$ via which  $\No$ can be equipped with the `correct' derivation and composition. Realizing this plan would vindicate the idea that $\H$ covers all regular orders of growth at infinity, as $\No$ does in a different way. 

A first step in the above direction is due to Schmeling~\cite{ Schm01} and his thesis advisor van der Hoeven. They constructed a
field of hyperseries that contains $\mathe_{\alpha}$ and~$\ell_{\alpha}$ for
{$\alpha < \omega^{\omega}$}, but they did not construct a derivation or composition on it. 
The purely logarithmic part of it will be recovered here as the subfield
$\mathbb{L}_{< \omega^{\omega}}$ of our $\LL$. 

  On a related topic, van der Hoeven's thesis~\cite{vdH:PhD} (with more details in~\cite{Schm01}) shows how to extend the derivation and composition
  on $\mathbb{T}$ to larger fields of transseries that contain elements such
  as $\mathe^{\sqrt{x} + \sqrt{\log x} + \sqrt{\log \log x} + \cdots}$. The recent paper~\cite{BM2} by Berarducci and Mantova shows how such generalized transseries naturally act on
  positive infinitely large surreal numbers so as to be compatible with
  composition and with the derivation on $\No$ constructed in~\cite{BM}.  While this line of work has some connection
  to the present paper, it goes into another~direction.

In the rest of this introduction we give canonical and precise
descriptions of $\LL$ with its `calculus' operations  and state its main properties. To prove existence and uniqueness
of the operations having these properties is not easy, and makes up the bulk of this paper.  First we define $\LL$ as an increasing union of
Hahn fields over $\R$. Throughout we let $\alpha,\beta,\gamma$ range over ordinals, an ordinal is identified with the set of smaller ordinals, and
$\alpha+\beta$ denotes the ordinal sum, to be thought of as
$\alpha$ followed by $\beta$.
Moreover, $m,n$, sometimes subscripted, range over $\N=\{0,1,2,\dots\}=\omega$.  By convention, a {\em differential field\/} has characteristic $0$;
given its derivation $\der$ and an element $y$ in the field we  also denote $\der(y)$ by $y'$, and $y'/y$ by $y^\dagger$.

\subsection*{The monomial group $\mathfrak{L}$} We fix once and for all symbols $\ell_{\alpha}$, one for each $\alpha$, with  
$\ell_{\alpha}\ne \ell_{\beta}$ whenever $\alpha\ne \beta$. The intended 
meaning of $\ell_{\alpha}$ is as the $\alpha$th iterated logarithm
of $x:=\ell_0$ in $\LL$, and accordingly we refer to these $\ell_{\alpha}$ as {\em hyperlogarithms}. (The totality of hyperlogarithms is too large to be a set; it is a proper class.
We shall freely use classes rather than sets when necessary: our set theory here is von Neumann-G\"odel-Bernays set theory with Global Choice (NBG), a conservative extension of ZFC in which all proper classes are in bijective correspondence with the class of all ordinals. Those who find these matters unpalatable may read {\em ordinal\/} as meaning
{\em countable ordinal}. Everything goes through with that restriction.)  

An {\em exponent sequence\/} is a family $(r_{\beta})$ of real numbers $r_{\beta}$, with $\beta$ ranging over all ordinals, such that
for some $\alpha$ we have $r_{\beta}=0$ for all $\beta\ge \alpha$. To each exponent sequence $r=(r_{\beta})$ we associate the formal
monomial $$\ell^r\ :=\ \prod_{\beta} \ell_{\beta}^{r_{\beta}},$$ a {\em logarithmic hypermonomial}. We make the class of logarithmic hypermonomials into an abelian (multiplicatively written) group $\mathfrak{L}$ with the obvious group operation: for exponent sequences $r=(r_{\beta})$ and $s=(s_{\beta})$ with corresponding logarithmic hypermonomials $\ell^r\ :=\ \prod_{\beta} \ell_{\beta}^{r_{\beta}}$ and $\ell^s\ :=\ \prod_{\beta} \ell_{\beta}^{s_{\beta}}$ we set $r+s:=(r_{\beta}+s_{\beta})$ and
$$\ell^r\cdot \ell^s\ :=\ \ell^{r+s}\ =\ \prod_{\beta} \ell_{\beta}^{r_{\beta}+s_{\beta}}.$$
The identity of $\mathfrak{L}$ is $1:= \ell^0$ with $0$ denoting the exponent sequence $(r_\beta)$ with $r_{\beta}=0$ for all $\beta$.
We make $\mathfrak{L}$ into a totally ordered abelian group by $\ell^r \prec \ell^s$ iff $r$ is lexicographically less than $s$, that is, $r\ne s$ and
$r_{\beta} < s_{\beta}$ for the least $\beta$ with $r_{\beta} \ne s_{\beta}$.
We identify $\ell_\alpha$ with $\ell^r$ where $r_{\alpha}=1$ and $r_{\beta}=0$ for all $\beta\ne \alpha$;  so $\ell_{\alpha}\succ 1$.  In this introduction we let $\fm, \fn$ range over logarithmic hypermonomials. 
We make $\R$ act on $\mathfrak{L}$: for $\fm=\prod_{\beta} \ell_{\beta}^{r_{\beta}}$ and $t\in \R$ we set 
$$\fm^t\ :=\ \prod_{\beta} \ell_{\beta}^{tr_{\beta}}\in \mathfrak{L}.$$ 
Thus we have the subgroup $\fm^\R:=\{\fm^t:\ t\in \R\}$ of $\mathfrak{L}$.
For $\fm=\ell^r$ we define $$\sigma(\fm)\ :=\ \{\beta:\ r_{\beta}\ne 0\},$$
a {\em set\/}  of ordinals (not just a class); we think of it as the support of $\fm$.
The set $$\mathfrak{L}_{<\alpha}\ :=\ \{\fm:\ \sigma(\fm) \subseteq \alpha\}$$
underlies an ordered subgroup of $\mathfrak{L}$. Note that
$$\mathfrak{L}_{<0}\ =\ \{1\}, \qquad \mathfrak{L}_{<1}\ =\ \ell_0^{\R}, \quad \dots \quad,\ \mathfrak{L}_{<n+1}\ =\ \ell_0^\R\cdots \ell_n^{\R}.$$ 
Given reals $r(\beta)$ for $\beta< \alpha$ we let
$\prod_{\beta< \alpha}\ell_\beta^{r(\beta)}$ denote the logarithmic 
hypermonomial $\ell^r$ where $r_{\beta}=r(\beta)$ for $\beta< \alpha$ and
$r_{\beta}=0$ for $\beta\ge \alpha$.  

\subsection*{The Hahn fields $\LL_{<\alpha}$}
The monomial group $\mathfrak{L}_{<\alpha}$ yields the ordered Hahn field
$$ \mathbb{L}_{<\alpha}\ :=\ \R[[\mathfrak{L}_{<\alpha}]]$$
consisting of the well-based series over $\R$ with monomials in
$\mathfrak{L}_{<\alpha}$. In particular, $\mathbb{L}_{<0}=\R$ and  
$\mathbb{L}_{<1}=\R[[\ell_0^{\R}]]$.  For $\beta\le \alpha$, we have 
$\mathfrak{L}_{<\beta} \subseteq \mathfrak{L}_{<\alpha}$, as ordered groups,
and so  $\mathbb{L}_{<\beta}\subseteq \mathbb{L}_{<\alpha}$, as ordered and valued 
fields.
We also set  $$\mathfrak{L}_{\le\alpha}\ :=\ \mathfrak{L}_{<\alpha+1}, \qquad
\mathbb{L}_{\le \alpha}\ :=\ \mathbb{L}_{<\alpha+1}\ =\ \R[[\mathfrak{L}_{\le \alpha}]].$$
Now $\LL:=\ \bigcup_{\alpha}\LL_{<\alpha}$ is an ordered and valued field
extension of each $\LL_{<\alpha}$. It does not have an underlying set, but it has an underlying proper class. 
We shall use the notations and conventions introduced
in \cite[Section 3.1 and Appendix A]{ADH1} to discuss these Hahn fields and their union $\LL$. (Section~\ref{prelim} below includes a summary of that material.) Thus for $f\in \LL^\times$ we have its dominant monomial $\fd(f)\in \mathfrak{L}\subseteq \LL$, with $f=c\fd(f)(1+\epsilon)$ for unique $c\in \R^\times$ and $\epsilon\prec 1$ (and
$\fd(0):=0\in \LL$ by convention), and $\R$ is viewed as an ordered subfield of $\LL$ and $\mathfrak{L}$ as an ordered subgroup of $\LL^{>}$. 

\subsection*{The logarithmic field $\LL$}
We define the logarithm $\log \fm$ of $\fm=\ell^r$ by
$$\log \fm\ :=\ \sum_{\beta}r_{\beta}\ell_{\beta+1}\in \LL.$$
Thus $\log \ell_{\alpha}=\ell_{\alpha+1}$, $\log \fm\fn=\log \fm + \log \fn$, and $\log \fm^t=t\log \fm$ for real $t$.  
For $f\in \LL^{>}$ we have $f=c\fd(f)(1+\epsilon)$ with $c\in \R^{>}$
and $\epsilon\prec 1$, and we set
$$\log f\ :=\ \log \fd(f) + \log c + 
\sum_{n=1}^{\infty} \frac{(-1)^{n-1}}{n}\epsilon^n,$$
where $\log c$ is the usual real logarithm of $c$.  
The map $f\mapsto \log f: \LL^{>}\to \LL$ is a strictly increasing
morphism of  the multiplicative ordered group $\LL^{>}$ into the ordered additive group of $\LL$. Note that if $\alpha$ is an infinite limit ordinal, then
$\log \LL_{<\alpha}^{>}\subseteq \LL_{<\alpha}$. 

\subsection*{The derivation on $\LL$} The intended derivation is
`derivative with respect to $x$' where $x:= \ell_0$. This derivation should respect logarithms and commute with infinite sums. To {\em respect logarithms\/} will be interpreted to mean that the derivative of $\ell_\alpha$ is $\prod_{\beta<\alpha} \ell_{\beta}^{-1}$. 
(Recall in this connection that the usual derivative of the $n$-times iterated real logarithm function $\log_n$ is $\prod_{m<n} (\log_m)^{-1}$.)
These requirements determine the derivation uniquely:

\begin{prop}\label{der} There is a unique $\R$-linear
derivation $\der$ on $\LL$ such that:
\begin{enumerate}
\item[(i)] $\der \ell_{\alpha}= \prod_{\beta<\alpha} \ell_{\beta}^{-1}$ for all $\alpha$; 
\item[(ii)] for every set $I$ and summable family
$(f_i)_{i\in I}$ in $\LL$ the family $(\der f_i)$ is summable as well and $\der \sum_i f_i= \sum_i \der f_i$.
\end{enumerate}
\end{prop}  

\noindent
The summability of a family $(f_i)$ in $\LL$ indexed by a set $I$ as in (ii) means: for some $\alpha$ all $f_i$ are in $\LL_{<\alpha}=\R[[\mathfrak{L}_{<\alpha}]]$ and $\sum_i f_i$ exists in this Hahn field. For $\alpha=0$, condition (i) says $\der x =1$. It is easy to see that the derivation 
of Proposition~\ref{der} must also respect logarithms in the sense that $\der \log f\ =\ \der f/f$ for all $f\in \LL^>$.   
We establish Proposition~\ref{der} in Section~\ref{derantider}, where we show in addition that the derivation $\der$ of that proposition has the following properties: 

\begin{theorem}\label{ci} $\{f\in \LL:\ \der f =0\}=\R$, $(\LL,\der)$ is an $H$-field, and $\der \LL= \LL$.  
\end{theorem} 

\noindent
Here $(\LL,\der)$ denotes the ordered field $\LL$ equipped with
the derivation $\der$. Recall from \cite[Chapter 10]{ADH1} that an $H$-field
is an ordered
differential field $K$ such that for the constant field $C$ of $K$ and all
$f\in K$ we have: if $f> C$, then $f'>0$, and, with $\mathcal{O}$ the convex hull of $C$ in $K$,  
if $f\in \mathcal{O}$, then $f=c+\epsilon$ for some $c\in C$ and $\epsilon\in K$
with $|\epsilon|< C^>$. Such an $H$-field $K$ is viewed as a 
{\em valued\/} field with valuation ring $\mathcal{O}$. 

\medskip\noindent
In the rest of this introduction $\LL$ is equipped with the above derivation $\der$. We also set $f':=\der f$, $f^{(n)}=\der^n f$ for $f\in \LL$ and introduce the distinguished integration
operator $f\mapsto \int f: \LL \to \LL$ that assigns to $f\in \LL$ the unique
$g\in \LL$ with $g'=f$ and $1\notin \supp g$; so the constant term of $\int f$ is $0$. For example, $\ell_{\alpha}=\int \prod_{\beta<\alpha} \ell_{\beta}^{-1}$. 

\subsection*{Composition} A good composition should reflect the composition of functions. To construct the `correct' composition on $\LL$ and show it has the desired properties takes considerable effort.
Let us define a {\em composition on $\LL$\/} to 
be an operation 
$$(f,g)\mapsto f\circ g\ :\ \LL\times \LL^{>\R} \to \LL$$ that has the following properties:
\begin{enumerate}
\item[(CL1)] for any $g\in \LL^{>\R}$ the map $f\mapsto f\circ g: \LL \to \LL$
is an $\R$-algebra endomorphism;
\item[(CL2)] $f\circ x=f$ for all $f\in \LL$ and $x\circ g=g$ for all $g\in \LL^{>\R}$;
\item[(CL3)] $\log(f\circ g)=(\log f)\circ g$ for all 
$f\in \LL^{>}$ and $g\in \LL^{>\R}$;
\item[(CL4)] for any summable family $(f_i)$ in $\LL$ and $g\in \LL^{>\R}$ the family
$(f_i\circ g)$ is summable and $(\sum_i f_i)\circ g=\sum_i f_i\circ g$;
\item[(CL5)] for all $f\in \LL$ and $g,h\in \LL^{>\R}$ we have $(f\circ g)\circ h=f\circ (g\circ h)$. 
\end{enumerate} 

\noindent
Note that (CL1) alone (and the fact that $\LL$ is real closed) gives that for fixed $g\in \LL^{>\R}$ the map $f\mapsto f\circ g: \LL\to \LL$ is an embedding of ordered fields sending $\LL^{>\R}$ into itself. Thus (CL3) and (CL5) make sense, assuming (CL1). 

\medskip\noindent
Thinking of $\ell_{\alpha}$ as the $\alpha$th iterated
logarithm of $\log x$ suggests 
$\ell_{\alpha}\circ \ell_{\beta}=\ell_{\beta+\alpha}$, but in view of $1+\omega=\omega$
this would give $\ell_\omega\circ \ell_1=\ell_{\omega}$ as a special case.
Since (CL2) gives $\ell_{\omega}\circ \ell_0=\ell_{\omega}$, this would be unreasonable, and in fact the composition we shall construct satisfies
$\ell_{\omega}\circ \ell_1=\ell_{\omega} -1$ instead. Our main result is the following characterization of this composition:

\begin{theorem}\label{uc} There is a unique composition $\circ$ on $\LL$
such that for all $f,g,h\in \LL$ with $g>\R$ and $g\succ h$ the sum $\sum_{n=0}^\infty \frac{f^{(n)}\circ g}{n!}h^n$ exists and
$$f\circ (g+h)\ =\ \sum_{n=0}^ \infty \frac{f^{(n)}\circ g}{n!}h^n\qquad (\text{Taylor expansion}),$$ 
and such that for all $\beta,\gamma$:
\begin{itemize}
\item $\ell_\gamma \circ \ell_{\omega^\beta}\ =\ \ell_{\omega^\beta+\gamma}$ if $\gamma < \omega^{\beta+1}$;
\item $\ell_{\omega^{\beta+1}} \circ \ell_{\omega^\beta}\ =\ \ell_{\omega^{\beta+1}}-1$;
\item the constant term of $\ell_{\omega^\gamma} \circ \ell_{\omega^\beta}$
is $0$ if $\gamma>\beta$ is a limit ordinal.
\end{itemize}
\end{theorem}

\noindent
We construct this composition in Sections~\ref{sec:comp1} and \ref{sec:comp2}, and use Sections~\ref{sec:pc} and ~\ref{teci} to prove the more subtle results about it: 
(CL5) (that is, associativity) and Taylor expansion.
In obtaining associativity we also establish the Chain Rule
(Proposition~\ref{l:FullChain}): $$(f\circ g)'\ =\ (f'\circ g)\cdot g' \text{ for all $f\in \LL$ and $g\in \LL^{>\R}$}.$$ 
All this concerns only the {\em existence\/} part of
Theorem~\ref{uc}. The {\em at most one\/} part is taken care of in the final Section~\ref{fs}. In the remainder of this introduction we let $\circ$ denote the composition on $\LL$ defined by Theorem~\ref{uc}. 

The $g\in \LL^{>\R}$ form a monoid under composition with $x$ as identity, and the invertible elements of
this monoid are the $g$ with $\min \sigma(\fd(g))=0$: Proposition~\ref{inversion}.

To construct our composition we  work inside Hahn fields $\LL_{<\alpha}$
where $\alpha=\omega^{\lambda}$ and $\lambda$ is an infinite limit ordinal,  and in fact, for such $\alpha$ we have
 $f\circ g\in \LL_{<\alpha}$ for $f,g\in \LL_{<\alpha}$ with $g>\R$; so the least $\alpha$ in this setting is $\omega^\omega$.

\medskip\noindent
Finally, we indicate in Section~\ref{fs} the natural ordered and valued field embedding of $\LL$ into $\No$ that is the identity on
$\R$, sends $x:= \ell_0$ to $\omega$, and respects logarithms and infinite sums: Proposition~\ref{embllno}. This is also a differential field
embedding where $\No$ is equipped with the derivation $\der_{\operatorname{BM}}$ constructed by Berarducci and Mantova~\cite{BM}.

\section{Preliminaries}\label{prelim}

\medskip\noindent
We summarize here some conventions, notations, and results concerning
monomial groups and Hahn fields and refer to \cite{vdH01} and
\cite[Section 3.1 and Appendix A]{ADH1} for proofs omitted here. We also consider some notions that are particularly useful in the
present paper and a planned sequel: multipliability, the support of linear operators on Hahn fields, Taylor deformations, and monomial groups with real powers. In addition we
include some miscellaneous facts needed later.

\subsection*{Monomial sets} 
A {\em monomial set\/} is a totally ordered set; we think of its elements as
monomials. 
Let $\fM$ be a monomial set and let $\fm, \fn$ range over elements of $\fM$.
Then $\fm \prec \fn$ indicates that $\fm$ is less than $\fn$ in the ordering
of $\fM$, and we use the notations $\fm\preceq \fn$, $\fm\succ \fn$, $\fm\succeq \fn$ likewise; for example, $\fm\preceq \fn\Leftrightarrow \fm\prec \fn$ or $\fm=\fn$. A set  $\mathfrak{S}\subseteq\fM$ is said to be {\em well-based\/} if it is well-ordered
in the reverse ordering, that is, there is no infinite strictly increasing sequence
$\fm_0 \prec \fm_1 \prec \fm_2 \prec \cdots$ in $\mathfrak{S}$. 

Let $\k$ be a field. Then $\k[[\fM]]$ consists of the formal series $f=\sum_{\fm} f_{\fm}\fm$ with coefficients $f_{\fm}\in \k$ whose support $$\supp f:=\{\fm:\ f_{\fm}\ne 0\}$$ is well-based. We construe $\k[[\fM]]$ as a vector space over $\k$ as suggested by the series notation and identify
$\fM$ with a subset of $\k[[\fM]]$ via $\fm\mapsto 1\fm$.

Let $(f_i)_{i\in I}$ be a family in $\k[[\fM]]$. We say that $(f_i)$ is
{\bf summable\/} if $\bigcup_i \supp f_i$ is well-based and
for each $\fm\in \fM$ there are only finitely many $i\in I$ such that 
$\fm\in \supp f_i$; in that case we define
its sum $\sum_i f_i$ to be the series $f\in \k[[\fM]]$ such that $f_{\fm}=\sum_i f_{i,\fm}$ for each $\fm\in \fM$. (This agrees with the usual notation for elements of $\k[[\fM]]$: for a series
$f=\sum_{\fm} f_{\fm}\fm\in \k[[\fM]]$ as above the family 
$(f_{\fm}\fm)$ is indeed summable with sum $f$; conversely, every summable family $(f_{\fm}\fm)$ with coefficients $f_{\fm}\in \k$ yields a
series $f=\sum_{\fm} f_{\fm}\fm\in \k[[\fM]]$.)  Instead of ``$(f_i)$ is summable'' we also say that $\sum_i f_i$ exists. Sometimes the following equivalence is useful: $(f_i)$ is not summable if and only if there is a sequence $(i_n)$ of distinct indices and an increasing sequence $(\fm_n)$
in $\fM$ with $\fm_n\in \supp(f_{i_n})$ for all $n$.  

The {\em dominant monomial\/} $\fd(f)\in \fM$ of a nonzero $f\in \k[[\fM]]$ is defined by $$\fd(f)\ :=\ \max \supp f.$$ We also set 
$\fd(0):= 0\in \k[[\fM]]$ and extend the ordering of $\fM$ to a total ordering on the disjoint union $\fM\cup\{0\}$ by $0\prec \fm$ for all $\fm\in \fM$. The binary relations $\prec$ and $\preceq$ on $\fM\cup\{0\}$ are extended to binary relations $\prec$ and $\preceq$ on $\k[[\fM]]$ as follows: 
$$f\prec g:\Leftrightarrow\ \fd(f)\prec \fd(g), \qquad f\preceq g:\Leftrightarrow\ \fd(f)\preceq \fd(g).$$ 
Let $\fN$ also be a monomial set and $\Phi: \k[[\fM]]\to \k[[\fN]]$ a map. 
We call $\Phi$ {\bf strongly additive\/} if it is additive and for every
summable family $(f_i)$ in $\k[[\fM]]$ the family $(\Phi(f_i))$ is summable in
$\k[[\fN]]$ and $\Phi(\sum_i f_i)=\sum_i \Phi(f_i)$. If $\Phi$ is strongly additive and $\Theta: \k[[\fM]]\to \k[[\fN]]$ is strongly additive, then so is $$\Phi+\Theta\ :\ \k[[\fM]]\to \k[[\fN]], \qquad f\mapsto \Phi(f)+\Theta(f).$$ 
If $\Phi$ is strongly additive, $\fG$ is a monomial set, and
$\Theta: \k[[\fG]]\to \k[[\fM]]$ is strongly additive, then so is
$\Phi\circ\Theta: \k[[\fG]]\to \k[[\fN]]$. 
We call $\Phi$ {\bf strongly $\k$-linear} if it is $\k$-linear and strongly additive; note that then
for any $f=\sum_{\fm} f_{\fm}\fm\in \k[[\fM]]$ the sum
$\sum_{\fm}f_{\fm}\Phi(\fm)$ exists in $\k[[\fN]]$ and equals $\Phi(f)$. 
Thus a strongly $\k$-linear map $\k[[\fM]]\to \k[[\fN]]$ is determined by
its restriction to $\fM$. The following converse is the ``totally ordered''
case of \cite[Proposition 3.5]{vdH01}:

\begin{lemma}\label{vdh} Let $\Phi: \fM\to \k[[\fN]]$ be such that for every well-based
$\mathfrak{S}\subseteq \fM$ the family $(\Phi(\fm))_{\fm\in \mathfrak{S}}$
is summable. Then $\Phi$ extends $($uniquely$)$ to a strongly $\k$-linear map
$\k[[\fM]]\to \k[[\fN]]$.
\end{lemma} 

\noindent
The next result on inverting strongly linear maps is almost the ``totally ordered'' case of \cite[Corollary 1.4]{ADH0}, which in turn follows from \cite[Theorems 6.1, 6.3]{vdH01}.   

\begin{lemma}\label{inverse} Let $\Phi: \k[[\fM]] \to \k[[\fM]]$ be a strongly $\k$-linear map with $\Phi(\fm)\prec \fm$ for all $\fm$. Let $I$ be the identity map on $\k[[\fM]]$. Then $I+\Phi: \k[[\fM]]\to \k[[\fM]]$ is bijective with strongly $\k$-linear inverse $(I+\Phi)^{-1}$ given by $(I+\Phi)^{-1}(f)=\sum_{n=0}^\infty(-1)^n\Phi^n(f)$, where the last sum always exists.
\end{lemma}
\begin{proof} For infinite $\k$ this is clear from \cite[Corollary 1.4]{ADH0}. For finite $\k$ we reduce to the previous case by extending $\k$ to an infinite field $K$ and using
Lemma~\ref{vdh} to extend $\Phi$ to a strongly $K$-linear map $K[[\fM]]\to K[[\fM]]$. \end{proof} 

\noindent
We only include the case of finite $\k$ for the sake of completeness, since
the results above only get applied in later sections of this paper for $\k$ of characteristic $0$.

\subsection*{Monomial groups and Hahn fields}
A {\em monomial group\/} is a monomial set $\fM$ equipped with a
 (multiplicatively written) group operation $\fM\times \fM\to \fM$ that makes $\fM$ into an ordered commutative group. Let $\fM$ be a monomial group. 
We indicate its identity by $1$ (or $1_{\fM}$ if we wish to specify $\fM$).
For sets $\mathfrak{S}_1, \mathfrak{S}_2\subseteq \fM$ we set
$$\mathfrak{S}_1\mathfrak{S}_2\ :=\ \{\fm\fn:\ \fm\in \mathfrak{S}_1,\ \fn\in \mathfrak{S}_2\},$$
and recall that if $\mathfrak{S}_1, \mathfrak{S}_2$ are well-based, then
so is $\mathfrak{S}_1\mathfrak{S}_2$, and for every $\fg\in\mathfrak{S}_1\mathfrak{S}_2$ there are only finitely many pairs $(\fm, \fn)\in \mathfrak{S}_1\times \mathfrak{S}_2$ with $\fg=\fm\fn$. 
For $\mathfrak{S}\subseteq \fM$ we define $\mathfrak{S}^n\subseteq \fM$ by recursion on $n$ 
by $\mathfrak{S}^0=\{1\}$, $\mathfrak{S}^{n+1}=\mathfrak{S}^n\mathfrak{S}$, and we also set $\mathfrak{S}^{\infty}:= \bigcup_n \mathfrak{S}^n$, the submonoid of $\fM$ generated by $\mathfrak{S}$. Recall Neumann's Lemma: if $\mathfrak{S}\subseteq \fM^{\preceq 1}$ is well-based, then so is $\mathfrak{S}^{\infty}$; if $\mathfrak{S}\subseteq \fM^{\prec 1}$ and
$\fg\in \mathfrak{S}^{\infty}$, then there are only finitely many tuples $(n,\fm_1,\dots, \fm_n)$ with $\fm_1,\dots, \fm_n\in \mathfrak{S}$ and $\fg=\fm_1\cdots\fm_n$.

Let $\k$ be a field. Recall from \cite[Section 3.1]{ADH1} how $\k[[\fM]]$ is then construed as a field extension of $\k$ with $\fM$ a subgroup of its multiplicative group.

 \begin{cor}\label{addsumm} Suppose $(\epsilon_i)_{i\in I}$ is a summable family in $\k[[\fM]]^{\prec 1}=\k[[\fM^{\prec 1}]]$. Then the family 
$(\epsilon_i^n)_{i\in I, n\ge 1}$ is summable, and so is the family 
$(\sum_{n=1}^\infty c_{in}\epsilon^n)_{i\in I}$ for any family $(c_{in})_{i\in I, n\ge 1}$ of coefficients in $\k$.
\end{cor}
\begin{proof} The first part is an easy consequence of Neumann's Lemma, and the second part follows from the first part.
\end{proof}

\noindent
We shall often use the following result whose proof is routine: 

\begin{lemma}\label{prfam} Suppose $(f_i)$ and $(g_j)$ are summable families in $\k[[\fM]]$. Then $(f_ig_j)$ is summable and $\sum_{i,j} f_ig_j=(\sum_i f_i)(\sum_j g_j)$. 
\end{lemma}

\noindent
Thus for $f\in \k[[\fM]]$ the map $g\mapsto fg: \k[[\fM]]\to \k[[\fM]]$ is strongly $\k$-linear. Given also a monomial group $\fN$ we have:

\begin{cor}\label{corprfam} Let $\Phi: \k[[\fM]]\to \k[[\fN]]$ be
 strongly additive, $(f_n)$ a summable family in $\k[[\fM]$ and
$\varepsilon\in \k[[\fN]]^{\prec 1}$. Then $\sum_n \Phi(f_n)\varepsilon^n$ exists. 
\end{cor} 
\begin{proof} Use Lemma~\ref{prfam} and the summability of $\big(\Phi(f_n)\big)$
and $(\varepsilon^n)$. 
\end{proof}

\noindent
We call $\k[[\fM]]$ a Hahn field over $\k$; it 
is a valued field with valuation ring $$\mathcal{O}\ =\ \{f\in \k[[\fM]]:\ f\preceq 1\}$$ and maximal ideal $\smallo=\{f\in \k[[\fM]]:\ f\prec 1\}$ of $\mathcal{O}$. For the corresponding valuation $v$ on $\k[[\fM]]$ and $f,g\in \k[[\fM]]$ we have
$$ f\preceq g\Leftrightarrow v(f)\ge v(g), \quad f\prec g\Leftrightarrow v(f) > v(g).$$
For $f\in \k[[\fM]]$ we have the decomposition $f=f_{\succ} + f_1 + f_{\prec}$ where $f_{\succ}:=\sum_{\fm\succ 1} f_{\fm}\fm$ is the {\em purely infinite part of $f$} and $f_{\prec}:=\sum_{\fm\prec 1} f_{\fm}\fm$ is the infinitesimal part of $f$. We also set $f_{\preceq 1}:= f_1+f_{\prec 1}$. 

If $\k$ is given as an ordered field (for example when $\k=\R$), then we equip $\k[[\fM]]$ with the field ordering such that $f>0\Leftrightarrow f_{\fd(f)}>0$ (for $f\in \k[[\fM]]^{\ne}$) and refer to the resulting ordered field extension $\k[[\fM]]$ of $\k$ as an {\em ordered Hahn field}.

\subsection*{Substitution in ordinary power series} Let $\k$ be a field and 
$\fM$ a monomial group. Let $t=(t_1,\dots, t_n)$ be a tuple of distinct variables and let
$$F\ =\ F(t)\ =\ \sum_{\nu}c_{\nu}t^{\nu}\in \k[[t]]\ :=\ \k[[t_1,\dots, t_n]]$$
be a formal power series over $\k$, the sum ranging over all $\nu=(\nu_1,\dots, \nu_n)\in \N^n$, and
$c_{\nu}\in \k$, $t^{\nu}:= t_1^{\nu_1}\cdots t_n^{\nu_n}$. For any tuple $\varepsilon=(\varepsilon_1,\dots, \varepsilon_n)$ of
elements of $\smallo=\k[[\fM]]^{\prec 1}$ the family $(c_{\nu}\epsilon^{\nu})$ is  summable, where $\varepsilon^{\nu}:= \varepsilon_1^{\nu_1}\cdots \varepsilon_n^{\nu_n}$ (Neumann's Lemma). Put
$$F(\varepsilon)\ :=\ \ \sum_{\nu}c_{\nu}\varepsilon^{\nu}\in \mathcal{O}\ =\  \k[[\fM]]^{\preceq 1}\ =\ \k[[\fM^{\preceq 1}]].$$
Fixing $\varepsilon$ and varying $F$ we obtain a $\k$-algebra morphism
$$F \mapsto F(\varepsilon)\ :\ \k[[t]] \to \k[[\fM]].$$ 
{\em In the rest of this subsection we assume that $\k$ characteristic $0$ and identify $\Q$ with a subfield of $\k$ in the usual way}.
Then we have the formal power series
$$ \exp(t)\ :=\ \sum_{i=0}^\infty t^i/i!\in \mathbb{Q}[[t]], \qquad \log (1+t)\ :=\ \sum_{j=1}^\infty (-1)^{j-1}t^j/j\in \mathbb{Q}[[t]]$$
in a single variable $t$. In $\mathbb{Q}[[t_1,t_2]]\subseteq \k[[t_1,t_2]]$  we have the identities
$$\exp(t_1+t_2)\ =\ \exp(t_1)\exp(t_2), \quad \log(1+t_1+t_2+t_1t_2)\ =\ \log(1+t_1)+\log(1+t_2).$$
Also
$\log\big(\exp(t)\big) =t$ and 
$\exp\big(\log(1+t))= 1+t$ 
in $\mathbb{Q}[[t]]\subseteq \k[[t]]$. Substituting elements of $\k[[\fM]]^{\prec 1}$ in these identities yields that 
$$h\ \mapsto\ \exp(h)\ =\ \sum_{i=0}^\infty h^i/i!\ :\ \k[[\fM]]^{\prec 1} \to 1+\k[[\fM]]^{\prec 1},$$
is an isomorphism of the additive subgroup $\k[[\fM]]^{\prec 1}$ of $\k[[\fM]]$
onto the multiplicative subgroup $1+\k[[\fM]]^{\prec 1}$ of $\k[[\fM]]^\times$, with inverse
$$1+\varepsilon\ \mapsto\ \log (1+\varepsilon)\ =\ \sum_{j=1}^\infty (-1)^{j-1}\varepsilon^j/j\ :\ 1+\k[[\fM]]^{\prec 1}\to \k[[\fM]]^{\prec 1}.$$

\begin{cor}\label{summlogexp} Let $(\epsilon_i)_{i\in I}$ be a family in $\k[[\fM]]^{\prec 1}$. Then
$$(\epsilon_i) \text{ is summable}\  \Longleftrightarrow\ \big(\log(1+ \epsilon_i)\big) \text{ is summable}.$$
\end{cor}
\begin{proof} The direction $\Rightarrow$ is a special case of Corollary~\ref{addsumm}. For $\Leftarrow$, apply that corollary to
the case $c_{in}:=1/n!$ using $-1+\exp(\log(1+ \epsilon_i))=\epsilon_i$.
\end{proof}

\subsection*{Multipliability}  Let $\k$ be a field of characteristic $0$ and $\fM$ a monomial group. Let $(\epsilon_i)_{i\in I}$ be a family of elements in $\k[[\fM]]^{\prec 1}$.  We declare $E$ to range over the finite subsets of $I$ and would like to define 
$\prod_i (1+\epsilon_i)$ as the
sum over all $E$ of the products $\prod_{i\in E} \epsilon_i$. This would require the family
$\big(\prod_{i\in E} \epsilon_i)_E$ to be summable, and thus in particular its subfamily $(\epsilon_i)_{i\in I}$ to be
summable. By Corollary~\ref{summlogexp} the summability of $(\epsilon_i)_{i}$ is equivalent to that of $\big(\log(1+\epsilon_i)\big)_i$. Moreover: 

\begin{lemma}\label{multip} Suppose $(\epsilon_i)_{i\in I}$ is summable.  Then the family
$\big(\prod_{i\in E} \epsilon_i)_E$ is also summable and $\exp\big(\sum_i \log(1+\epsilon_i)\big)=\sum_E\prod_{i\in E} \epsilon_i$. 
\end{lemma}
\begin{proof} The summability of $\big(\prod_{i\in E} \epsilon_i)_E$ follows from Neumann's Lemma: use that for $|E|=n$ we have $\supp \prod_{i\in E} \epsilon_i \subseteq (\bigcup_{i\in I} \supp\epsilon_i)^n$. Next, the desired identity holds  for
finite $I$, and then follows easily for arbitrary $I$ using similar reasoning as needed for summability of $\big(\prod_{i\in E} \epsilon_i)_E$.
\end{proof} 

\noindent
Accordingly we say that the family $(1+\epsilon_i)$ is {\bf multipliable\/} if $(\epsilon_i)$ is summable (equivalently,  $\big(\log(1+\epsilon_i)\big)_i$ is summable), and
in that case we set 
$$\prod_i(1+\epsilon_i)\ :=\ \sum_E\prod_{i\in E} \epsilon_i\ =\ 1+\sum_{E\ne \emptyset}\prod_{i\in E}\epsilon_i\  \in\ 1+\k[[\fM]]^{\prec 1},$$
with  $\log \prod_i(1+\epsilon_i)=\sum_i \log(1+\epsilon_i)$.
 Instead of calling $(1+\epsilon_i)$ multipliable we also say that  $\prod_i(1+\epsilon_i)$ exists.
The basic facts about these infinite products follow easily from corresponding
facts about infinite sums by taking logarithms.

\subsection*{A useful identity} It is routine to check that for any elements $g_1, g_2, g_3,\dots$ in a~field $K$ of characteristic $0$ we have an identity
$$\log \big(1+\sum_{n=1}^\infty\frac{g_n}{n!}t^n\big)\ =\ \sum_{n=1}^{\infty}\frac{L_n(g_1,\dots, g_n)}{n!}t^n$$
in the ring $K[[t]]$ of formal power series over $K$, where the
$L_n\in \Q[X_1,\dots, X_n]$ are polynomials independent of the sequence $g_1, g_2, g_3,\dots$. The $L_n$ are the {\em logarithmic polynomials\/}
from \cite[p. 140]{C}, but we don't need further details given there about them. In the later subsection on Taylor deformations we shall use the following:

\begin{lemma}\label{polyid} Let $K$ be a differential field, $y\in K^\times$, and $n\ge 1$. Then
$$ (y^\dagger)^{(n-1)}\ =\ L_n\left(\frac{y'}{y},\dots, \frac{y^{(n)}}{y}\right) .$$
\end{lemma}
\begin{proof} If these identities hold for some $y$ that is differentially transcendental (over $\Q$), then they hold for all $y$ as in the lemma. Take a real analytic function $f: I \to \R$ on a nonempty open interval $I\subseteq \R$ such that $f$ is differentially transcendental and everywhere positive. (Thus
$f$ lies in the differential fraction field of the differential domain of real analytic functions on $I$.) For $a\in I$ the Taylor series of
$f$ at $a$ is the formal series $\sum_n \frac{1}{n!}f^{(n)}(a)t^n\in \R[[t]]$. Likewise, the Taylor series of $\log f$ at $a$ is 
$$ \sum_{n=0}^\infty \frac{1}{n!}(\log f)^{(n)}(a)t^n\ =\ \log f(a) + \sum_{n=1}^\infty \frac{1}{n!} (f^\dagger)^{(n-1)}(a)t^n.$$
Now $f=f(a)\cdot\big(1+\frac{f-f(a)}{f(a)}\big)$, so $\log f=\log f(a) + 
\log\big(1+\frac{f-f(a)}{f(a)}\big)$, the Taylor series
of $\frac{f-f(a)}{f(a)}$ at $a$ is $\sum_{n=1}^\infty \frac{1}{n!}
\frac{f^{(n)}(a)}{f(a)}t^n$, so the Taylor series of $\log f$ at $a$ also equals 
$$\log f(a) + \sum_{n=1}^\infty \frac{1}{n!}L_n\big(\frac{f'(a)}{f(a)},\dots, \frac{f^{(n)}(a)}{f(a)}\big)t^n.$$
This yields $(f^\dagger)^{(n-1)}(a)=L_n\big(\frac{f'(a)}{f(a)},\dots, \frac{f^{(n)}(a)}{f(a)}\big)$ for all $a\in I$, that is,
$(f^\dagger)^{(n-1)}=L_n\big(\frac{f'}{f},\dots, \frac{f^{(n)}}{f}\big)$,
which gives the desired result.
\end{proof}

\noindent
This lemma can also be proved more formally by expressing the $L_n$ in terms of the Bell polynomials as in \cite[p.140]{C}, but the details would take up considerable space.

\subsection*{The support of a linear operator} This notion will play a role similar to that of the norm of a linear operator on a 
Banach space. Let $\k$ be a field, $\fM$ a monomial group, $\fG$ a subset of $\fM$, and let a map $S: \fG \to \k[[\fM]]$ be given. Then we define the (operator) support of $S$, denoted by $\supp S$, to be the smallest set $\mathfrak{S}\subseteq \fM$ such that
$\supp S(\fg)\subseteq \mathfrak{S}\fg$ for all $\fg\in \mathfrak G$. 
The proof of the next lemma is routine.

\begin{lemma}\label{supp1} Suppose $\supp S$ is well-based. Then $S$ extends uniquely to
a strongly $\k$-linear operator $\k[[\fG]]\to \k[[\fM]]$. Denoting this extension
also by $S$, we have $\supp S(f)\subseteq (\supp S)(\supp f)$ for all $f\in \k[[\fG]]$.  \end{lemma}

\noindent
For a strongly $\k$-linear map $T: \k[[\fG]]\to \k[[\fM]]$ we define
$\supp T$ as the support of its restriction to $\fG$. 
If $\mathfrak G=\fM$ and $\supp S$ is well-based, then we have for each $n$ the strongly $\k$-linear operator $S^n: \k[[\fM]]\to \k[[\fM]]$ with $\supp S^n\subseteq (\supp S)^n$. Simple applications of Neumann's Lemma give:

\begin{lemma}\label{supp2} Suppose $\mathfrak G=\fM$ and $\supp S$ is well-based. Let $h\in \k[[\fM]]$ be such that 
$(\supp S)(\supp h)\prec 1$, and let $(s_n)$ be any sequence in $\k$. Then
$\sum_{n=0}^\infty s_nS^n(\fm)h^n$ exists for all $\fm$, and the 
map $P\ :\ \fM\to \k[[\fM]]$ given by 
$P(\fm):=\sum_{n=0}^\infty s_nS^n(\fm)h^n$ 
has well-based support $\supp P\subseteq \big((\supp S)(\supp h)\big)^\infty$.
\end{lemma}

\begin{lemma}\label{supp3} If $T: \k[[\fG]]\to \k[[\fM]]$ is $\k$-linear,
$\mathfrak{S}\subseteq \fM$ is well-based, and $\supp T(f) \subseteq \mathfrak{S}\cdot \supp f$ for all $f\in \k[[\fG]]$, then $T$ is strongly $\k$-linear.
\end{lemma}

\noindent
Thus with the hypothesis and notation of Lemma~\ref{supp2} the sum
$\sum_{n=0}^\infty s_nS^n(f)h^n$ exists for all $f\in \k[[\fM]]$ and the
map $T\ :\ \k[[\fM]]\to \k[[\fM]]$ given by
$$  T(f)\ :=\ \sum_{n=0}^\infty s_nS^n(f)h^n $$
is the unique strongly $\k$-linear operator $\k[[\fM]]\to \k[[\fM]]$
that extends $P$. Moreover, $\supp T(f)\subseteq \big((\supp S)(\supp h)\big)^\infty\cdot \supp f$ for $f\in \k[[\fM]]$. In the next lemma, an easy variant of Lemma~\ref{inverse}, we
let $I$ be the identity map on $\k[[\fM]]$.  

\begin{lemma}\label{invsupp} Suppose $D: \k[[\fM]]\to \k[[\fM]]$ is strongly $\k$-linear and
$\supp D$ is well-based and $\supp D\prec 1$. Then $I+D: \k[[\fM]]\to \k[[\fM]]$ is bijective with
strongly $\k$-linear inverse $(I+D)^{-1}= I +E$, where $E: \k[[\fM]]\to \k[[\fM]]$ is strongly $\k$-linear, $\supp E\subseteq \bigcup_{n=1}^\infty (\supp D)^n\prec 1$, and
$E(f)=\sum_{n=1}^\infty (-1)^nD^n(f)$ for $f\in \k[[\fM]]$.  
\end{lemma}

\subsection*{Taylor Deformations} Let $\k$ be a field of characteristic $0$ and $\fM$ a subgroup of the monomial group $\fN$, so $\k[[\fM]]$ is a subfield of $\k[[\fN]]$. Let there be given a $\k$-linear derivation
$\der$ on $\k[[\fM]]$ with well-based support $\supp \der\prec 1$
and a strongly $\k$-linear
field embedding $\Phi: \k[[\fM]]\to \k[[\fN]]$. 
Let $\varepsilon\in \k[[\fN]]^{\prec 1}$. Then for $f\in \k[[\fM]]$ the sum 
$\sum_n \frac{\der^n f}{n!}$ exists in $\k[[\fM]]$ by the remark following Lemma~\ref{supp3}, hence
$$\sum_{n=0}^\infty \frac{\Phi(\der^n f)}{n!}\varepsilon^n\ =\ \Phi(f) + \Phi(\der f)\varepsilon + \frac{\Phi(\der^2 f)}{2}\varepsilon^2 + \cdots$$
exists in $\k[[\fN]]$ by Corollary~\ref{corprfam}. This yields a $\k$-linear (Taylor) map
$$T\ :\ \k[[\fM]]\to \k[[\fN]], \qquad T(f)\ :=\ \sum_{n=0}^\infty \frac{\Phi(\der^n f)}{n!}\varepsilon^n.$$
For $\varepsilon =0$ we have $T=\Phi$; in general we view $T$ as a {\em deformation\/} of $\Phi$. 


\begin{lemma}\label{tay1} $T: \k[[\fM]]\to \k[[\fN]]$ is a strongly $\k$-linear field embedding.
\end{lemma}
\begin{proof} It is routine to check that $T(1)=1$ and $T(fg)=T(f)T(g)$
for $f,g\in \k[[\fM]]$. Let the family $(f_i)$ in $\k[[\fM]]$ be summable. Then so is
the family $(\der^n f_i)_{i,n}$ as is easily verified. Hence $\sum_{i,n}\frac{\Phi(\der^n f_i)}{n!}$ exists.  To derive 
that $\sum_i T(f_i)$ exists and equals $T(\sum_i f_i)$, 
use Lemma~\ref{prfam} and regroup terms. 
\end{proof}

\noindent
To show that suitable logarithm maps on $\k[[\fN]]$ that commute with $\Phi$ also commute with its deformation $T$ we assume
for the next lemma that $\k$ is an {\em ordered\/} field (so $\k[[\fM]]$ and $\k[[\fN]]$ are ordered Hahn fields over $\k$), and that $\Phi$ is an embedding of {\em ordered\/} and {\em valued\/} fields. 
In addition
we assume that $\k[[\fN]]$ is equipped with a map $\log :\k[[\fN]]^{>} \to \k[[\fN]]$ such that $\log(1+h)=\sum_{n=1}^\infty \frac{(-1)^{n-1}}{n}h^n$
for $h\in \k[[\fN]]^{\prec 1}$, $\log(fg) = \log(f)+\log(g)$ for $f,g\in \k[[\fN]]^{>}$, and 
$\log \k[[\fM]]^{>} \subseteq \k[[\fM]]$. 

\begin{lemma}\label{l:EqualLogs}
Suppose $f \in \R[[\fM]]^{>}$, $(\log f)'=f^\dagger$, and $\log \Phi(f) = \Phi(\log f)$. Then $$\log T(f)\ =\ T(\log f).$$
\end{lemma}
\begin{proof} From $T(f) = \Phi(f)\cdot \left(1+\sum_{n=1}^\infty \frac{1}{n!}\Phi\left(\frac{f^{(n)}}{f}\right)\varepsilon^n\right)$ we obtain
\[
\log T(f)\ =\ \log \Phi(f) + \log\left(1+\sum_{n=1}^\infty \frac{1}{n!}\Phi\left(\frac{f^{(n)}}{f}\right)\varepsilon^n\right).
\]
Using $\log \Phi(f) = \Phi(\log f)$ and $f^{(n)}\preceq f$ for all $n$, this yields
\begin{align*}
\log T(f)\ &=\  \Phi(\log f)+ \sum_{n=1}^\infty \frac{1}{n!}L_n\left(\Phi\left(\frac{f'}{f}\right),\ldots,\Phi\left(\frac{f^{(n)}}{f}\right)\right)\epsilon^n\\
&=\  \Phi(\log f)+ \sum_{n=1}^\infty \frac{1}{n!}\Phi\left(L_n\left(\frac{f'}{f},\ldots,\frac{f^{(n)}}{f}\right)\right)\epsilon^n.
\end{align*}
Lemma~\ref{polyid} gives 
$L_n\left(\frac{f'}{f},\ldots,\frac{f^{(n)}}{f}\right) = (\log f)^{(n)}$
for $n\ge 1$, so
\[
\log T(f)\ =\  \Phi(\log f)+ \sum_{n=1}^\infty \frac{\Phi\left(\log(f)^{(n)}\right)}{n!}\epsilon^n\ =\ T(\log f). \qedhere
\]
\end{proof}

\noindent
Suppose next that $\der$ comes with a strongly
$\k$-linear extension to a derivation on $\k[[\fN]]$, also denoted by $\der$, and
that the embedding $\Phi$ obeys a `chain rule' in the sense that we are given an element $\phi\in \k[[\fN]]$ such that $\der\Phi(f)=\Phi(\der f)\cdot \phi$ for all $f\in \k[[\fM]]$. Then a routine computation yields also a chain rule for $T$:

\begin{lemma}\label{tay2} $\der(Tf)\ =\ T(\der f)\cdot(\phi+\der \varepsilon)$ for all $f\in \k[[\fM]]$.
\end{lemma}

\subsection*{Monomial groups with real powers} Let the monomial group $\fM$ have real powers, that is, it is equipped with an operation
$(s,\fm)\mapsto \fm^s\ :\ \R\times \fM\to \fM$ such that for all $s,t\in \R$ and all $\fm, \fn$ we have 
$$\fm^1\ =\ \fm,\quad \fm^{s+t}\ =\ \fm^s\fm^t,\quad (\fm^s)^t\ =\ \fm^{st},\quad (\fm\fn)^s\ =\ \fm^s\fn^s\quad \text{ (so $\fm^0=1$)}.$$ 
Then we extend this operation to a power operation 
$$(s,f)\mapsto f^s\ :\  \R\times \R[[\fM]]^{>} \to \R[[\fM]]^{>}$$ as follows: first, if $f=1+\epsilon$ with
$\epsilon\prec 1$, then we set 
$$f^s\ :=\ \exp(s\log f)\ =\ \sum_{n=0}^\infty\binom{s}{n} \epsilon^n\ \in\ 1+\R[[\fM]]^{\prec 1},$$
and for $f=c\fd(f)(1+\epsilon)$ ($c\in \R^{>}$, $\epsilon\prec 1$), we set
$$f^s\ :=\ c^s\fd(f)^s(1+\epsilon)^s\ \in\ \R[[\fM]]^{>}$$
where $c^s$ has the usual value in $\R^{>}$. It is easy to verify that
then for all $s,t\in \R$ and $f,g\in \R[[\fM]]^{>}$ we have
$$ f^0\ =\ 1,\quad f^1\ =\ f,\quad f^{s+t}=f^sf^t,\quad (f^s)^t=f^{st},\quad (fg)^s=f^sg^s.$$
In the introduction we introduced the $\frak{L}_{<\alpha}$ as monomial groups with real powers, and also defined a logarithm map on $\LL^{>}$. Note that the definitions given there lead to $\log f^t=t\log f$ for $f\in \LL^{>}$ and $t\in \R$.

\subsection*{A useful well-ordering} Let $\N^\N$ be lexicographically ordered and
consider the set $D$ of all sequences $(d_0, d_1, d_2,\dots)\in \N^\N$ such that $d_0\ge d_1\ge d_2 \ge \cdots$ and $d_n=0$ for all
sufficiently large $n$. 

\begin{lemma}\label{well} $D$ is a well-ordered subset of $\N^{\N}$.
\end{lemma}
\begin{proof} Consider the map that assigns to any sequence
$(d_0, d_1, d_2,\dots)\in D$ the ordinal $\omega^{d_0} + \omega^{d_1} + \cdots + \omega^{d_m}$ if $m$ is such that $d_m\ne 0$ and $d_n=0$ for all $n>m$, and assigns to the sequence $(0,0,0,\dots)$ the ordinal $0$.  
Observe that this map
is injective and order preserving.  
\end{proof} 

\noindent
Next, let $D_{\infty}$ be the larger set of all sequences $(d_0, d_1, d_2,\dots)\in \N^\N$ such that $$d_0\ \ge\ d_1\ \ge\  d_2\ \ge\ \cdots.$$

\begin{cor}\label{morewell} $D_{\infty}$ is a well-ordered subset of $\N^{\N}$.
\end{cor}
\begin{proof} Any strictly decreasing infinite sequence in $D_{\infty}$
would be a sequence in $D_m:=\{(d_0, d_1,d_2,\dots)\in D_{\infty}:\ d_0\le m\}$ for some $m$, so it is enough to show that $D_m$ is well-ordered.
Now $D_m$ is the disjoint union of its subsets $D_{m,i}$, $i=0,\dots,m$,
where $D_{m,i}$ consists of the $(d_0, d_1,\dots)\in D_m$ with $d_n=i$
for all sufficiently large $n$, and it follows easily from Lemma~\ref{well} that each of the sets $D_{m,i}$ is well-ordered.
\end{proof}


\subsection*{Some more notation} For ordinals $\alpha< \gamma$ we let
$\mathfrak{L}_{[\alpha, \gamma)}$ be the convex subgroup of $\mathfrak{L}_{<\gamma}$ whose elements are the hypermonomials
$\prod_{\alpha\le \beta < \gamma}\ell_{\beta}^{r_{\beta}}$. This gives the
Hahn subfield $\LL_{[\alpha,\gamma)}:=\R[[\mathfrak{L}_{[\alpha, \gamma)}]]$ of $\LL_{<\gamma}$.  
Note that
$$\mathfrak{L}_{<\gamma}\ =\ \mathfrak{L}_{[\alpha, \gamma)}\cdot \mathfrak{L}_{<\alpha}\ \text{ with }\ \mathfrak{L}_{[\alpha, \gamma)}\cap \mathfrak{L}_{<\alpha}\ =\ \{1\}.$$ As in \cite[p. 713]{ADH1} this yields an identification of ordered fields
$$\LL_{<\gamma}\ =\ \LL_{[\alpha, \gamma)}[[\mathfrak{L}_{<\alpha}]],$$ that we shall use for certain $\alpha < \gamma$.

\section{Differentiating and Integrating in $\LL$}\label{derantider}

\noindent
In the Introduction we defined the Hahn fields
$\LL_{<\alpha}=\R[[\mathfrak{L}_{<\alpha}]]$ over $\R$ and their union $\LL$.
Using the preliminary section it is easy to verify the results stated in the Introduction up to 
(but not including) the subsection on the derivation of $\LL$.
Towards Proposition~\ref{der} we shall 
construct for every $\alpha$ a strongly $\R$-linear derivation $\der_{\alpha}$ on $\LL_{<\alpha}$; the derivation
$\der$ on $\LL$ will be the common extension of these $\der_{\alpha}$. The main work in this section is then to show that
$\der\LL=\LL$. 

\subsection*{The derivation}  We set
$$\ell_{\alpha}'\ :=\ \prod_{\beta<\alpha} \ell_{\beta}^{-1}\in\mathfrak{L}_{<\alpha},\qquad  \ell_{\alpha}^\dagger\ :=\  \prod_{\beta\le \alpha} \ell_{\beta}^{-1}\ =\ 
\ell_{\alpha+1}'
\in \mathfrak{L}_{\le\alpha}.$$
Note that $\ell_{\alpha}'\preceq 1$ and $\ell_{\alpha}^\dagger\preceq x^{-1}\prec 1$, and that for $\alpha > \beta$ we have
$$ \ell_{\alpha}'\ \prec\ \ell_{\beta}'\  \text{ in }\ \mathfrak{L}_{<\alpha}, \qquad 
 \ell_{\alpha}^\dagger \prec\ \ell_{\beta}^\dagger\  \text{ in }\  \mathfrak{L}_{\le\alpha}.$$
Next, we extend the above to any logarithmic hypermonomial $\fm=\ell^r$ by
$$\fm^\dagger\ :=\ \sum_{\beta} r_{\beta}\ell_{\beta}^\dagger, \qquad
   \fm'\  :=\ \fm \fm^\dagger\ =\ \sum_{\beta} r_{\beta}\fm\ell_{\beta}^\dagger.$$
Thus if $\fm\in \mathfrak{L}_{<\alpha}$, then $\fm^\dagger , \fm'\in \mathbb{L}_{<\alpha}$. As in the Introduction we let $\fm, \fn$ range over $\mathfrak{L}=\bigcup_{\alpha}\mathfrak{L}_{<\alpha}$. 

\begin{lemma}\label{6things} The following hold for all $\fm,\fn$: \begin{enumerate}
\item[(i)] $(\fm\fn)^\dagger=\fm^\dagger + \fn^\dagger$, and $(\fm^t)^\dagger=t\fm^\dagger$ for $t\in \R$;
\item [(ii)]$(\fm\fn)'=\fm'\fn + \fm\fn'$;
\item[(iii)] $\fm \ne 1\ \Rightarrow\ \fm', \fm^\dagger\ne 0$;
\item[(iv)] $\fm\prec 1,\ \fn\ne 1\ \Rightarrow\ \fm'\prec \fn^\dagger$;
\item[(v)] $\fm\prec \fn\ne 1\ \Rightarrow\ \fm'\prec \fn'$;
\item[(vi)]  $\fm\in \mathfrak{L}_{<\alpha}\ \Rightarrow\ \supp \fm'\  \subseteq\  \{\ell_{\beta}^{\dagger}:\ \beta< \alpha\}\fm$.
\end{enumerate}
\end{lemma} 
\begin{proof} This is mostly routine, and we only prove here (iv) and (v).
So assume $\fm\prec 1$ and $\fn\ne 1$. For $\beta=\min \sigma(\fm)$  we have $\fm=\ell^{r_{\beta}}_{\beta}\prod_{\beta< \rho<\alpha} \ell_{\rho}^{r_{\rho}}$, $r_{\beta}<0$, so
$$\fd(\fm')\ =\ \fm\ell_{\beta}^\dagger\ =\ \big(\prod_{\rho<\beta}\ell_{\rho}^{-1}\big)\cdot \ell^{r_{\beta}-1}_{\beta}\cdot \prod_{\beta< \rho<\alpha} \ell_{\rho}^{r_{\rho}}.$$
Also, for $\gamma=\min \sigma(\fn)$ we have $\fd(\fn^\dagger)=\prod_{\rho\le \gamma} \ell_{\rho}^{-1}$. By distinguishing the cases $\gamma< \beta$ and $\gamma\ge \beta$ and recalling that $r_{\beta}<0$ we get $\fd(\fm') \prec \fd(\fn^\dagger)$, so $\fm'\prec \fn^\dagger$. 

As to (v), assume $\fm\prec \fn\ne 1$. Then $\fm=\fn\fv$ with $\fv\prec 1$, so
$\fm'=\fn'(\fv + \fv'/\fn^\dagger)$. It remains to note that 
$\fv'\prec \fn^\dagger$ by (iv). 
\end{proof}

\noindent 
Item (vi) and Lemma~\ref{supp1} yield a unique strongly $\R$-linear derivation $\der_{\alpha}$ on $\LL_{<\alpha}$ such that
$\der_{\alpha}(\fm)=\fm'$ for all $\fm\in \mathfrak{L}_{<\alpha}$.
Note that (vi) and that lemma also gives
$$\supp \der_{\alpha}(f)\ \subseteq\ 
\{\ell_{\beta}^{\dagger}:\ \beta< \alpha\}\cdot \supp f\quad \text{ for }\ f\in \LL_{<\alpha}$$
and that $\{\ell_{\beta}^{\dagger}:\ \beta< \alpha\}$ is a well-based subset of $\mathfrak{L}_{<\alpha}$ with largest element $\ell_0^\dagger=x^{-1}$. 
In particular, $\supp \der_{\alpha}=\{\ell_{\beta}^{\dagger}:\ \beta< \alpha\}\preceq x^{-1}$,  so $\supp \der_{\alpha}$ is well-based.

 It is clear that for $\alpha > \beta$ the derivation $\der_{\alpha}$ extends $\der_{\beta}$.
Thus we have a common extension of the $\der_{\alpha}$ to a derivation $\der$ on $\LL$. This is the derivation
of Proposition~\ref{der}, which is thereby established. We set $f':= \der f$ and $f^{(n)}:= \der^n f$ for $f\in \LL$ and $g^\dagger:= g'/g$ for $g\in \LL^\times$; this creates no notational conflict, since for $f=\fm$ or
$g=\fm$ this agrees with the previously defined $\fm'$ and $\fm^\dagger$. 
It is also easy to check that $\fm^\dagger=(\log \fm)'$ and
$(1+\epsilon)^\dagger=[\log(1+\epsilon)]'$ for $\epsilon\in \LL^{\prec 1}$, from which it follows that $g^\dagger=(\log |g|)'$ for $g\in \LL^{\ne}$ and
$(f^t)^\dagger=tf^\dagger$ for $f\in \LL^{>}$. 

Below we consider $\mathbb{L}_{<\alpha}$ as a differential 
field with derivation $\der_{\alpha}$, and also as an ordered and valued field. For the rest of this section we assume familiarity with the basic facts on $H$-fields and their asymptotic couples from \cite{ADH1}. 

\begin{lemma}\label{l4}  $\mathbb{L}_{<\alpha}$ is an $H$-field with constant field $\R$.
\end{lemma}
\begin{proof} Note that if $\fm\succ 1$, then $\fm'>0$. Let
$f\in \mathbb{L}_{<\alpha}$. 

Suppose $f>0$ and $f\succ 1$. Then $\fd(f)\succ 1$, so $\fd(f)'>0$. Also
$\fd(f)'\succ \fm'$ for all $\fm\in \supp(f)\setminus \{\fd(f)\}$ by
Lemma~\ref{6things}(v), and thus $f'>0$.  

Next, assume $f\notin \R$; we claim that $f'\ne 0$. 
By subtracting a real number from $f$ we arrange $1\notin \supp(f)$. Then
the same item (v) yields $f'\ne 0$.
\end{proof}

\noindent
We make the additive group $\Gamma$ of exponent sequences into an ordered abelian group by $r < s :\Leftrightarrow \ell^r \prec \ell^s$. We define the valuation $v: \LL^\times \to \Gamma$
by $v(f)=-r$ if $\fd(f)=\ell^r$; thus $v(f) > v(g)\Leftrightarrow f\prec g$ for all $f,g\in \LL$. We have 
$$\Gamma_{<\alpha}\ :=\ v(\LL_{<\alpha}^{\times})\ =\ \{r\in \Gamma:\ r_{\beta}=0  \text{ for all }\beta\ge \alpha\}.$$

\medskip\noindent
\noindent
Note that if $\beta< \alpha$, then $\mathbb{L}_{<\beta}$ is an $H$-subfield of
$\mathbb{L}_{<\alpha}$. Next we consider the asymptotic couple 
$(\Gamma_{<\alpha}, \psi_{<\alpha})$ of $\mathbb{L}_{<\alpha}$. We have an order-preserving
bijection 
$$\beta \mapsto  v(\ell_{\beta}^\dagger)\ :\ \alpha \to \Psi_{<\alpha}$$
from $\alpha$ onto the $\Psi$-set $\Psi_{<\alpha}$ of $\mathbb{L}_{<\alpha}$. 
In particular, if $\alpha\ne 0$, then 
$\Psi_{<\alpha}$ has least element $v(\ell_0^{-1})$, and this element is
positive and is the unique fixed point of $\psi$. 

\begin{lemma}\label{l5} If $\alpha=\beta+1$, then 
$v(\ell_{\alpha}')=v(\ell_{\beta}^\dagger )=\max \Psi_{<\alpha}>0$. If $\alpha\ne 0$ is a limit ordinal, then $v(\ell_{\alpha}')>0$ is a gap in $\mathbb{L}_{<\alpha}$.
\end{lemma}
\begin{proof} The first claim follows from the above order-preserving bijection
$\alpha\to \Psi_{<\alpha}$. Suppose $\alpha\ne 0$ is a limit. Then in 
$\Gamma_{<\alpha}$,
$$\Psi_{<\alpha}\ <\  v(\ell_{\alpha}')\ <\  (\Gamma_{<\alpha}^{>})',$$
so  $v(\ell_{\alpha}')>0$ is indeed a gap in $\mathbb{L}_{<\alpha}$.
\end{proof}

\subsection*{Integration} {\em In this subsection $\fm$ and $\fn$ range over $\mathfrak{L}_{<\alpha}$}. We use the
modified derivation $\derdelta:= \frac{1}{\ell_{\alpha}'}\der_{\alpha}$
on $\mathbb{L}_{< \alpha}$, which is strongly $\R$-linear. It follows from Lemma~\ref{l5} that for the
$\Psi$-set $\Psi_{\derdelta}$ of the asymptotic couple of the $H$-field 
$(\mathbb{L}_{< \alpha}, \derdelta)$ we have $\max \Psi_{\derdelta} =0$ if $\alpha$ is
a successor ordinal, and $\sup \Psi_{\derdelta} =0\notin \Psi$, otherwise.
Thus
$\derdelta$ is small, but $\derdelta(\fm) \succeq \fm$ for $\fm\ne 1$. Moreover:

\begin{lemma} If $\fm\succ 1$, then $\supp \derdelta(\fm)\succ 1$. If $\fm\prec 1$, then $\supp \derdelta(\fm)\prec 1$. 
\end{lemma}
\begin{proof} Let $\fm=\ell^r$. Then 
$$\derdelta(\fm)\ =\ \big(\prod_{\beta<\alpha}\ell_{\beta}\big)\cdot \big(\sum_{\gamma\in \sigma(\fm)}r_{\gamma}\fm\ell_{\gamma}^\dagger\big)\ =\ 
\sum_{\gamma\in \sigma(\fm)}r_{\gamma}\big(\fm \prod_{\gamma< \beta<\alpha} \ell_{\beta}).$$
It remains to note that for $\gamma\in \sigma(\fm)$ and $\fn :=\fm\prod\limits_{\gamma<\beta<\alpha}\ell_\beta$ we have $\min\sigma(\fn) =  \min \sigma (\fm)$ and $r_{\min\sigma(\fn)} = r_{\min \sigma(\fm)}$. 
\end{proof}

\noindent
Thus $\derdelta$ maps $\R[[\mathfrak{L}_{<\alpha}^{\ne 1}]]$ into itself. 
For $\xi\ne 0$ in the asymptotic couple of $(\mathbb{L}_{< \alpha}, \derdelta)$ we set $\xi^\dagger:=\psi_{\derdelta}(\xi)\le 0$ and
$\xi':= \xi + \xi^\dagger$, so $\xi^\dagger=o(\xi)$ by \cite[Lemma 9.2.10(iv)]{ADH1}, hence $(\xi-\xi^\dagger)^\dagger=\xi^\dagger$,
and thus $(\xi-\xi^\dagger)'=\xi$. It follows that for any $\fm\ne 1$
there is a unique $\fn\ne 1$ with
$\derdelta(\fn)\asymp \fm$, namely $\fn=\fd(\frac{\fm}{\derdelta(\fm)/\fm})$,
and the map that assigns to any $\fm\ne 1$ the unique $\fn\ne 1$
with $\derdelta(\fn)\asymp \fm$ is an automorphism of the ordered set
$\mathfrak{L}_{<\alpha}^{\ne 1}$. We define
$T\ :\  \mathfrak{L}_{<\alpha}^{\ne 1}\ \to\ \R^\times\cdot \big( \mathfrak{L}_{<\alpha}^{\ne 1}\big)$ by
$$ T(\fm)\ :=\ \c\fn,\ \text{ with  $c\in \R^\times$ and $\fn\in\mathfrak{L}_{<\alpha}^{\ne 1}$ such that }\c\derdelta(\fn)\sim \fm.$$
Using Lemma~\ref{vdh} we note that $T$ extends uniquely to a strongly $\R$-linear bijection $\R[[\mathfrak{L}_{<\alpha}^{\ne 1}]]\to \R[[\mathfrak{L}_{<\alpha}^{\ne 1}]]$, also denoted by $T$, with a strongly $\R$-linear inverse $T^{-1}$. 
By virtue of the definition of $T$ we have for nonzero $g\in  \R[[\mathfrak{L}_{<\alpha}^{\ne}]]$:
$$\derdelta(Tg)\ =\ g+ E(g) \text{ with }E(g)\prec g.$$
This determines the strongly $\R$-linear selfmap $E:= \derdelta\circ T -I$ on 
$\R[[\mathfrak{L}_{<\alpha}^{\ne 1}]]$, where $I$ is the identity
on $\R[[\mathfrak{L}_{<\alpha}^{\ne}]]$. Since $E(g)\prec g$ for all 
nonzero $g\in  \R[[\mathfrak{L}_{<\alpha}^{\ne 1}]]$, it follows from Lemma~\ref{inverse} that
the strongly $\R$-linear selfmap $I+E$ on $\R[[\mathfrak{L}_{<\alpha}^{\ne 1}]]$
is bijective with strongly $\R$-linear inverse
$(I+E)^{-1}$. From $\derdelta\circ T= I + E$ we get
$\derdelta\circ T\circ(I+E)^{-1}=I$, that is, 
$\derdelta^{-1}:=T\circ(I+E)^{-1}$ is a strongly $\R$-linear 
right inverse to $\derdelta$ on  $\R[[\mathfrak{L}_{<\alpha}^{\ne 1}]]$.  
 In terms of the original derivation $\der$, this yields a distinguished 
strongly $\R$-linear bijective integration operator 
\[  \int\  \colon\  \R[[\mathfrak{L}_{<\alpha}\setminus\{\ell_{\alpha}'\}]]\ \to\ \R[[\mathfrak{L}_{<\alpha}^{\ne 1}]],\qquad
  \int f\  :=\ \derdelta^{-1}\big(f/\ell_{\alpha}'\big).\]
We call it an integration operator because $\der\big(\int f\big)=f$ for 
$f\in\R[[\mathfrak{L}_{<\alpha}\setminus\{\ell_{\alpha}'\}]]$. 

\subsection*{Integration, continued} The domain of the above integration operator depends on $\alpha$, but it assigns to each $f$ in its domain the unique $g\in \LL$ with $g'=f$ and $1\notin \supp g$. It follows that these operators for the various $\alpha$ have a common extension to an operator
$\int: \LL \to \LL$ that assigns to each $f\in \LL$ the unique $g\in \LL$ with $g'=f$ and $1\notin \supp g$. Thus we have now fully established
Theorem~\ref{ci} and the rest of the subsection ``The derivation on $\LL$'' in the Introduction. Note also that $\int$ maps $\LL_{<\alpha}$ into
$\LL_{<\alpha} + \R\ell_{\alpha}$, more precisely, bijectively onto
$\R[[\mathfrak{L}_{<\alpha}^{\ne 1}]] +\R\ell_{\alpha}$.  

The remainder of this section will not be used, but relates the above to
material in ~\cite{ADH1}. We assume now that $\alpha$ is an infinite limit ordinal, and set $$\LL^{\cup}_{<\alpha}\ :=\ \bigcup_{\beta<\alpha}\LL_{<\beta}.$$ 
We saw that $\LL_{<\alpha}$ is not closed under $\int$, but we now observe that 
its $H$-subfield $\LL^{\cup}_{<\alpha}$ is closed under $\int$ and is the union of its chain of spherically complete $H$-subfields $\LL_{<\beta}$
with $\beta<\alpha$, and if such $\beta$ is a successor ordinal, then
$\LL_{<\beta}$ is grounded. Thus by \cite[Corollary 11.7.15, Theorem 15.0.1]{ADH1}:

\begin{cor} The $H$-field $\LL^{\cup}_{<\alpha}$ is $\upo$-free and newtonian.
\end{cor}

\noindent
Of course the $H$-field $\LL$ is likewise $\upo$-free and newtonian. As to the case $\alpha=\omega$, we recall from \cite[Appendix A]{ADH1} that $\T$ has distinguished elements $\ell_n$. We have a unique field embedding
$\LL^{\cup}_{<\omega}\to \T$ that is the identity on $\R$, sends $\ell_n^r\in \LL^{\cup}_{<\omega}$ to $\ell_n^r\in \T$ for all $n$ and all $r\in \R$, and respects infinite sums. This embedding also respects the natural logarithm maps on
the multiplicative groups of positive elements of $\LL^{\cup}_{<\omega}$ and $\T$, and the natural derivations on these fields. The image of this embedding is the $H$-subfield 
$\T_{\log}$ of $\T$; we identify $\LL^{\cup}_{<\omega}$ with $\T_{\log}$ via this embedding.  

\section{Preliminaries on Composition}\label{comp}

\noindent In this section $\fN$ is a monomial group with real powers. We fix $\R[[\fN]]$ as an ambient Hahn field equipped with its natural ordering and valuation. Let $\fM$ be a power closed monomial subgroup of $\fN$ with a distinguished element $x\in \fM^{\succ 1}$. Then we have the (ordered valued) Hahn subfield $K := \R[[\fM]]$ of $\R[[\fN]]$.
We let $h$ range over the elements of $\R[[\fN]]^{>\R}$.  

\subsection*{Composition on Hahn Fields}
A \textbf{$K$-composition with $h$} is a map $$f \mapsto f\circ h\ :\ K \to \R[[\fN]]$$ 
such that the following conditions are satisfied: 
\begin{enumerate}
\item $1\circ h=1$ and $x\circ h=h$;
\item\label{i:CompProducts} for all $\fm_1, \fm_2 \in \fM$, $(\fm_1\fm_2) \circ h = (\fm_1 \circ h)\cdot(\fm_2 \circ h)$;
\item\label{i:CompSums} for all $f\in K$,  $(\fm \circ h)_{\fm \in \supp(f)}$ is summable and $\sum_\fm f_\fm (\fm \circ h) = f \circ h$.
\end{enumerate}

\begin{lemma}
\label{l:CompIsLinear}
Any $K$-composition with $h$ is an ordered field embedding $K \to \R[[\fN]]$
and is strongly $\R$-linear.
\end{lemma}
\begin{proof} Let a $K$-composition with $h$ be given.
Since $K$ and $\R[[\fN]]$ are real closed fields, the map $f\mapsto f\circ h: K \to \R[[\fN]]$ will be an ordered field embedding if it is a ring morphism. 
Let $f,g\in K$. Then 
\begin{align*}
(f+g)\circ h\ &=\ \left(\sum_\fm (f_\fm+g_\fm) \fm\right)\circ h\ =\  \sum_\fm (f_\fm+g_\fm)(\fm\circ h) \\
&=\ \sum_\fm (f_\fm (\fm\circ h) +g_\fm(\fm\circ h))\ =\ \sum_\fm f_\fm (\fm \circ h) + \sum_\fm g_\fm (\fm \circ h)\\
&=\ (f\circ h) + (g\circ h).
\end{align*}
Similarly, using Lemma~\ref{prfam}, 
\begin{align*}
(fg)\circ h\ &=\ \sum_\fm \left(\sum_{\fm_1\fm_2=\fm} f_{\fm_1}g_{\fm_2}\right)\fm\circ h\ =\ \sum_\fm \left(\sum_{\fm_1\fm_2=\fm} f_{\fm_1}(\fm_1\circ h)g_{\fm_2}(\fm_2\circ h)\right)\\ 
&=\ \left(\sum_{\fm_1}f_{\fm_1}(\fm_1\circ h)\right)\left( \sum_{\fm_2}g_{\fm_2}(\fm_2\circ h)\right)\ =\ (f\circ h) (g\circ h).
\end{align*}
Strong linearity follows from Lemma~\ref{vdh} and clause (3) above.
\end{proof}

\noindent
Here are some consequences of Lemma~\ref{l:CompIsLinear} for a $K$-composition $f\mapsto f\circ h$ with $h$:  $\fm \circ h > 0$ for
$\fm\in \fM$,  so $(\fm \circ h)^t$ is defined for all real $t$, and
for $f, f_1, f_2\in K$, 
$$f>\R\ \Leftrightarrow\ f \circ h > \R, \qquad f_1 \preceq f_2\ \Leftrightarrow\ f_1 \circ h \preceq f_2 \circ h.$$
Thus for $f\prec 1$ in $K$ we have $f\circ h\prec 1$ in $\R[[\fN]]$, and
$$\exp(f)\circ h=\exp(f\circ h), \qquad \big(\log(1+f)\big)\circ h= \log\big((1+f)\circ h\big).$$

\begin{lemma}\label{powercomp} Let a $K$-composition with $h$ be given such that $\fm^t \circ h = (\fm\circ h)^t$ for all $\fm \in \fM$ and $t\in \R$. Then $f^t \circ h = (f \circ h)^t$ for all $f \in K^{>}$ and $t\in \R$. 
\end{lemma}
\begin{proof} Let $f\in K^{>}$. Then $f = c\fm(1+\epsilon)$ where $c \in \R^{>}$, $\fm \in \fM$ and $\epsilon \in K^{\prec 1}$, so
\begin{align*} 
f^t \circ h\ &=\ (c\fm(1+\epsilon))^t \circ h\ =\ c^t (\fm^t \circ h)\left(\left(\sum_n \binom{t}{n}\epsilon^n\right)\circ h \right)\\
&=\ c^t (\fm\circ h)^t \sum_n \binom{t}{n}\left(\epsilon\circ h\right)^n\ =\ (f \circ h)^t \qquad (t\in \R). \qedhere
\end{align*} 
\end{proof}

\noindent
A \textbf{$K$-composition} is a map $\circ:K \times \R[[\fN]]^{>\R}\to\R[[\fN]]$ such that for all $h$ the map $f \mapsto f\circ h :K \to \R[[\fN]]$ is a $K$-composition with $h$.

\begin{lemma}\label{l:ReductionToMonomials}
Let $\circ$ be a $K$-composition, and let $g \in K^{>\R}$, $h\in \R[[\fN]]^{>\R}$ be such that $\fm\circ g\in K$ and $(\fm \circ g) \circ h= \fm \circ(g \circ h)$ for all $\fm \in \fM$. Then $f\circ g\in K$ and
$(f \circ g) \circ h = f \circ(g\circ h)$ for all $f \in K$.
\end{lemma}
\begin{proof} Let $f\in K$. It is clear that then
$f\circ g=\sum_{\fm}f_{\fm}(\fm\circ g)\in K$, and
\[
(f \circ g) \circ h\ =\ \sum_\fm f_\fm(\fm\circ g) \circ h\ =\ \sum_\fm f_\fm\fm\circ(g \circ h)\ =\ f\circ(g\circ h).
\qedhere\]
\end{proof}

\noindent 
We now consider the case $\fM=\fN$, and define
a \textbf{composition on $K$} for $K=\R[[\fN]]$ to be a map $\circ: K \times K^{>\R} \to  K$ such that:
\begin{enumerate}
\item $\circ$ is a $K$-composition;
\item $\fm^t \circ g  = (\fm \circ g)^t$ for all $\fm\in \fN$, $t\in \R$, and $g\in K^{>\R}$;
\item $(\fm \circ g)\circ h = \fm \circ (g \circ h)$ for all 
$g,h \in K^{>\R}$.
\end{enumerate}
Thus given a composition $\circ$ on $K$ it follows from Lemmas~\ref{powercomp} and ~\ref{l:ReductionToMonomials} that clauses (2) and (3) hold in a more general form: $f^t \circ g = (f \circ g)^t$ for all $f \in K^{>}$, $t\in \R$, and $g\in K^{>\R}$, and $(f \circ g) \circ h = f \circ(g\circ h)$ for all $f \in K$ and $g,h\in K^{>\R}$.

\subsection*{Taylor Expansion} Let there be given a 
$K$-composition $f\mapsto f\circ h: K \to \R[[\fN]]$ with $h$ and
an $\R$-linear derivation $\der$ on $K$ with well-based support $\supp \der\prec 1$ and $\der x=1$, and an element $\varepsilon\in \R[[\fN]]^{\prec 1}$. Then we set $g:= h+\varepsilon$ and `deform' the above $K$-composition with $h$ to a $K$-composition $f\mapsto f\circ g: K \to \R[[\fN]]$
with $g$ as follows:  with $\Phi$ the above $K$-composition with $h$ we
apply the subsection on Taylor deformations in Section~\ref{prelim} 
to obtain a strongly
$\R$-linear operator 
$$\circ_{g}\ :\  K\ \to\ \R[[\fN]],\qquad  f\mapsto \sum_{n=0}^\infty \frac{\der^n(f)\circ h}{n!}\varepsilon^n.$$
We think of $\circ_g$ as composition with $g$ on the right, which explains the notation. In this subsection we set $f\circ g:=\circ_g(f)$ for $f\in K$.
Note that $1\circ g=1$ and $x\circ g=g$. Then by Lemma~\ref{tay1}: 
   
\begin{lemma}\label{taexp1} The map $f\mapsto f\circ g: K\to \R[[\fN]]$ is a $K$-composition with $g$. 
\end{lemma}

\noindent
Assume in addition that $\der$ extends to a strongly $\k$-linear derivation on $\k[[\fN]]$, denoted also by $\der$. Then by Lemma~\ref{tay2} the `chain rule' is inherited:  

\begin{lemma}\label{taexp2} If $\der(f\circ h)=\big((\der f)\circ h\big)\cdot \der h$ for all $f\in \k[[\fM]]$, then $\der(f\circ g)=\big((\der f)\circ g\big)\cdot \der g$ for all $f\in \k[[\fM]]$.
\end{lemma}

\subsection*{Revisiting multipliability} Let $(f_i)_{i \in I}$ be a family in $\LL^>$, where $I$ is a set. We call $(f_i)$
 \emph{multipliable} if the family $(\log(f_i))$ is summable. Note that if $f_i=1+\epsilon_i$, with all
$\epsilon_i\in \R[[\mathfrak{L}_{<\alpha}]]^{\prec 1}$ for a fixed $\alpha$, then this agrees with $(1+\epsilon_i)$
being multipliable as defined
in Section~\ref{prelim}.  In general we have $\alpha$ such that $f_i=c_i \fm_i(1+\epsilon_i)$, $c_i\in \R^{>}$, $\fm_i\in \mathfrak{L}_{<\alpha}$, $\epsilon_i\in \R[[\mathfrak{L}_{<\alpha}]]^{\prec 1}$ for all $i$. Then 
$(f_i)$ is multipliable if and only if $(1+\epsilon_i)$ is multipliable, $c_i=1$ for all but finitely many $i$, and, with
$\fm_i=\prod_{\beta<\alpha}\ell_{\beta}^{r_{\beta i}}$,  there are for every $\beta<\alpha$ only finitely many $i$ with
$r_{\beta i}\ne 0$.

 If $(f_i)$ is multipliable, then so is $(f_i^{r_i})$ for any family $(r_i)$ of real numbers. Suppose the family $(f_i)$ in $\LL_{<\alpha}^{>}$ is multipliable. Then 
\begin{equation}\label{e:FormOfLog}
\sum_{i}\log(f_i)\ =\ \sum_{\beta<\alpha} s_\beta \ell_{\beta+1}+c+\epsilon
\end{equation} 
where the $s_\beta$ and $c$ are real numbers and $\epsilon\in \LL_{< \alpha}^{\prec 1}$. Thus we may define
\[
\prod_{i \in I}f_i\ :=\ \left(\prod_{\beta<\alpha}\ell_\beta^{s_\beta}\right)\ex^c\sum_{n=0}^\infty\frac{1}{n!}\epsilon^n\ \in \LL_{<\alpha}^{>}.
\]
Then $\log \prod_{i}f_i = \sum_i \log f_i$, so $\left(\prod_{i}f_i\right)^\dagger = \sum_{i}f_i^\dagger$ and $\prod_i f_i^t=\left(\prod_i f_i\right)^t$ for $t\in \R$. If also the family $(g_i)$ in $\LL_{<\alpha}^{>}$ is multipliable, then $(f_ig_i)$ is multipliable, and
$$\prod_i f_i\cdot \prod_i g_i\ =\ \prod_i f_ig_i.$$ Any family $(g_j)_{j=1,\dots,n}$ in $\LL^{>}$
is multipliable with $\prod_j g_j=g_{1}\cdots g_{n}$. Also, for any family
$(r_{\beta})_{\beta<\alpha}$ of real numbers the family $(\ell_{\beta}^{r_{\beta}})_{\beta<\alpha}$ is multipliable, and $\prod_{\beta< \alpha}\ell_{\beta}^{r_{\beta}}$ is the logarithmic hypermonomial that we expressed
this way earlier. Retracing the definitions gives:

\begin{lemma}\label{fdprod} Suppose the family $(f_i)$ in $\LL_{<\alpha}^{>}$ is multipliable. Then the family $(\fd(f_i))$ is multipliable as well and $\fd(\prod_i f_i)=\prod_i \fd(f_i)$.
\end{lemma}    

\noindent
We define the function $\log_n:\LL^{>\R} \to \LL^{>\R}$ by recursion on $n$: $$\log_0(g)\ :=\ g, \qquad \log_{n+1}(g)\ :=\ \log( \log_n(g)).$$ Thus $\log_n$ maps $\LL_{<\alpha}^{>\R}$ into itself if $\alpha$ is an infinite limit ordinal. For $g\in \LL^{>\R}$ and $\lambda:=\min \sigma(\fd g)$ we have $\fd(\log g)=\ell_{\lambda +1}$,
and an easy induction on $n$ gives
\[
\log_n(g)_\succ\ =\ \ell_{\lambda+n}\ \text{ for } n \geq 2,\ \qquad  \log_n(g)_\succeq\ =\ \ell_{\lambda+n}\ \text{ for } n \geq 3.
\]
Here is a useful lemma regarding the functions $\log_n$:

\begin{lemma}
\label{l:SumOfLogs}
Let $g \in \LL^{>\R}$. Then the family $(\log_n(g))_{n}$ is multipliable.
\end{lemma}
\begin{proof}
By the above remarks, we have for $n \geq 2$ that $\log_n(g) = \ell_{\lambda+n}+\epsilon_n$ where $\lambda=\lambda_g$ and $\epsilon_n \preceq 1$. Thus, for $\big(\log_n(g)\big)$ to be multipliable, it suffices that $(\epsilon_n)_{n \geq 2}$ is summable. For $n\geq 2$, we have
\begin{align*}
\log_{n+1}(g)\ &=\  \log(\ell_{\lambda+n} + \epsilon_n)\ =\ \log\left(\ell_{\lambda+n}\big(1+\frac{\epsilon_n}{\ell_{\lambda+n}}\big)\right)\\
 &=\ \ell_{\lambda +n+1} +\sum_{i =1}^\infty \frac{(-1)^{i-1}}{i}\left(\frac{\epsilon_n}{\ell_{\lambda+n}}\right)^i,\ \text{ so }\\
\epsilon_{n+1}\ &=\ \sum_{i =1}^\infty \frac{(-1)^{i-1}}{i}\left(\frac{\epsilon_n}{\ell_{\lambda+n}}\right)^i. 
\end{align*} Using this equality for $\epsilon_{n+1}$, a
straightforward induction on $n$ shows that every $\fm\in \supp \epsilon_n$ with $n\ge 2$ is of the form $$\fn_1\cdots\fn_{d_1}\ell_{\lambda+2}^{-d_2}\cdots \ell_{\lambda+n-1}^{-d_{n-1}}\ \in (\supp \epsilon_2)^{\infty}\cdot\mathfrak{S}_n,$$ where 
$\fn_1,\dots, \fn_{d_1}\in \supp \epsilon_2$, $d_1,\dots, d_{n-1}\in \N^{\ge 1}$,
$d_1\ge d_2\ge \cdots \ge d_{n-1}$, and 
$$ \mathfrak{S}_n\ :=\ \left\{\prod_{2\leq j < n}\ell_{\lambda+j}^{-d_j}:\  d_2,\dots, d_{n-1} \in \N^{\ge 1},\ d_2\ge d_3\ge \cdots \ge d_{n-1}\right\}
$$
The set $(\supp \epsilon_2)^\infty$ is well-based by Neumann's Lemma.
The (disjoint) union $\mathfrak{S} :=\ \bigcup_{n\ge 2}\mathfrak{S}_n$ is well-based by Lemma~\ref{well}. Thus the family $(\epsilon_n)_{n\geq2}$ is summable.
\end{proof}

\begin{lemma}\label{lemuc} Let $\circ$ be a composition on $\LL$ as defined in the Introduction and let $f,g\in \LL$, $f>0$, $g>\R$. Then 
$f^t\circ g=(f\circ g)^t$ for $t\in \R$. If in addition the family $(f_i)$
in $\LL^{>}$ is multipliable, then  the family $(f_i\circ g)$ is multipliable, and $$(\prod_i f_i)\circ g\ =\ \prod_i (f_i\circ g).$$
\end{lemma}
\begin{proof} By (CL1) and (CL3) we have 
$$\log(f^t\circ g)\ =\ (\log f^t)\circ g\ =\ (t\log f)\circ g\ =\ t\log (f\circ g)\ =\ \log [(f\circ g)^t],$$
so $f^t\circ g=(f\circ g)^t$. The second part follows likewise by taking logarithms.
\end{proof}

\section{Composing with Hyperlogarithms}\label{sec:comp1}
\noindent Recall that our goal is to construct a `good' composition operation on $\LL$. In this section we only compose on the right with  hyperlogarithms. 

We fix an ordinal $\alpha = \omega^\lambda$ where $\lambda$ is an infinite limit ordinal. Then $\xi+\eta < \alpha$ for all ordinals $\xi, \eta < \alpha$. We work in the Hahn field $\LL_{<\alpha}=\R[[\mathfrak{L}_{<\alpha}]]$ over $\R$. Let $\beta< \lambda$ be given and set $\mu:= \omega^{\beta +1}< \alpha$, so $\omega^\beta+\gamma<\mu$ for $\gamma<\mu$. 

\subsection*{Composing with $\ell_{\omega^\beta}$}
We shall use the modified derivation $\derdelta := \frac{1}{\ell_\mu'}\der_\alpha$ on $\LL_{<\alpha}$. Note that $\frac{1}{\ell_\mu'}=\prod_{\rho < \mu} \ell_{\rho}$. Hence for $\fm \in \mfL_{[\mu,\alpha)}$ we have
\[
\supp(\derdelta\fm)\ \subseteq\ \big\{\prod_{\mu \leq \rho \leq \gamma}\ell_\rho^{-1}:\ \mu\le \gamma<\alpha\big\}\cdot \fm.
\]
Thus the strongly $\R$-linear operator $\derdelta$ on $\LL_{<\alpha}$ maps
$\LL_{[\mu,  \alpha)}$ into itself. 

To explain and motivate the role of $\derdelta$ in defining $f\circ \ell_{\omega^\beta}$ for
$f\in \LL_{[\mu,\alpha)}$ we include the following remark; it is important for understanding what is going on, but is of a purely heuristic nature and can be skipped.

\medskip\noindent
{\bf Remark.} The composition $\circ$ on $\LL$ to be constructed will be such that the map $f\mapsto f\circ \ell_{\mu}: \LL_{[\mu,\alpha)}\to \LL_{[\mu,\alpha)}$ is bijective, which gives an inverse map 
$$f\mapsto f^{\uparrow_\mu}\ :\  \LL_{[\mu,\alpha)}\to \LL_{[\mu,\alpha)}, \qquad f^{\uparrow_\mu}\circ \ell_{\mu}=f\ 
\text{ for }\ f\in \LL_{[\mu,\alpha)}.$$
We also want $\circ$ to obey the Chain Rule and admit Taylor expansion, and
to satisfy $\ell_{\mu}\circ \ell_{\omega^\beta} =\ell_{\mu}-1$. Then for $f\in \LL_{[\mu,\alpha)}$ we have $(f^{\uparrow_\mu}\circ \ell_{\mu})'=\big((f^{\uparrow_\mu})'\circ \ell_{\mu}\big)\cdot \ell_{\mu}'=f'$, so $(f^{\uparrow_\mu})'\circ \ell_{\mu}=f'/\ell_{\mu}'=\derdelta(f)$, and thus by induction on $n$,
$$(f^{\uparrow_\mu})^{(n)}\circ \ell_{\mu}\ =\ \derdelta^n(f).$$
Using Taylor expansion this leads for such $f$ to
\begin{align*} f\circ \ell_{\omega^{\beta}}\ &=\ (f^{\uparrow_\mu}\circ \ell_{\mu})\circ \ell_{\omega^{\beta}}\ =\ f^{\uparrow_\mu}\circ (\ell_{\mu}\circ \ell_{\omega^{\beta}})\ =\ f^{\uparrow_\mu}\circ (\ell_{\mu}-1)\\ &=\ 
\sum_{n=0}^\infty \frac{(f^{\uparrow_\mu})^{(n)}\circ \ell_{\mu}}{n!}(-1)^n\ =\ \sum_{n=0}^\infty \frac{(-1)^n}{n!}\derdelta^n(f).
\end{align*}

After this remark we now resume the formal exposition. The restriction of $\derdelta$ to an operator
on $\LL_{[\mu, \alpha)}$ has
support contained in the well-based set 
$$\mathfrak{S}\ :=\ \big\{\prod_{\mu \leq \rho \leq \gamma}\ell_\rho^{-1}:\ \mu\le \gamma<\alpha\big\}\ \prec\ 1.$$ 
Thus by Lemma~\ref{supp2} and the remark following Lemma~\ref{supp3} the sum 
$\sum_{n=0}^\infty \frac{(-1)^n}{n!}\derdelta^n f$ exists in $\LL_{[\mu, \alpha)}$ for all $f\in \LL_{[\mu, \alpha)}$, and we have a strongly
$\R$-linear operator
\begin{equation}
\label{e:PartialTaylor} f\mapsto f \circ \ell_{\omega^\beta}\ :\ \LL_{[\mu,\alpha)} \to \LL_{[\mu,\alpha)}, \qquad
f \circ \ell_{\omega^\beta}\ :=\ \sum_{n=0}^\infty \frac{(-1)^n}{n!}\derdelta^n f,
\end{equation}
with $\supp(f\circ \ell_{\omega^\beta})\subseteq \mathfrak{S}^\infty\cdot \supp f$ for $f\in \LL_{[\mu, \alpha)}$. A routine computation gives
\begin{equation}
\label{e:HyperlogCom} 
(fg) \circ \ell_{\omega^\beta}\ =\ 
(f\circ \ell_{\omega^\beta})\cdot(g\circ \ell_{\omega^\beta}) \qquad (f,g\in \LL_{[\mu,  \alpha)}). 
\end{equation}

\begin{lemma}\label{car} We have $f\sim f\circ \ell_{\omega^\beta}$ for nonzero $f\in\LL_{[\mu,\alpha)}$, and the map $f\mapsto f\circ \ell_{\omega^\beta}$ is an automorphism of the field $\LL_{[\mu,\alpha)}$.
\end{lemma}
\begin{proof} Let $f$ range over $\LL_{[\mu, \alpha)}$ and set $\Phi(f)=\sum_{n=1}^\infty \frac{(-1)^n}{n!}\derdelta^n f$.
Then the map $\Phi: \LL_{[\mu,\alpha)} \to \LL_{[\mu,\alpha)}$ is $\R$-linear with well-based support 
$\supp \Phi\subseteq \bigcup_{n=1}^\infty \mathfrak{S}^n\prec 1$, and thus 
$\Phi(f)\prec f$ if $f\ne 0$. Now $f\circ \ell_{\omega^\beta}=f+\Phi(f)$, so
$f\sim f\circ\ell_{\omega^\beta}$ for $f\ne 0$, and $f\mapsto f\circ \ell_{\omega^\beta}: \LL_{[\mu,\alpha)}\to \LL_{[\mu,\alpha)}$ is bijective by Lemma~\ref{inverse}.  
\end{proof}  

\noindent
{\bf Remark.}  It is natural to denote the operator 
$$f\ \mapsto\  f\circ \ell_{\omega^\beta}\ =\ \left[ \sum_{n=0}^\infty \frac{(-1)^n}{n!}\derdelta^n\right](f)$$
on $\LL_{[\mu,\alpha)}$ by $\ex^{-\derdelta}$. More generally, any $s\in \R$ yields an operator 
$$\ex^{s\derdelta}\ :\ \LL_{[\mu,\alpha)}\to \LL_{[\mu,\alpha)}, \quad f\mapsto  \sum_{n=0}^\infty \frac{s^n}{n!}\derdelta^n f,$$
and $\ex^{s\derdelta}\circ \ex^{t\derdelta}=\ex^{(s+t)\derdelta}$ for $s,t\in \R$, so we have a group $\ex^{\R\derdelta}$ of such operators.

\medskip
Next we define for a monomial $\prod_{\gamma<\mu}\ell_\gamma^{r_\gamma}\in \mfL_{<\mu}$, 
\begin{equation}
\label{e:BigMonomials}
\left(\prod_{\gamma<\mu}\ell_\gamma^{r_\gamma}\right)\circ \ell_{\omega^\beta}\ :=\  \prod_{\gamma<\mu}\ell_{\omega^\beta  +\gamma}^{r_\gamma}\in \mfL_{<\mu}.
\end{equation}
Note that $\fm\mapsto \fm\circ \ell_{\omega^\beta}: \mfL_{<\mu}\to \mfL_{<\mu}$ is an embedding of ordered groups, and that this map is contractive: $\fm\circ\ell_{\omega^\beta}\succ \fm$ if $\fm\in \mfL_{<\mu}^{\prec 1}$, and $\fm\circ\ell_{\omega^\beta}\prec \fm$ if $\fm\in \mfL_{<\mu}^{\succ 1}$.

\medskip\noindent Finally, using $\LL_{<\alpha}=\LL_{[\mu,\alpha)}[[\mfL_{<\mu}]]$ and representing $f \in \LL_{<\alpha}$ as
\[
f\ =\ \sum_{\fm \in \mfL_{<\mu}}f_{[\fm]}\fm
\]
where all $f_{[\fm]}\in \LL_{[\mu,\alpha)}$ and $\{\fm \in \mfL_{<\mu}:f_\fm \neq 0\}$ is well-based, we note that $$\sum_{\fm \in \mfL_{<\mu}}\left(f_{[\fm]}\circ \ell_{\omega^\beta}\right)\left(\fm\circ \ell_{\omega^\beta}\right)$$
exists in the Hahn field $\LL_{[\mu,\alpha)}[[\mfL_{<\mu}]]$ over 
$\LL_{[\mu,\alpha)}$, since all $f_{[\fm]} \circ \ell_{\omega^\beta} \in \LL_{[\mu,\alpha)}$ and $\fm \circ \ell_{\omega^\beta} \prec \fn\circ \ell_{\omega^\beta} $ for all $\fm\prec\fn$ in $\mfL_{<\mu}$.  Thus we may define the operation
\begin{equation}
\label{e:HyperlogComp} f\mapsto f\circ \ell_{\omega^\beta}\ :\ \LL_{<\alpha}\to \LL_{<\alpha}, \qquad
f \circ \ell_{\omega^\beta}\ :=\ \sum_{\fm \in \mfL_{<\mu}}\left(f_{[\fm]}\circ \ell_{\omega^\beta}\right)\left(\fm\circ \ell_{\omega^\beta}\right).
\end{equation}
We do not create here a conflict of notation: if $f\in \LL_{[\mu, \alpha)}$ or $f\in \mathfrak{L}_{<\mu}$, then this agrees with the previously defined
$f\circ \ell_{\omega^\beta}$. In particular, $\ell_0\circ \ell_{\omega^\beta}=\ell_{\omega^\beta}$.

\begin{lemma}
\label{l:PartialSpecialComp}
The map $f \mapsto f \circ \ell_{\omega^\beta}:\LL_{<\alpha} \to \LL_{<\alpha}$ is an $\LL_{<\alpha}$-composition with $\ell_{\omega^\beta}$,
where we take $x:=\ell_0$ as the distinguished element of $\mfL_{<\alpha}^{\succ 1}$. 
\end{lemma}
\begin{proof}
To see that clause (2) in the definition of ``$K$-composition with $h$'' holds, let $\fm, \fn\in \mathfrak{L}_{<\alpha}$. Then $\fm = \fm_{< \mu}\fm_{\geq \mu}$, $\fn = \fn_{< \mu}\fn_{\geq \mu}$, $\fm_{< \mu}, \fn_{<\mu} \in \mfL_{<\mu},\ \fm_{\geq \mu}, \fn_{\ge \mu} \in \mfL_{[\mu, \alpha)}$. Using (\ref{e:BigMonomials}), $(\fm_{<\mu}\fn_{<\mu})\circ \ell_{\omega^\beta} = (\fm_{<\mu}\circ \ell_{\omega^\beta})(\fn_{<\mu}\circ \ell_{\omega^\beta})$, and by (\ref{e:HyperlogCom}), \begin{align*}
(\fm_{\geq \mu}\fn_{\geq \mu}) \circ \ell_{\omega^\beta}\ =\  (\fm_{\geq \mu}\circ \ell_{\omega^\beta})(\fn_{\geq \mu} \circ \ell_{\omega^\beta}).
\end{align*} 
Now using also (\ref{e:HyperlogComp}), we have
\[
(\fm\fn) \circ \ell_{\omega^\beta}\ =\ \big((\fm_{<\mu}\fn_{<\mu})\circ \ell_{\omega^\beta}\big)\cdot \big((\fm_{\geq \mu}\fn_{\geq \mu})\circ \ell_{\omega^\beta}\big)\ =\  (\fm\circ \ell_{\omega^\beta})(\fn\circ \ell_{\omega^\beta}).
\]
That clause (3) is satisfied follows easily from the strong linearity of the map in (\ref{e:PartialTaylor}) and the existence of the sums in (\ref{e:HyperlogComp}). 
\end{proof}

\begin{cor}\label{corineq} Let $0\ne f\in \LL_{<\alpha}^{\preceq 1}$. Then 
$f\circ \ell_{\omega^{\beta}}\prec \ell_{\mu}^{-n}$ for all $n$, or $f\circ \ell_{\omega^{\beta}}\sim f$.
\end{cor}
\begin{proof} With $f= \sum_{\fm \in \mfL_{<\mu}}f_\fm\fm$ as above,
set $\fd:=\max\{\fm\in \mfL_{<\mu}:\ f_{\fm}\ne 0\}$. Then
either $\fd\prec 1$ and $f\circ \ell_{\omega^{\beta}}\prec \ell_{\mu}^{-n}$ for all $n$, or $\fd=1$ and $f\circ \ell_{\omega^{\beta}}\sim f$ by Lemma~\ref{car}.
\end{proof}

\begin{lemma}\label{lemsp} Let $\nu\le \mu$. Then $\LL_{[\nu,\alpha)}\circ \ell_{\omega^\beta}= \LL_{[\omega^\beta+\nu,\alpha)}$.
\end{lemma}
\begin{proof} For $\nu=\mu$ this follows from Lemma~\ref{car} and
$\omega^\beta+\mu=\mu$. Let $\nu < \mu$. Then $\LL_{[\nu,\alpha)}=\LL_{[\mu,\alpha)}[[\mathfrak{L}_{[\nu, \mu)}]]$. Accordingly, for $f\in \LL_{[\nu,\alpha)}$ we have $f=\sum_{\fm} f_{[\fm]}\fm$ with $\fm$ ranging over $\mathfrak{L}_{[\nu,\mu)}$ and all $f_{[\fm]}\in \LL_{[\mu,\alpha)}$.
Then $$f\circ \ell_{\omega^\beta}\ =\ \sum_{\fm} (f_{[\fm]}\circ \ell_{\omega^\beta})(\fm\circ \ell_{\omega^\beta}).$$
Now $\LL_{[\mu,\alpha)}\circ \ell_{\omega^\beta}=\LL_{[\mu,\alpha)}$, and
$\mathfrak{L}_{[\nu,\mu)}\circ \ell_{\omega^\beta}= \mathfrak{L}_{[\omega^\beta+\nu,\omega^\beta+\mu)}=\mathfrak{L}_{[\omega^\beta+\nu,\mu)}$ 
by (\ref{e:BigMonomials}). It remains to note that $\LL_{[\omega^\beta+\nu,\alpha)}=\LL_{[\mu,\alpha)}[[\mathfrak{L}_{[\omega^\beta+\nu, \mu)}]]$.  
\end{proof} 

\noindent
For $\nu=0$ this gives $\LL_{<\alpha}\circ \ell_{\omega^\beta}=\LL_{[\omega^\beta,\alpha)}$.

\begin{lemma}\label{l:PartialChain}
$
(f \circ \ell_{\omega^\beta})' =  (f'\circ \ell_{\omega^\beta})\cdot \ell_{\omega^\beta}'\ $ for $f \in \LL_{<\alpha}$.
\end{lemma}
\begin{proof}
For a monomial $\fm = \prod_{\gamma<\mu}\ell_\gamma^{r_\gamma}\in \mfL_{<\mu}$, we have by (\ref{e:BigMonomials}):
\begin{align*}
(\fm \circ \ell_{\omega^\beta})'\ &=\ \left(\prod_{\gamma<\mu}\ell_{\omega^\beta  +\gamma}^{r_\gamma}\right)'\ =\ (\fm \circ \ell_{\omega^\beta})\sum_{\gamma < \mu} r_\gamma \ell_{\omega^\beta+\gamma}^\dagger\\
&=\ (\fm \circ \ell_{\omega^\beta}) \sum_{\gamma<\mu}r_\gamma \left(\prod_{\rho\leq \omega^\beta+\gamma}\ell_\rho^{-1}\right)\\
&=\ (\fm \circ \ell_{\omega^\beta}) \sum_{\gamma<\mu}r_\gamma \left(\prod_{\rho<\omega^\beta}\ell_\rho^{-1}\right)\left(\prod_{\rho\leq \gamma}\ell_{\omega^\beta+\rho}^{-1}\right)\\
&=\ \left(\prod_{\rho<\omega^\beta}\ell_\rho^{-1}\right)(\fm \circ \ell_{\omega^\beta}) \sum_{\gamma<\mu}r_\gamma \left(\prod_{\rho\leq \gamma}\ell_\rho^{-1}\right)\circ \ell_{\omega^\beta}\\
&=\ \left(\prod_{\rho<\omega^\beta}\ell_\rho^{-1}\right) \sum_{\gamma<\mu}r_\gamma (\fm \circ \ell_{\omega^\beta})(\ell_\gamma^\dagger\circ \ell_{\omega^\beta})\ 
=\  \ell_{\omega^\beta}'\left(\fm' \circ \ell_{\omega^\beta}\right).
\end{align*}
For $g \in \LL_{\geq \mu,<\alpha}$ we have  $g \circ \ell_{\omega^\beta} = \sum_{n=0}^\infty \frac{(-1)^n}{n!}(\derdelta^n g)$ by (\ref{e:PartialTaylor}), so
\begin{align*}
(g \circ \ell_{\omega^\beta})'\  &=\  \sum_{n=0}^\infty \frac{(-1)^n}{n!}(\derdelta^n g)'\ =\ \ell_{\mu}'\sum_{n=0}^\infty \frac{(-1)^n}{n!}(\derdelta^{n+1} g)\\
&=\  \ell_{\mu}' \sum_{n=0}^\infty \frac{(-1)^n}{n!}\left(\derdelta^{n} \left(\frac{g'}{\ell_\mu'}\right)\right)\ =\  \ell_{\mu}' \left(\left(\frac{g'}{\ell_\mu'}\right)\circ \ell_{\omega^\beta}\right)\\
&=\ \left(\prod_{\rho<\mu}\ell_\rho^{-1} \right) \left(\prod_{\rho<\mu}\ell_{\omega^\beta+\rho} \right)(g'\circ\ell_{\omega^\beta}) = \left(\prod_{\rho<\omega^\beta}\ell_\rho^{-1}\right)(g'\circ \ell_{\omega^\beta})\\
 &=\ \ell_{\omega^\beta}'\cdot (g'\circ \ell_{\omega^\beta}).
\end{align*}
Finally, for $f \in \LL_{<\alpha}$ we have $f = \sum_{\fm \in \mfL_{<\mu}}f_\fm\fm$ where all $f_\fm\in \LL_{\geq \mu,<\alpha}$, so
\begin{align*}
(f\circ \ell_{\omega^\beta})'\ &=\  \sum_{\fm \in \mfL_{<\mu}}\big((f_\fm\circ\ell_{\omega^\beta})( \fm\circ \ell_{\omega^\beta})\big)'\\
&=\ \sum_{\fm \in \mfL_{<\mu}}(f_\fm\circ\ell_{\omega^\beta})'( \fm\circ \ell_{\omega^\beta})+(f_\fm\circ\ell_{\omega^\beta})( \fm\circ \ell_{\omega^\beta})'\\
&=\  \ell_{\omega^\beta}'\cdot \sum_{\fm \in \mfL_{<\mu}} (f_\fm'\circ\ell_{\omega^\beta})( \fm\circ \ell_{\omega^\beta})+(f_\fm\circ\ell_{\omega^\beta})( \fm'\circ \ell_{\omega^\beta})\\
&=\  \ell_{\omega^\beta}'\cdot\sum_{\fm \in \mfL_{<\mu}}\big((f_\fm \fm)'\circ\ell_{\omega^\beta}\big)\ =\ \ell_{\omega^\beta}'\cdot (f' \circ \ell_{\omega^\beta}).\qedhere
\end{align*}
\end{proof}

\noindent
Note that by (\ref{e:BigMonomials}) we have $\ell_{\gamma}\circ \ell_{\omega^\beta}=\ell_{\omega^\beta +\gamma}$ for $\gamma<\mu$. The next lemma gives more information
about $\ell_{\gamma}\circ \ell_{\omega^\beta}$ for $\gamma\ge \mu$. 

\begin{lemma}\label{comgam} We have $\ell_{\mu}\circ \ell_{\omega^{\beta}}=\ell_{\mu}-1$. If
$\mu <\gamma<\alpha$, then 
$$\qquad\qquad\qquad\quad \ell_{\gamma}\circ \ell_{\omega^\beta}\ =\ \ell_{\gamma}-\epsilon_{\gamma}, \quad \text{ with } 0 < \epsilon_{\gamma}\preceq \ell_{\mu}^{-1}\prec \ell_{\gamma}^{-1}\prec 1.$$
\end{lemma}
\begin{proof} Let $\mu\le \gamma <\alpha$. Then $\ell_{\gamma}\in \mathfrak{L}_{[\mu,\alpha)}$ and $\derdelta \ell_{\gamma}=
\prod_{\mu\leq \rho < \gamma}\ell_\rho^{-1}$, so
$\derdelta \ell_{\mu}=1$ and 
$\derdelta \ell_\gamma\preceq \ell_{\mu}^{-1}\prec 1$ if $\gamma>\mu$.
The derivation $\derdelta$ on $\LL_{[\mu,\alpha)}$ has support 
$\preceq \ell_{\mu}^{-1}$, so 
$\derdelta^n\ell_{\gamma}\preceq \ell_{\mu}^{-2}$ if $n\ge 2$ and $\gamma>\mu$.
Therefore $\ell_{\gamma}\circ \ell_{\omega^\beta}=\ell_{\gamma}-\derdelta\ell_{\gamma} +\frac{1}{2}\derdelta^2\ell_{\gamma} -\cdots$
is as described in the lemma.
\end{proof}

\subsection*{Composing with arbitrary hyperlogarithms} {\em In this subsection we assume that $\gamma<\alpha$}. We have
\begin{equation}\label{e:LongCantor}
\gamma\ =\ \omega^{\beta_1}+\omega^{\beta_2}+\cdots+ \omega^{\beta_k}\
  \qquad(k\in \N)
\end{equation}
where $\lambda>\beta_1\geq\beta_2\geq\ldots\geq\beta_k$; this is essentially the Cantor normal form of $\gamma$, but we allow the exponents to be repeated and require all coefficients to be 1. For $f \in \LL_{<\alpha}$, we set
\begin{equation}\label{circLongCantor}
f \circ \ell_\gamma\ :=\ \left(\left(\cdots \left(\left(f\circ \ell_{\omega^{\beta_k}}\right)\circ\ell_{\omega^{\beta_{k-1}}}\right)\circ\cdots\right)\circ \ell_{\omega^{\beta_2}}\right)\circ\ell_{\omega^{\beta_1}}.
\end{equation}
For $\gamma=0$ (so $k=0$), this means $f\circ \ell_0:= f$, by convention.
\noindent
Several of the results below are proved by induction on the length $k$ of the representation in (\ref{e:LongCantor}), using also $\ell_\nu\circ \ell_{\omega^\beta}=\ell_{\omega^\beta+\nu}$ for $\nu <\omega^{\beta+1}$, 
which holds by definition according to (\ref{e:BigMonomials}). For example,
recalling that $x:=\ell_0$, such an induction easily gives $x\circ \ell_{\gamma}=\ell_{\gamma}$.
As a consequence of Lemma~\ref{l:PartialSpecialComp} we obtain in this way:

\begin{cor}
\label{c:HyperlogComp}
The map $f \mapsto f \circ \ell_{\gamma}:\LL_{<\alpha} \to \LL_{<\alpha}$ is an $\LL_{<\alpha}$-composition with $\ell_{\gamma}$.
\end{cor}

\noindent
We use Lemma~\ref{l:PartialChain} likewise to obtain:

\begin{cor}\label{c:HyperlogChain}
$
(f \circ \ell_\gamma)' =  (f'\circ \ell_\gamma)\cdot \ell_\gamma'\ $ for $f \in \LL_{<\alpha}$.
\end{cor}

\begin{cor}\label{circadd} Suppose $\gamma\ne 0$ and $\nu< \omega^{\beta_k +1}$. Then 
\[\ell_{\nu}\circ \ell_{\gamma}\ =\ \ell_{\gamma+\nu}, \qquad  
(f\circ \ell_{\nu})\circ \ell_{\gamma}\ =\ f\circ \ell_{\gamma+\nu}\ \text{ for }\ f\in \LL_{<\alpha}.\]
\end{cor} 

\noindent
For later use we also record the following variants:

\begin{lemma}\label{upmon} Let $\gamma\ne 0$ and $\nu\le \omega^{\beta_k+1}$. Then for $\fm=\prod_{\rho<\nu} \ell_{\rho}^{r_{\rho}}\in \mathfrak{L}_{<\nu}$ we have
$\fm\circ \ell_{\gamma}=\prod_{\rho<\nu} \ell_{\gamma+\rho}^{r_{\rho}}\in \mathfrak{L}_{[\gamma, \gamma+\nu)}$, and so the map
$$\fm \mapsto \fm\circ \ell_{\gamma}\ :\  \mathfrak{L}_{<\nu}\to \mathfrak{L}_{[\gamma, \gamma+\nu)}$$ is an isomorphism of ordered groups. In particular, we have an isomorphism $$\fm\mapsto \fm\circ \ell_{\gamma}\ :\ \mathfrak{L}_{<\omega} \to \mathfrak{L}_{[\gamma, \gamma+\omega)}$$ of ordered groups.  
\end{lemma}

\begin{lemma}\label{uptrans} Let $\gamma\ne 0$, $\nu\le \omega^{\beta_k+1}$. Then $\LL_{[\nu,\alpha)}\circ \ell_{\gamma}= \LL_{[\gamma+\nu, \alpha)}$. 
\end{lemma}
\begin{proof} By
induction on $k\ge 1$. The case $k=1$ is Lemma~\ref{lemsp}. The inductive step from $k-1$ to $k$ uses that for $k>1$ we have 
$\omega^{\beta_k}+\nu\le \omega^{\beta_{k-1}+1}$.
\end{proof} 

\noindent
The proof of associativity in Section~\ref{sec:pc} will depend on the next two lemmas. In the first one we assume 
$\beta < \lambda$ and set $\mu:=\omega^{\beta+1}$.

\begin{lemma} Let $n\ge 1$. Then $\ell_{\gamma}\circ \ell_{\omega^\beta n}$ takes the following values:
$$\ell_{\omega^\beta n + \gamma} \text{ for }\gamma<\mu, \quad \ell_{\mu}-n \text{ for }\gamma=\mu, \quad \ell_{\gamma}-\epsilon \text{ with }0<\epsilon\preceq \ell_{\mu}^{-1} \text{ for } \gamma>\mu.$$
\end{lemma}
\begin{proof} For $n=1$ this is Lemma~\ref{comgam}. Assuming inductively that
the lemma holds for a certain $n$ we use 
$$\ell_{\gamma}\circ \ell_{\omega^{\beta}(n+1)}\ =\ (\ell_{\gamma}\circ \ell_{\omega^{\beta}n})\circ \ell_{\omega^\beta}$$
and Corollary~\ref{corineq} to show it holds for $n+1$ instead of $n$.
\end{proof}

\noindent
Let $g\in \LL^{>\R}$. Set $\lambda_g:= \min \sigma(\fd g)$ (an ordinal) 
and call it the {\em logarithmicity of $g$}. Thus $\lambda_{\ell_{\nu}}=\nu$ for any ordinal $\nu$. Consider the Cantor normal form of $\lambda_g$:
$$\lambda_g\ =\ \omega^{\beta_1}n_1 + \cdots + \omega^{\beta_k}n_k\qquad (k\in \N,\ \beta_1> \cdots > \beta_k,\ n_1,\dots, n_k\in \N^{\ge 1}).$$
For any ordinal $\nu$ we set $$\lambda_{g;\nu}\ :=\ n_i\ 
\text{ if }\ \nu=\omega^{\beta_i +1}, \qquad \lambda_{g;\nu}\ :=\ 0\ \text{ if }\ \nu\notin \{\omega^{\beta_1+1},\dots, \omega^{\beta_k+1}\}.$$

\begin{lemma}\label{bg}  Let $\nu<\alpha$. Then $\ell_{\nu}\circ \ell_{\gamma}=\ell_{\gamma+\nu}-\lambda_{\ell_\gamma;\nu}-\epsilon$ with $0\le \epsilon\prec 1$.
\end{lemma} 
\begin{proof} This is clear for $\gamma=0$. Assume $\gamma>0$ has Cantor normal form $$\gamma\ =\ \omega^{\beta_1}n_1+\cdots + \omega^{\beta_k}n_k \qquad(\beta_1>\cdots > \beta_k,\ k, n_1,\dots, n_k\ge 1).$$ 
We first show by induction on $k$: \begin{enumerate}
\item if $\nu=\omega^{\beta_1+1}$, then $\ell_{\nu}\circ \ell_{\gamma}=\ell_{\nu}-n_1-\epsilon$ with $0\le \epsilon\prec 1$;
\item if $\nu>\omega^{\beta_1+1}$, then $\ell_{\nu}\circ \ell_{\gamma}=\ell_{\nu}-\epsilon$ with $0\le \epsilon\prec 1$.
\end{enumerate}
The previous lemma gives this for $k=1$. Assume it holds for a certain $\gamma$ as above. Then with $\beta_k > \beta_{k+1}$ and $n_{k+1}\ge 1$ the definitions easily yield
$$\ell_{\nu}\circ \ell_{\gamma+\omega^{\beta_{k+1}}n_{k+1}}\ =\ 
(\ell_\nu\circ \ell_{\omega^{\beta_{k+1}}n_{k+1}})\circ \ell_{\gamma}.$$
Let $\nu\ge \omega^{\beta_1+1}$. Then $\ell_{\nu}\circ \ell_{\omega^{\beta_{k+1}}n_{k+1}}=\ell_{\nu}-\epsilon$ with $0<\epsilon\prec 1$ by the previous lemma. If $\nu=\omega^{\beta_1+1}$, we get
$\ell_{\nu}\circ \ell_{\gamma+\omega^{\beta_{k+1}}n_{k+1}}=
(\ell_{\nu}-\epsilon)\circ \ell_{\gamma}=\ell_{\nu}-n_1-\epsilon^*$ with
$0<\epsilon^*\prec 1$ by
(1). If $\nu> \omega^{\beta_1+1}$, then we get 
$\ell_{\nu}\circ \ell_{\gamma+\omega^{\beta_{k+1}}n_{k+1}}=
(\ell_{\nu}-\epsilon)\circ \ell_{\gamma}=\ell_{\nu}-\epsilon^*$ with $0<\epsilon^*\prec 1$ by (2). This concludes the proof of (1) and (2) and shows that the lemma holds for $\nu\ge \omega^{\beta_1+1}$.

Now assume that $\nu<\omega^{\beta_1+1}$. We first consider the subcase that $\nu=\omega^{\beta_i+1}$ where $1< i\le k$. Then $\gamma=\gamma_1+\gamma_2$ with $\gamma_1=\omega^{\beta_1}n_1+\cdots + \omega^{\beta_{i-1}}n_{i-1}$
and $\gamma_2=\omega^{\beta_i}n_i+\cdots + \omega^{\beta_k}n_k$, hence
$\ell_{\nu}\circ \ell_{\gamma}=(\ell_{\nu}\circ \ell_{\gamma_2})\circ \ell_{\gamma_1}$. By (1) above with $\gamma_2$ in the role of $\gamma$ we
have $\ell_{\nu}\circ \ell_{\gamma_2}=\ell_{\nu}-n_i-\epsilon$ with
$0\le \epsilon\prec 1$, so
\[\ell_{\nu}\circ \ell_{\gamma}\ =\ (\ell_{\nu}-n_i-\epsilon)\circ \ell_{\gamma_1}\ =\ \ell_{\nu}\circ\ell_{\gamma_1}-n_i-\epsilon^*,\qquad 0\le\epsilon^*\prec 1,\] 
and $\ell_{\nu}\circ \ell_{\gamma_1}=\ell_{\gamma_1+\nu}$ by Corollary~\ref{circadd}. Now $\gamma_2+\nu=\nu$, so $\gamma+\nu=\gamma_1+\nu$, and thus  $\ell_{\nu}\circ \ell_{\gamma}= \ell_{\gamma+\nu}-n_i-\epsilon^*=\ell_{\gamma+\nu}-\lambda_{\ell_{\gamma};\nu}-\epsilon^*$, so the lemma holds in this case.
Next assume we are in the subcase $\omega^{\beta_{i-1}+1}> \nu>\omega^{\beta_i+1}$ where $1<i\le k$. With $\gamma=\gamma_1+\gamma_2$ as before we argue as in the previous subcase, using (2) with $\gamma_2$
in the role of $\gamma$, and obtain that the lemma holds in this case as well. The remaining subcase $\nu<\omega^{\beta_k+1}$ is taken care of by Corollary~\ref{circadd}. 
\end{proof}

\section{Composition with Arbitrary Elements}\label{sec:comp2}

\noindent
As in the previous section, $\alpha=\omega^\lambda$, where $\lambda$ is an infinite limit ordinal.
We now fix $g \in \LL_{<\alpha}^{>\R}$. We shall define $\fm \circ g$ for $\fm \in \mfL_{<\omega}$ and $f \circ g$ for $f \in \LL_{[\omega,\alpha)}$ and then use this to define the map $f\mapsto f \circ g: \LL_{<\alpha} \to \LL_{<\alpha}$.

\medskip\noindent 
Note that for $\fm = \prod_{n} \ell_n^{r_n}\in \mfL_{<\omega}$ we have $\log_n(g)^{r_n}\in \LL_{<\alpha}$ for all $n$ and that  the family $\big(\log_n(g)^{r_n}\big)$ is multipliable by Lemma \ref{l:SumOfLogs}.
Therefore  we may define
\begin{equation}\label{e:ProductOfLogs}
\fm \circ g\ :=\ \prod_{n}\log_n(g)^{r_n}\ \in\ \LL_{<\alpha}.
\end{equation}
Thus $1\circ g= 1$ and $\ell_n\circ g=\log_n(g)$. Also for $\fm, \fn\in \mfL_{<\omega}$, $t\in \R$,
$$(\fm\fn)\circ g\ =\ (\fm\circ g)(\fn\circ g),\quad 
\fm^t\circ g\ =\ (\fm\circ g)^t.$$

\begin{lemma}\label{l:SmallCompForm}
There exists a well-based set $\fS=\fS(g)\subseteq \mathfrak{L}_{<\alpha}$
such that for all $\fm= \prod_{n} \ell_n^{r_n}\in\mfL_{<\omega}$ we have 
\[
\supp(\fm \circ g)\ \subseteq\  \fd(g)^{r_0}\left(\prod_{n\ge 1}\ell_{\lambda_g+n}^{r_n}\right)\cdot\fS.
\]
\end{lemma}

\begin{proof}
Set $\fS := \left(\bigcup_{n\geq 1}(\supp \log_n(g))^{\prec 1}\right)^\infty$. By Lemma \ref{l:SumOfLogs}, $\bigcup_{n\geq 1}(\supp \log_n(g))^{\prec 1}$ is well-based, and so is $\fS$ by Neumann's Lemma. By (\ref{e:FormOfLog}), we have 
\[
\sum_{n=0}^\infty r_n\log_{n+1}(g)\ =\ \sum_{\beta<\alpha} s_\beta \ell_{\beta+1}+c+\epsilon
\]
where the $s_\beta$ and $c$ are reals and $\epsilon\prec 1$, so $\supp \epsilon \subseteq \bigcup_{n\geq 1}(\supp \log_n(g))^{\prec 1}$. Now
\[
\fm \circ g\ =\ \fd(\fm\circ g)\cdot e^c \cdot \sum_{m=0}^\infty \frac{1}{m!}\epsilon^m,
\]
and so $\supp(\fm \circ g) \subseteq \fd(\fm\circ g)\cdot\fS$.  It remains to note that by Lemma~\ref{fdprod}, 
\[
\fd\left(\fm \circ g\right)\ =\ \fd(g)^{r_0}\left(\prod_{n\ge 1}\ell_{\lambda_g+n}^{r_n}\right).\qedhere
\]
\end{proof}

\medskip\noindent Let $f \in \LL_{[\omega,\alpha)}$; we shall define $f \circ g$ by introducing $f^{\uparrow _3}\in \LL_{[\omega,\alpha)}$, to be thought of as $f\circ(\exp \circ \exp\circ \exp)$, and then setting $f\circ g := f^{\uparrow_3} \circ \log_3(g)$, exploiting that $\log_3(g)_{\succeq}$ is a hyperlogarithm.  Lemma~\ref{car} for $\beta=0$ gives $f \circ \ell_1\sim f$ if $f\ne 0$, so
\[
f\circ \ell_3\ =\ ((f\circ \ell_1)\circ \ell_1)\circ \ell_1\ =\ f+R(f),
\]
where $R(f) \prec f$ for $f\ne 0$ and $R(0)=0$. The map $f \mapsto f \circ \ell_1$ is a strongly $\R$-linear field automorphism of $\LL_{[\omega,\alpha)}$ by Lemma~\ref{car}, and so is
$f \mapsto f \circ \ell_3$, and the latter has
inverse $f \mapsto f^{\uparrow_3}:=\sum_{n=0}^\infty (-1)^nR^n(f)$ by Lemma~\ref{inverse}. Thus
\begin{equation}\label{compup3} f^{\uparrow_3}\circ \ell_3\ =\ (f\circ \ell_3)^{\uparrow_3}\ =\ f .\end{equation}
Now $\log_3(g) = \ell_{\lambda_g+3} + \varepsilon$ where $\varepsilon  \preceq \ell_{\lambda_g+2}^{-1}\prec 1$. By Corollary~\ref{c:HyperlogComp} we have the $\LL_{<\alpha}$-composition 
$\phi\mapsto \phi\circ\ell_{\lambda_g+3}: \LL_{<\alpha}\to \LL_{<\alpha}$
with $\ell_{\lambda_g+3}$. Then Lemma~\ref{taexp1} yields a deformation of it to
an $\LL_{<\alpha}$-composition $T_g: \LL_{<\alpha}\to \LL_{<\alpha}$ with $\log_3(g)$  by
\begin{equation}\label{e:TaylorO}
T_g(\phi)\ :=\  \sum_{n=0}^\infty \frac{\phi^{(n)}\circ \ell_{\lambda_g+3}}{n!}\varepsilon^n.
\end{equation}
Thus $T_g(\ell_0)=\log_3(g)$, and Lemma~\ref{taexp2} gives for $\phi\in \LL_{<\alpha}$,
\begin{equation}\label{e:chaintg} T_g(\phi)\ =\ T_g(\phi')\log_3(g)'.
\end{equation} Composing $T_g$ with 
$f\mapsto f^{\uparrow_3}$ yields the strongly $\R$-linear embedding 
\begin{equation}\label{e:TaylorOverOmega}
f\ \mapsto\  f \circ g\ :=\  \sum_{n=0}^\infty \frac{(f^{\uparrow_3})^{(n)}\circ \ell_{\lambda_g+3}}{n!}\varepsilon^n\ \colon\ \LL_{[\omega,\alpha)}\to \LL_{<\alpha}
\end{equation} 
of ordered and valued fields.
Towards extending this to $\LL_{<\alpha}$, let $f\in \LL_{<\alpha}$. Using $\LL_{<\alpha}=\LL_{[\omega,\alpha)}[[\mfL_{<\omega}]]$ we have $f= \sum_{\fm \in \mfL_{<\omega}}f_{[\fm]}\fm$
where all $f_{[\fm]}\in \LL_{[\omega,\alpha)}$ and the set  $\{\fm \in \mfL_{<\omega}:f_{[\fm]} \neq 0\}$ is well-based. To justify defining
$$f\circ g\ :=\ \sum_{\fm \in \mfL_{<\omega}}\left(f_{[\fm]}\circ g\right)\left(\fm\circ g\right)$$ we need the following:

\begin{prop}\label{p:WellDef}
The sum $\sum_{\fm \in \mfL_{<\omega}}\left(f_{[\fm]}\circ g\right)\left(\fm\circ g\right)$ exists.
\end{prop}

\noindent Towards establishing this proposition we define ``$\beta$-summability'' and prove some lemmas about it. Let $\beta < \alpha$. Then for $h\in \LL_{<\alpha}$ we have $h = \sum_{\fn \in \mfL_{<\beta}} h_{[\fn]}\fn$ with all $h_{[\fn]}\in \LL_{[\beta,\alpha)}$, and well-based $\supp_\beta(h) := \{\fn\in \mfL_{<\beta}:\ h_{[\fn]} \neq 0\}$. 

Let $(h_i)_{i \in I}$ be a family of elements in $\LL_{<\alpha}$.
We say that $(h_i)$ is \emph{$\beta$-summable} if $\bigcup_{i \in I}\supp_\beta(h_i)$ is well-based and $\{i \in I:\fn \in \supp_\beta(h_i)\}$ is finite for every $\fn \in \mfL_{<\beta}$. If $(h_i)$ is a $\beta$-summable, then $\sum_{i\in I}h_i$ exists as an element of the Hahn field $\LL_{[\beta,\alpha)}[[\mfL_{<\beta}]]$ over $\LL_{[\beta,\alpha)}$, and therefore also as an
element of the Hahn field $\LL_{<\alpha}$ over $\R$ with the same value under the usual identification of $\LL_{[\beta,\alpha)}[[\mfL_{<\beta}]]$ with $\LL_{<\alpha}$. If $(h_i)$ is summable, then $\bigcup_{i \in I}\supp_\beta(h_i)$ is well-based, but $(h_i)$ is not necessarily $\beta$-summable. If $(h_i)$ is $\beta$-summable, then $(c_ih_i)$ is also $\beta$-summable for any family $(c_i)$ from $\LL_{[\beta,\alpha)}$.
For $\fm\in \mfL_{<\alpha}$ we have
$\fm = \fm_{< \beta}\fm_{\geq \beta}$ with $\fm_{< \beta}\in \mfL_{<\beta},\ \fm_{\geq \beta}\in \mfL_{[\beta, \alpha)}$, and then $\supp_\beta(\fm)=\{\fm_{<\beta}\}$.
Here is a consequence of Lemma \ref{l:SmallCompForm}:

\begin{cor}\label{c:SmallComp}
Suppose the family $(\fm_i)_{i \in I}$ in $\mfL_{<\omega}$ is summable. Then the family $(\fm_i \circ g)$ is $(\lambda_g+\omega)$-summable.
\end{cor}
\begin{proof}
Set $\beta:= \lambda_g+\omega$. Lemma \ref{l:SmallCompForm} gives a well-based set $\fS\subseteq \mathfrak{L}_{<\alpha}$ such that for every monomial $\fm = \prod_{n}\ell_n^{r_n} \in \mfL_{<\omega}$, we have
\[
\supp(\fm \circ g)\ \subseteq\  \fd(g)^{r_0}\left(\prod_{n\ge 1}\ell_{\lambda_g+n}^{r_n}\right)\cdot\fS.
\]
Set $\fS_{<\beta} := \{\fn_{<\beta}:\ \fn \in \fS\}$. This set is still well-based and we have for $\fm\in \mathfrak{L}_{<\omega}$:
\[
\supp_\beta(\fm \circ g)\ \subseteq\  \fg^{r_0}\left(\prod_{n\ge 1}\ell_{\lambda_g+n}^{r_n}\right)\cdot\fS_{<\beta} , \qquad \fg\ :=\ \fd(g)_{<\beta}.
\]
Now $\fg=\prod_n \ell_{\lambda_g +n}^{s_n}$ with reals $s_n$ and $s_0>0$, so we have
an embedding 
$$\Phi\ :\  \mathfrak{L}_{<\omega} \to \mathfrak{L}_{[\lambda_g, \beta)}, \qquad \fm=\prod_n \ell_n^{r_n}\ \mapsto\ \fg^{r_0}\prod_{n\ge 1} \ell_{\lambda_g +n}^{r_n},$$
of ordered groups. Suppose towards a contradiction that $(\fm_i \circ g)$ is not $\beta$-summable. Then we have a sequence $(i_n)$ of distinct indices and an increasing sequence $(\fn_n)$ in $\mfL_{<\beta}$ with
$\fn_n\in \supp_{\beta}(\fm_{i_n}\circ g)$ for all $n$. By passing to a subsequence we arrange that $(\fm_{i_n})$ is strictly
decreasing.
Now $\fn_n=\Phi(\fm_{i_n})\fv_n$ with $\fv_n\in \fS_{<\beta}$, and $\Phi$ is order-preserving, so
$(\fv_n)$ is strictly increasing, contradicting that $\fS_{<\beta}$  is well-based. \end{proof}

\begin{lemma}\label{f:ReductionToOmega} Let $0 < \gamma < \alpha$ and $f\in \LL_{<\alpha}$. Then
\[
\supp_{\gamma+\omega}(f \circ \ell_\gamma)\ =\ \{\fm \circ \ell_\gamma:\ \fm \in \supp_\omega(f)\}.
\]
\end{lemma}
\begin{proof} We have 
$f = \sum_{\fm \in \supp_{\omega}(f)}f_{[\fm]}\fm$
where all $f_{[\fm]}\in \LL_{[\omega,\alpha)}$. Then 
\[f\circ \ell_{\gamma}\ =\  \sum_{\fm \in \supp_{\omega}(f)}(f_{[\fm]}\circ \ell_{\gamma})(\fm\circ \ell_{\gamma}).\]
It remains to note that by Lemmas~\ref{uptrans} and ~\ref{upmon} we have $f_{[\fm]}\circ \ell_{\gamma}\in \LL_{[\gamma+\omega,\alpha)}$ and $\fm\circ \ell_{\gamma}\in \mathfrak{L}_{<\gamma+\omega}$, for all $\fm\in \supp_{\omega}(f)$. 
\end{proof}

\noindent
Next a result about the map $T_g$ introduced in (\ref{e:TaylorO}). It involves the set
\[\mathfrak{D}\ :=\ \left\{\prod_{m}\ell_m^{-d_m}:\ d_0, d_1,d_2,\dots\in \N,\ d_0\ge d_1\geq d_2\geq \cdots\right\}\ \subseteq\ \mathfrak{L}_{<\omega}.\]
Note that $\mathfrak{D}$ is well-based by Corollary~\ref{morewell}.

\begin{lemma}\label{suppTg} Let $\phi\in \LL_{[\omega,\alpha)}$ and set $\beta:= \lambda_g+\omega$. Then 
\[ \supp_{\beta} T_g(\phi)\ \subseteq\  (\mathfrak{D}\circ\ell_{\lambda_g+3})\cdot (\supp_{\beta}\varepsilon)^\infty,\] 
and the right hand side is a well-based subset of $\mathfrak{L}_{<\beta}$ independent of $\phi$. 
\end{lemma} 
\begin{proof} The operator support of the derivation $\der_{\alpha}$ on $\LL_{<\alpha}$
and $\supp_\omega(\phi) \subseteq \{1\}$ give
\[
\supp_\omega(\phi^{(n)})\ \subseteq\ \left\{\prod_{m}\ell_m^{-d_m}:\ d_0, d_1,d_2,\dots\in \N,\ n= d_0\ge d_1\geq d_2\geq \cdots\right\}.
\]
Thus by Lemmas~\ref{upmon} and ~\ref{f:ReductionToOmega}, 
and $\lambda_g+3 +\omega= \lambda_g+\omega=\beta$, 
\[
\bigcup_{n}\supp_\beta\left(\phi^{(n)}\circ \ell_{\lambda_g+3}\right)\ \subseteq\ \mathfrak{D}\circ \ell_{\lambda_g+3}\ \subseteq\ \mathfrak{L}_{<\beta}. \]
Now $\mathfrak{D}\circ \ell_{\lambda_g+3}$ is well-based by
Lemma~\ref{upmon}, and $\varepsilon \preceq \ell_{\lambda_g+2}^{-1} $ gives $\supp_{\beta}\varepsilon\prec 1$. 
\end{proof} 

\begin{proof}[Proof of Proposition \ref{p:WellDef}]
We shall prove the stronger result that the family $$\left((f_{[\fm]}\circ g)(\fm \circ g)\right)_{\fm \in \supp_\omega(f)}$$ is $\beta$-summable for $\beta := \lambda_g+\omega$.
Suppose towards a contradiction that $(\fm_i)_{i<\omega}$ is a sequence of distinct elements of $\supp_\omega(f)$ and $(\fn_i)_{i<\omega}$ is an
increasing sequence in $\mfL_{<\beta}$ such that $\fn_i \in \supp_\beta\left((f_{[\fm_i]}\circ g)(\fm_i \circ g)\right)$ for all $i$. We have $\fn_i = \fp_i\fq_i$ where $\fp_i \in \supp_\beta(f_{[\fm_i]}\circ g)$ and 
$\fq_i\in \supp_\beta(\fm_i \circ g)$. By Corollary \ref{c:SmallComp}, the family $(\fm_i \circ g)$ is $\beta$-summable and so $(\fq_i)$ has a strictly decreasing subsequence. Thus $(\fp_i)$ has a strictly increasing subsequence. 
This contradicts the fact that Lemma~\ref{suppTg} gives a well-based set $\mathfrak{S}_g\subseteq \mathfrak{L}_{<\beta}$ such that $\supp_{\beta}\phi\circ g\subseteq \mathfrak{S}_g$ for all $\phi\in \LL_{[\omega,\alpha)}$.
\end{proof}

\noindent
We can now define
\begin{equation}
\label{e:FullComp}
f \circ g\  :=\ \sum_{\fm \in \mfL_{<\omega}}\left(f_{[\fm]}\circ g\right)\left(\fm\circ g\right).
\end{equation}
This agrees for $f=\fm\in \mathfrak{L}_{<\omega}$ and $f\in\LL_{[\omega,\alpha)}$ with $f\circ g$ as defined in (\ref{e:ProductOfLogs}) and (\ref{e:TaylorOverOmega}). It also agrees for $g=\ell_{\gamma}$ ($\gamma< \alpha$) with $f\circ \ell_{\gamma}$ as defined at the end of the previous section, but this requires an argument: For such $g$ we have $\lambda_g=\gamma$ and 
$\log_3(g)=\ell_{\gamma+3}$, $\varepsilon=0$, and so for $f\in \LL_{[\omega,\alpha)}$ and with $f\circ g$ as defined by (\ref{e:TaylorOverOmega}), 
$$f\circ g\ =\ f^{\uparrow_3}\circ \ell_{\gamma+3}\ =\ (f^{\uparrow_3}\circ \ell_3)\circ \ell_{\gamma}\ =\ f\circ \ell_{\gamma}$$ 
where the second equality uses Lemma~\ref{circadd}. For $f=\fm=\prod_n \ell_n^{r_n}\in \mathfrak{L}_{<\omega}$, we have
$$f\circ g\ =\ \prod_n \log_n(g)^{r_n}\ =\ \prod_n \ell_{\gamma+n}^{r_n}\ =\ \fm\circ \ell_{\gamma}$$
where for the last equality we use the first part of Lemma~\ref{upmon}. For arbitrary $f\in \LL_{<\alpha}$ we have $f=\sum_{\fm\in \supp_{\omega}(f)}f_{[\fm]}\fm$
with all $f_{[\fm]}\in \LL_{[\omega,\alpha)}$, and then
$$f\circ g\ =\ \sum_{\fm}(f_{[\fm]}\circ g)(\fm\circ g)\ =\ \sum_{\fm} (f_{[\fm]}\circ \ell_{\gamma})(\fm\circ \ell_{\gamma})\ =\ f\circ \ell_{\gamma}
$$
where the last equality uses Corollary~\ref{c:HyperlogComp} and $f\circ \ell_{\gamma}$ is defined as in (\ref{circLongCantor}).

\begin{lemma}\label{l:FullProto}
The map $f \mapsto f \circ g:\LL_{<\alpha} \to \LL_{<\alpha}$ is an $\LL_{<\alpha}$-composition with $g$.
\end{lemma}
\begin{proof}
Let $\fm,\fn \in \mfL_{<\alpha}$; to get $(\fm\fn)\circ g=(\fm\circ g)\cdot (\fn\circ g)$, we use the decomposition $\fm = \fm_{<\omega}\fm_{\geq \omega}$ where $\fm_{<\omega} \in \mfL_{<\omega}$ and $\fm_{\ge \omega} \in \mfL_{[\omega,\alpha)}$, and likewise for
$\fn$. The desired equality then follows from the relevant definitions and properties we already stated.  

 To verify that clause (3) in the definition of
``$K$-composition with $h$'' is satisfied, with $\fM=\fN=\mathfrak{L}_{<\alpha}$ and $g$ in the role of $h$,
let $f\in \LL_{<\alpha}$ and let $P$ be the set of pairs $(\fm,\fn)$ with
$\fm\in \mathfrak{L}_{<\omega}$, $\fn\in \mathfrak{L}_{[\omega,\alpha)}$,
and $\fm\fn\in \supp(f)$. Then $f=\sum_{(\fm,\fn)\in P}f_{\fm\fn}\fm\fn$, with all $f_{\fm\fn}\in \R^\times$; our job is to show
that then $\big((\fm\fn)\circ g\big)_{(\fm,\fn)\in P}$ is summable and $\sum_{(\fm,\fn)\in P}f_{\fm\fn}\big((\fm\fn)\circ g\big)=f\circ g$.
Note that 
$$\supp_{\omega}f\ =\ \{\fm\in \mathfrak{L}_{<\omega}:\ (\fm, \fn)\in P \text{ for some }\fn\in \mathfrak{L}_{[\omega,\alpha)}\}.$$
For $\fm\in \supp_{\omega}f$ we set $P(\fm):=\{\fn\in \mathfrak{L}_{[\omega,\alpha)}:\ (\fm, \fn)\in P\}$ and
$$f_{[\fm]}\ :=\ \sum_{\fn\in P(\fm)}f_{\fm\fn}\fn\ \in\ \LL_{[\omega,\alpha)}.$$
Then $f=\sum_{\fm\in \supp_{\omega}f}f_{[\fm]} \fm$, so $f\circ g=\sum_{\fm\in \supp_{\omega}f}(f_{[\fm]}\circ g)(\fm\circ g)$; we shall need this equality at the end of the proof and first address the summability issue.   

Suppose towards a contradiction that $\big((\fm\fn)\circ g\big)_{(\fm,\fn)\in P}$ is not summable. Then we have a sequence $(\fm_i, \fn_i)_{i<\omega}$ of distinct elements in $P$ and an increasing sequence $(\fg_i)$ in $\mathfrak{L}_{<\alpha}$ with $\fg_i\in \supp \big((\fm_i\fn_i)\circ g\big) $ for all $i$. 
Now $\fm_i\in \supp_{\omega} f$ for all $i$; by passing to a subsequence we arrange that either $(\fm_i)$ is constant, or $(\fm_i)$ is strictly decreasing, and so we now consider these two cases.  

\medskip\noindent
{\bf Case 1}: $(\fm_i)$ is constant, say $\fm_i=\fm$ for all $i$. Then
$\fn_i\in P(\fm)$ for all $i$, and
$$ f_{[\fm]}\circ g\ =\ \sum_{\fn\in P(\fm)}f_{\fm\fn}(\fn\circ g), \qquad (\fm\circ g)(f_{[\fm]}\circ g)\ =\ \sum_{\fn\in P(\fm)}f_{\fm\fn}\big((\fm\fn)\circ g\big),$$
using the first part of the proof. In particular, the sum $\sum_{\fn\in P(\fm)}f_{\fm\fn}((\fm\fn)\circ g)$ exists, contradicting
that $(\fg_i)$ is increasing. So we must be in the next case: 

\medskip\noindent
{\bf Case 2}: $(\fm_i)$ is strictly decreasing. Now $$\fg_i\in \supp\big((\fm_i\fn_i)\circ g\big)\ \subseteq\ \supp\big((\fm_i\circ g)(\fn_i\circ g)\big),$$ so $\fg_i=\fp_i\fq_i$ with $\fp_i\in \supp(\fm_i\circ g)$ and $\fq_i\in \supp(\fn_i\circ g)$. Set $\beta:= \lambda_g+\omega$. 
Since $(\fm_i)$ is strictly decreasing, we can arrange by passing to a subsequence that
$((\fp_i)_{<\beta})$ is strictly decreasing, in view of Lemma~\ref{l:SmallCompForm} and
the proof of Corollary~\ref{c:SmallComp}. Since $(\fg_i)$ is increasing, 
$((\fq_i)_{<\beta})$ must be strictly increasing, contradicting Lemma~\ref{suppTg}. We have now shown that $\big((\fm\fn)\circ g\big)_{(\fm,\fn)\in P}$ is summable. Its sum is $f\circ g$:
\begin{align*}  \sum_{(\fm,\fn)\in P}f_{\fm\fn}\big((\fm\fn)\circ g\big)\ &=\ \sum_{\fm\in \supp_{\omega}f}\big(\sum_{\fn\in P(\fm)}f_{\fm\fn}(\fm\fn)\circ g\big)\\
&=\ \sum_{\fm\in \supp_{\omega}f}\big(\sum_{\fn\in P(\fm)}f_{\fm\fn}(\fn\circ g)(\fm\circ g)\big)\\
&=\  
\sum_{\fm\in \supp_{\omega}f}\big(\sum_{\fn\in P(\fm)}f_{\fm\fn}(\fn\circ g)\big)\cdot\fm\circ g\\
&=\ \sum_{\fm\in \supp_{\omega}f}(f_{[\fm]}\circ g)(\fm\circ g)\ 
=\ f\circ g. \qedhere
\end{align*}    
\end{proof}

\subsection*{Extending to $\LL$} For $f,g\in \LL$ with $g>\R$
there are of course many $\alpha$ for which $f,g\in \LL_{<\alpha}$, and for each of those we have a value $f\circ g\in \LL_{<\alpha}$;
it is easy to check that this value $f\circ g$ is independent of $\alpha$.
Thus we have constructed a map 
$$(f,g)\mapsto f\circ g: \LL\times \LL^{>\R}\to \LL.$$ 
In the next section we show that this map is a composition on $\LL$ as defined in the introduction. {\bf Given $f,g\in \LL$ with $g>\R$ we let from now on
 $f\circ g$ denote the value of the above map at $(f,g)$.} We continue nevertheless our work in the setting of
 a Hahn field $\LL_{<\alpha}=\R[[\mathfrak{L}_{<\alpha}]]$ with $\alpha$ as before.

\section{Properties of Composition}\label{sec:pc}

\noindent
As before, $\alpha = \omega^\lambda$, where $\lambda$ is an infinite limit ordinal. Our job is to show that the map $(f,g)\mapsto f\circ g: \LL_{<\alpha}\times \LL_{<\alpha}^{>\R}\to \LL_{<\alpha}$
defined in the previous section is a composition on $\LL_{<\alpha}$ in the sense of Section~\ref{comp}. We
first prove the chain rule, and then use this to derive associativity.  But our proof of the chain rule requires a special case of associativity,
namely $\log(f\circ g) = (\log f)\circ g$ for $f \in \LL_{<\alpha}^>$ and $g \in \LL_{<\alpha}^{>\R}$. This is our starting point: 

\subsection*{Compatibility of taking logarithms and composition}
The first lemma treats the case that $g$ in $f\circ g$ is a hyperlogarithm.

\begin{lemma}\label{l:logassoclog}
Let $f \in \LL_{<\alpha}^>$ and $\gamma<\alpha$. Then $\log(f \circ \ell_\gamma) = (\log f)\circ \ell_\gamma$.
\end{lemma}
\begin{proof}
Suppose $\gamma=\omega^\beta$, and set $\mu := \omega^{\beta+1}$.
By Lemma \ref{l:PartialChain}, 
\[
\big(\log(f\circ \ell_\gamma)\big)'\ =\ \frac{(f\circ \ell_\gamma)'}{f\circ \ell_\gamma}\ =\ \frac{f'\circ \ell_\gamma}{f\circ \ell_\gamma}\ell_\gamma'\ =\ \big((\log f)'\circ \ell_\gamma\big)\ell_\gamma'\ =\ \big((\log f)\circ \ell_\gamma\big)',
\]
So it remains to show that $\log(f \circ \ell_\gamma)$ and $(\log f)\circ \ell_\gamma$ have the same constant term. We have $f = c\fm(1+\epsilon)$, where $c\in \R^{>}$, $\fm = \prod_{\rho<\alpha}\ell_\gamma^{r_\rho}$, and $\epsilon \prec 1$. Then
\[
\log f\ =\ \log c+\sum_{\rho<\alpha}r_\rho\ell_{\rho+1} + \log(1+\epsilon)
\]
As $\log(1+\epsilon)$ is infinitesimal and  $\ell_{\rho+1}\circ \ell_\gamma$ has constant term $0$ by Corollary \ref{comgam} and its proof, the constant term of $(\log f)\circ \ell_\gamma$ is $\log c$.
Note that $\fm \circ \ell_\gamma$ has leading coefficient 1: if $\fm \in  \mfL_{<\mu}$, then this follows from (\ref{e:BigMonomials}); if $\fm \in \LL_{[\mu,\alpha)}$, then it follows from $\fm \circ \ell_\gamma \sim \fm$,
which holds by Lemma \ref{car}. Since
\[
f \circ \ell_\gamma\ =\ c(\fm \circ \ell_\gamma)\big(1+(\epsilon\circ \ell_\gamma)\big)
\]
and $\epsilon\circ \ell_\gamma \prec 1$, the leading coefficient of $f \circ \ell_\gamma$ is $c$,  so the constant term of $\log(f \circ \ell_\gamma)$ is $\log c$ as well. The general case now follows by induction on $k$ in (\ref{circLongCantor}).
\end{proof}

\noindent We now turn our attention to $\LL_{[\omega,\alpha)}$ and the map $f \mapsto f^{\uparrow_3}$. Note that if $f \in  \LL_{[\omega,\alpha)}^>$, then also $\log f\in \LL_{[\omega,\alpha)}$. Thus the statement of the following lemma makes sense:

\begin{lemma}\label{logup3} Let $f\in \LL_{[\omega,\alpha)}^{>}$. Then
$\log(f^{\uparrow_3}) = (\log f)^{\uparrow_3}$.
\end{lemma}
\begin{proof}
Using Lemma \ref{l:logassoclog} we have
\[
\log(f^{\uparrow_3})\circ \ell_3\ =\ \log(f^{\uparrow_3}\circ \ell_3)\ =\ \log f\ =\ (\log f)^{\uparrow_3}\circ \ell_3
\]
and so $\log(f^{\uparrow_3}) = (\log f)^{\uparrow_3}$.
\end{proof}

\begin{lemma}\label{l:logassocsmall} Let $f \in \LL_{[\omega,\alpha)}^>$ and $g \in \LL_{<\alpha}^{>\R}$. Then
$\log(f \circ g) = (\log f)\circ g$. \end{lemma}
\begin{proof}
Let $T_g$ be the Taylor deformation in (\ref{e:TaylorO}), so $f \circ g = T_g(f^{\uparrow_3})$. By Lemma \ref{l:logassoclog} we have $\log(f^{\uparrow_3} \circ \ell_{\lambda_g+3}) =\log(f^{\uparrow_3}) \circ \ell_{\lambda_g+3}$. Then Lemmas~\ref{l:EqualLogs} and ~\ref{logup3} give 
\[
\log(f\circ g)\ =\  \log T_g(f^{\uparrow_3})\ =\ T_g\big(\log(f^{\uparrow_3})\big)\ =\ T_g\big((\log f)^{\uparrow_3}\big)\ =\ (\log f)\circ g.\qedhere
\] 
\end{proof}

\noindent
We can now prove the main result of this subsection:
 
\begin{prop}\label{l:logassoc} Let $f \in \LL_{<\alpha}^>$ and $g \in \LL_{<\alpha}^{>\R}$. Then $\log(f\circ g) = (\log f) \circ g$.
\end{prop}
\begin{proof}
We have $f = c\fd(f)(1+\epsilon)$ where $c \in \R^>$ and $\epsilon \prec 1$.
Then
\[
(\log f) \circ g\ =\ \log c +  \big(\log \fd(f)\big)\circ g + 
\big(\log(1+\epsilon)\big)\circ g.
\]
The strong linearity of composition with $g$ gives $\big(\log(1+\epsilon)\big)\circ g = \log\big((1+\epsilon)\circ g\big)$. Thus, it remains to show that $\big(\log \fd(f)\big)\circ g = \log\big(\fd(f) \circ g\big)$. Now $\fd(f) = \fm\fn$ where $\fm \in \mfL_{<\omega}$ and $\fn \in \mfL_{[\omega,\alpha)}$, so $\log \fd(f) = \log \fm + \log \fn$. By Lemma \ref{l:logassocsmall}, $(\log \fn)\circ g  =\log(\fn\circ g)$. We have $\fm = \prod_{n}\ell_n^{r_n}$, so 
\[
(\log \fm)\circ g\ =\ \left(\sum_{n}r_n\ell_{n+1}\right)\circ g\ =\ \sum_{n}r_n\log_{n+1}(g)\ =\ \log(\fm\circ g),
\]
where the last equality uses Lemma~\ref{l:SumOfLogs}. Thus $\big(\log \fd(f)\big)\circ g\ =\ \log\big(\fd(f)\circ g\big)$.
\end{proof}

\begin{cor}\label{c:FullPowers}
Let $f \in \LL_{<\alpha}^{>}$, $g \in \LL_{<\alpha}^{>\R}$, and $t \in \R$. Then $(f \circ g)^t = f^t\circ g$.
\end{cor}
\begin{proof} Take logarithms and use Proposition~\ref{l:logassoc}.
\end{proof}

\begin{cor}\label{c:products} Suppose the family $(f_i)$ in $\LL_{<\alpha}^{>}$ is multipliable and $g\in \LL_{<\alpha}^{>\R}$. Then the family 
$(f_i\circ g)$ is multipliable and $(\prod_i f_i)\circ g=\prod_i f_i\circ g$.
\end{cor}
\begin{proof} Take logarithms and use Proposition~\ref{l:logassoc}.
\end{proof}

\subsection*{The chain rule} Let $g\in \LL_{<\alpha}^{>\R}$. Recall that $T_g$ is a strongly $\R$-linear endomorphism of the ordered field $\LL_{<\alpha}$. We show that $T_g$ coincides with
$\phi\mapsto (\phi\circ \ell_3)\circ g$: 

\begin{lemma}\label{ltg} For all $\phi\in \LL_{<\alpha}$ we have $T_g(\phi)= (\phi\circ \ell_3)\circ g$.
\end{lemma} 
\begin{proof} The map  $\phi\mapsto (\phi\circ \ell_3)\circ g: \LL_{<\alpha}\to \LL_{<\alpha}$ is also
a strongly $\R$-linear endomorphism of the ordered field $\LL_{<\alpha}$; it agrees with $T_g$ on $\LL_{[\omega,\alpha)}$, since for $\phi\in \LL_{[\omega,\alpha)}$ we have
$\phi\circ \ell_3\in \LL_{[\omega,\alpha)}$ and $\phi=(\phi\circ \ell_3)^{\uparrow 3}$. By the strong linearity of both maps it is enough to prove the lemma for $\phi\in \mathfrak{L}_{<\alpha}$, and for such
$\phi$ we have $\phi=\fm\fn$ with $\fm\in \mathfrak{L}_{<\omega}$ and
$\fn\in \mathfrak{L}_{[\omega,\alpha)}$. Thus it is enough to show that
$T_g(\fm)= (\fm\circ \ell_3)\circ g$ for $\fm\in \mathfrak{L}_{<\omega}$. 
Taking logarithms this reduces to showing for such $\fm$ that
 $\log T_g(\fm)= \log\big((\fm\circ \ell_3)\circ g\big)$; by Lemma~\ref{l:EqualLogs} and Proposition~\ref{l:logassoc}, this is equivalent to $T_g(\log \fm)=\big(\log(\fm\circ \ell_3)\big)\circ g$, and thus by Proposition~\ref{l:logassoc} to $T_g(\log \fm)=\big((\log \fm)\circ \ell_3\big)\circ g$.  Since for $\fm\in \mathfrak{L}_{<\omega}$ we have
 $\log \fm=\sum_n r_n\ell_{n+1}$ this reduces further to showing that
 $T_g(\ell_n)=(\ell_n\circ \ell_3)\circ g$ for all $n$. 
 This holds for $n=0$ by earlier remarks, and for arbitrary $n$ it follows by induction on $n$, using
 again at each step Lemma~\ref{l:EqualLogs} and Proposition~\ref{l:logassoc}.
\end{proof}

\begin{prop}\label{l:FullChain} Let $f \in \LL_{<\alpha}$. Then
$(f \circ g)' =  (f'\circ g) g'$.
\end{prop}
\begin{proof} Note that for $f\ne 0$ we have: $(f\circ g)'= (f'\circ g)g'\Leftrightarrow (f\circ g)^\dagger=(f^\dagger\circ g)g'$.
An easy induction gives $\log_n(g)^\dagger = (\ell_n^\dagger\circ g)g'$. Thus for $\fm = \prod_{n}\ell_n^{r_n}\in \mfL_{<\omega}$,  
\begin{align*}
(\fm \circ g)^\dagger&= \big(\prod_{n}\log_n(g)^{r_n}\big)^\dagger\ =\ \sum_n r_n\log_n(g)^\dagger\ =\ \sum_n r_n(\ell_n^\dagger \circ g)g'\\
&=\ \ \big(\sum_n r_n (\ell_{n}^\dagger\circ g)\big)g'\ 
 =\ \big((\sum_{n}r_n\ell_n^\dagger)\circ g\big)g'\ 
=\ (\fm^\dagger\circ g)g',
\end{align*}
so $(\fm \circ g)'=(\fm'\circ g)g'$.
Let $f \in \LL_{[\omega,\alpha)}$. Then $f\circ g=T_g(f^{\uparrow_3})$, so
by (\ref{e:chaintg}),
\begin{equation}
\label{cr1} (f\circ g)'\ =\ \big(T_g(f^{\uparrow_3})\big)'\ =\ T_g\big((f^{\uparrow_3})'\big)\cdot \log_3(g)'\ =\ T_g\big((f^{\uparrow 3})'\big)\cdot (\ell_3'\circ g)\cdot g'.
\end{equation}
We have $f'=(f^{\uparrow_3}\circ \ell_3)'=\big((f^{\uparrow_3})'\circ \ell_3\big)\ell_3'$ by (\ref{compup3}) and Corollary \ref{c:HyperlogChain}, 
so $(f^{\uparrow_3})'\circ \ell_3=(f'/\ell_3')$. Applying Lemma~\ref{ltg} to
$\phi:= (f^{\uparrow_3})'$, this gives 
\begin{equation}
\label{cr2} T_g\big((f^{\uparrow_3})'\big)\  =\ (f'/\ell_3')\circ g\ =\ (f'\circ g)/(\ell_3'\circ g).
\end{equation}
Combining (\ref{cr1}) and (\ref{cr2}) gives $(f\circ g)'=(f'\circ g)g'$. 
Finally, for arbitrary $f \in \LL_{<\alpha}$ we have $f = \sum_{\fm \in \mfL_{<\omega}}f_{[\fm]}\fm$ where all $f_{[\fm]}\in \LL_{[\omega,\alpha)}$. Then
\begin{align*}
(f\circ g)'\  &=\  \sum_{\fm}\big((f_{[\fm]}\circ g)( \fm\circ g)\big)'\\
&=\ \sum_{\fm}\big((f_{[\fm]}\circ g)'( \fm\circ g)+(f_{[\fm]}\circ g)( \fm\circ g)'\big)\\
&=\  g'\sum_{\fm}\big((f_{[\fm]}'\circ g)( \fm\circ g)+(f_{[\fm]}\circ g)( \fm'\circ g)\big) \\
&=\  g'\sum_{\fm}\big((f_{[\fm]} \fm)'\circ g\big)= g'(f' \circ g).\qedhere
\end{align*}
\end{proof}

\subsection*{Associativity} Towards proving associativity we use the next lemma to get a handle on the infinite part and constant term of  $f\circ g$ for various $f,g$.  {\em Throughout this subsection we fix  $g,h \in \LL_{<\alpha}^{>\R}$}.

\begin{lemma}\label{l:infcomp}
Assume $\omega\le\gamma < \alpha$. Then $\ell_\gamma\circ g  -\ell_{\gamma}\circ \ell_{\lambda_g}\prec 1$.
\end{lemma}
\begin{proof}
For $n\ge 1$ we have $\ell_{\gamma}^{\uparrow_3}\circ \ell_3=\ell_{\gamma}\prec \ell_{n+3}=\ell_n\circ \ell_3$, hence $\ell_{\gamma}^{\uparrow_3}\prec \ell_n$,
and thus $(\ell_{\gamma}^{\uparrow_3})'\prec \ell_n'\prec 1$. 
Thus by (\ref{e:TaylorOverOmega}), 
\[
\ell_\gamma \circ g-\left(\ell_\gamma^{\uparrow_3}\circ \ell_{\lambda_{g+3}}\right)\ =\ \sum_{n=1}^\infty \frac{(\ell_\gamma^{\uparrow_3})^{(n)}\circ \ell_{\lambda_g+3}}{n!}\left(\log_3(g)-\ell_{\lambda_{g+3}}\right)^n\ \prec\ 1.
\]
Now $\ell_{\lambda_g+3} = \log_3(\ell_{\lambda_g})$, so by (\ref{e:TaylorOverOmega}) with $\ell_{\lambda_g}$ in the role of $g$ and $\epsilon=0$,  
\[
\ell_\gamma \circ \ell_{\lambda_g}\ =\ \ell_\gamma^{\uparrow_3}\circ \ell_{\lambda_g+3}.\qedhere
\]
\end{proof}

\noindent
Combining Lemmas~\ref{bg} and ~\ref{l:infcomp} gives:

\begin{cor}\label{domconstant}
Assume $\omega\le\gamma<\alpha$. Then 
\[\ell_\gamma\circ g\  =\ \ell_{\lambda_g+\gamma}-\lambda_{g;\gamma}+\epsilon, \qquad \epsilon \prec 1.\]
\end{cor}

\begin{lemma}\label{logofcomp} We have $\lambda_{(g\circ h)}=\lambda_h+\lambda_g$ $($ordinal sum$)$.
\end{lemma}
\begin{proof}
Recall  that $\fd(\log f) = \ell_{\lambda_f+1}$ for $f  \in \LL^{>\R}$, so it suffices to show:
\[
\log(g \circ h)\ \asymp\ \ell_{\lambda_h+\lambda_g+1}.
\]
By Proposition~\ref{l:logassoc}, we have 
\[
\log(g \circ h)\ =\  (\log g)\circ h\ \asymp\ \ell_{\lambda_g+1}\circ h.
\]
We end by noting that $\ell_{\lambda_g+1}\circ h \asymp \ell_{\lambda_h+\lambda_g+1}$; this is a consequence of the remarks after Lemma~\ref{fdprod} if $\lambda_g$
is finite, and  follows by Corollary \ref{domconstant} if $\lambda_g \geq \omega$.
\end{proof}

\noindent
To use this lemma we recall that for ordinals $\mu$ and $\nu$  with Cantor normal forms
$$\mu\ =\ \omega^{\beta_1}m_1 + \cdots + \omega^{\beta_k}m_k, \qquad  \nu\ =\ \omega^{\gamma_1}n_1 + \cdots + \omega^{\gamma_l}n_l,\qquad (k,l\ge 1),$$  
the Cantor normal form of $\nu+\mu$ is as follows: \begin{itemize}
\item if $\gamma_1<\beta_1$, then $\nu+\mu=\mu$;
\item if $1\le j \le l$ and $\gamma_j = \beta_1$, then
\[\qquad\qquad \nu + \mu\ =\ \omega^{\gamma_1}n_1 + \cdots + \omega^{\gamma_j}(n_j+m_1)+\omega^{\beta_2}m_2 + \cdots + \omega^{\beta_k}n_k;\] 
\item if $1\le j \le l$ and $\gamma_j > \beta_1$, $\gamma_{j'}< \beta_1$ for all $j'$ with $j < j'\le l$, then
\[\nu + \mu\ =\ \omega^{\gamma_1}n_1 + \cdots + \omega^{\gamma_j}n_j+\omega^{\beta_1}m_1 + \cdots + \omega^{\beta_k}n_k.\]
\end{itemize}
Here $\mu,\nu\ge 1$. Below we also need the trivial case where $\mu=0$ or $\nu=0$.

\begin{lemma}\label{logofcomp2}
Let $\gamma\ge \omega$. Then
\[
\lambda_{(g\circ h);\gamma}\ =\ \lambda_{g;\gamma}+ \lambda_{h;\lambda_g+\gamma}.
\]
\end{lemma}
\begin{proof}
We express $\lambda_g$ and $\lambda_h$ in Cantor normal form (allowing
$k=0$ or $l=0$):
 $$\lambda_g\ =\ \omega^{\beta_1}m_1 + \cdots + \omega^{\beta_k}m_k, \qquad \lambda_h\ =\ \omega^{\gamma_1}n_1 + \cdots + \omega^{\gamma_l}n_l.$$
 Using Lemma \ref{logofcomp} and the above remarks about Cantor normal forms, we have
\[
\lambda_{(g\circ h);\gamma}\ =\ \left\{
\begin{array}{ll}
0 &\mbox{if }  \gamma  \not\in \{\omega^{\beta_1+1},\dots, \omega^{\beta_k+1},\omega^{\gamma_1+1},\dots, \omega^{\gamma_l+1}\}\\
m_1 & \mbox{if } k\ge 1,\ \gamma  = \omega^{\beta_1+1}, \mbox{ and } \beta_1 \not\in \{\gamma_1,\dots,\gamma_l\}\\
m_i & \mbox{if } k>1,\ \gamma  = \omega^{\beta_i+1},\ 1<i\le k\\ 
n_j+m_1 &\mbox{if } k\ge 1,\ \gamma  = \omega^{\gamma_j+1},\ 1\le j\le l, \mbox{ and } \gamma_j = \beta_1\\
n_j  & \mbox{if } k\ge 1,\ \gamma  = \omega^{\gamma_j+1},\ 1\le j\le l, \mbox{ and } \gamma_j > \beta_1\\
0 & \mbox{if } k\ge 1,\ \gamma\notin\{\omega^{\beta_1+1},\dots, \omega^{\beta_k+1}\},\ \gamma=\omega^{\gamma_j+1},1\le j\le l, \gamma_j<\beta_1\\
n_j & \mbox{if }k=0,\ \gamma=\omega^{\gamma_j+1}, 1\le j\le l.
\end{array}
\right.
\]
It remains to calculate the values of $\lambda_{g;\gamma}$ and $\lambda_{h;\lambda_g+\gamma}$:
\begin{itemize}
\item If $\gamma\not\in \{\omega^{\beta_1+1},\dots, \omega^{\beta_k+1},\omega^{\gamma_1+1},\dots, \omega^{\gamma_l+1}\}$, then $\lambda_{g;\gamma} = 0$ and moreover $\lambda_g+\gamma\notin\{\omega^{\gamma_1+1},\dots, \omega^{\gamma_l+1}\}$, so $\lambda_{h;\lambda_g+\gamma} = 0$.
\item If $k\ge 1$, $\gamma  = \omega^{\beta_1+1}$ and $\beta_1 \not\in \{\gamma_1,\dots,\gamma_l\}$, then $\lambda_{g;\gamma} = m_1$ and $\lambda_g+\gamma = \gamma$, so $\lambda_{h;\lambda_g+\gamma} = 0$. 
\item If $k>1$, $\gamma  = \omega^{\beta_i+1}$, $1<i\le k$, then $\lambda_{g;\gamma} = m_i$, $\lambda_g+\gamma\notin\{\omega^{\gamma_1+1},\dots, \omega^{\gamma_l+1}\}$, so $\lambda_{h;\lambda_g+\gamma} = 0$.
\item If $k\ge 1$, $\gamma  = \omega^{\gamma_j+1}$, $1\le j\le l$, and $\gamma_j =\beta_1$, then $\lambda_{g;\gamma} = m_1$ and $\lambda_g+\gamma = \gamma$, so $\lambda_{h;\lambda_g+\gamma} = n_j$.
\item If $k\ge 1$, $\gamma  = \omega^{\gamma_j+1}$, $1\le j\le l$, and $\gamma_j > \beta_1$, then $\lambda_{g;\gamma} = 0$ and $\lambda_g+\gamma = \gamma$, so $\lambda_{h;\lambda_g+\gamma} = n_j$.
\item If $k\ge 1,\ \gamma\notin\{\omega^{\beta_1+1},\dots, \omega^{\beta_k+1}\},\ \gamma=\omega^{\gamma_j+1},\ \gamma_j<\beta_1$, then
$\lambda_{g;\gamma}=0$ and $\lambda_g+\gamma\notin\{\omega^{\gamma_1+1},\dots, \omega^{\gamma_l+1}\}$,   so $\lambda_{h;\lambda_g+\gamma} =0$.
\item If $k=0,\ \gamma=\omega^{\gamma_j+1},\ 1\le j\le l$, then $\lambda_{g;\gamma}=0$, $\lambda_g+\gamma=\gamma$, so $\lambda_{h;\lambda_g+\gamma}=n_j$.
\end{itemize}
Thus $\lambda_{(g\circ h);\gamma}\ =\ \lambda_{g;\gamma}+ \lambda_{h;\lambda_g+\gamma}$ in all cases. 
\end{proof}

\begin{lemma}\label{closeassoc}
Let $\omega\le \gamma< \alpha$. Then $\big(\ell_\gamma\circ (g\circ h)\big) - \big((\ell_\gamma\circ g)\circ h \big) \prec 1$. 
\end{lemma}
\begin{proof}
Corollary \ref{domconstant} gives $\epsilon, \epsilon^*\prec 1$ in $\LL_{<\alpha}$ such that
\begin{align*}
(\ell_\gamma\circ g)\circ h\  &=\ \ (\ell_{\lambda_g+\gamma} - \lambda_{g;\gamma}+\epsilon) \circ h\ =\  (\ell_{\lambda_g+\gamma} \circ h) -\lambda_{g;\gamma} + (\epsilon\circ h)\\
 &=\  \big(\ell_{\lambda_h+\lambda_g+\gamma} - \lambda_{h;\lambda_g+\gamma}+\epsilon^*\big)- \lambda_{g;\gamma} + (\epsilon\circ h).
\end{align*}
Corollary \ref{domconstant} and Lemmas \ref{logofcomp} and \ref{logofcomp2} give $\epsilon^{**}\prec 1$ in $\LL_{<\alpha}$ such that
\[
\ell_\gamma \circ (g\circ h)\ =\ \ell_{\lambda_{(g\circ h)}+\gamma} - \lambda_{(g\circ h);\gamma}+ \epsilon^{**}\ =\ \ell_{\lambda_h+\lambda_g+\gamma} -  \lambda_{h;\lambda_g+\gamma}- \lambda_{g;\gamma}+ \epsilon^{**}. \qedhere
\]
\end{proof}

\begin{prop}\label{fullassoc} For all $f\in \LL_{<\alpha}$ we have 
$f\circ(g\circ h)=(f\circ g)\circ h$.
\end{prop}

\begin{proof}  By strong linearity this reduces to 
$\fm\circ(g\circ h)=(\fm\circ g)\circ h$ for $\fm\in \mfL_{<\alpha}$.
Taking logarithms and using Proposition~\ref{l:logassoc} this reduces further to showing
$\ell_{\gamma}\circ(g\circ h)=(\ell_{\gamma}\circ g)\circ h$ for $\gamma<\alpha$. This goes by induction on
$\gamma$. The case $\gamma=0$ is obvious, and for the step from $\gamma$ to
$\gamma+1$ we take logarithms and use Proposition~\ref{l:logassoc}. Let now $\gamma<\alpha$ be an infinite limit ordinal.
Proposition \ref{l:FullChain} (and an inductive assumption for the second equality below) give
\begin{align*}
\big(\ell_\gamma \circ (g\circ h)\big)'\ &=\ \big(\ell_\gamma' \circ (g\circ h)\big)(g'\circ h)h'\ =\
 \big((\ell_\gamma' \circ g)\circ h)\big)(g'\circ h)h'\\
 &=\ \big(\big((\ell_\gamma' \circ g)g'\big)\circ h\big)h'\ =\ \big((\ell_\gamma \circ g)' \circ h\big)h'\ =\ \big((\ell_\gamma \circ g)\circ h\big)'
\end{align*}
so it remains to check that $\ell_{\gamma}\circ(g\circ h)$ and $(\ell_{\gamma}\circ g)\circ h$ have the same constant terms. This follows from Corollary \ref{closeassoc}.
\end{proof}

\subsection*{Characterizing composition recursively} 
The above shows that $$(f,g)\mapsto f\circ g\ :\ \LL_{<\alpha}\times \LL_{<\alpha}^{>\R}\to \LL_{<\alpha}$$ is a composition on $\LL_{<\alpha}$ as defined in Section~\ref{comp}. We have also shown that 
$$(f,g)\mapsto f\circ g\ :\ \LL\times \LL^{>\R}\to \LL$$ 
is a composition on $\LL$ as defined in the Introduction: it satisfies (CL1)--(CL5). This composition satisfies the following recursion:

\begin{cor}\label{recu}
Let $\gamma\ge \omega$ and $g \in \LL^{>\R}$. Then 
\[
\ell_\gamma\circ g\ =\ \int \big[\left(\ell_\gamma' \circ g\right) g'\big] - \lambda_{g;\gamma}.
\]
\end{cor}
\begin{proof} By Corollary~\ref{domconstant} the constant term of 
$\ell_{\gamma}\circ g$ equals $-\lambda_{g;\gamma}$. It remains to use the
Chain Rule, Proposition \ref{l:FullChain}. 
\end{proof}

\noindent 
Note that 
$\ell_{\gamma}'\circ g=\big(\prod_{\beta<\gamma}\ell_{\beta}\circ g\big)^{-1}$
for $g\in \LL^{>\R}$. In combination with the identities
$\ell_{n+1} \circ g = \log (\ell_n \circ g)$ for such $g$, this gives us the right to speak of a recursion. 

\begin{cor}\label{coru} There is a unique composition $\ast$ on $\LL$, namely $\circ$, that satisfies the above recursion: $\ell_\gamma\ast g =\int \big[\left(\ell_\gamma' \circ g\right) g'\big] - \lambda_{g;\gamma}$
for all $\gamma\ge \omega$ and $g\in \LL^{>\R}$.
\end{cor}
\begin{proof} Let $\ast$ be a composition on $\LL$ that satisfies the above recursion. By (CL4) (strong linearity) it suffices that 
$\fm\ast g=\fm\circ g$ for $\fm\in \mathfrak{L}$ and $g\in \LL^{>\R}$. By Lemma~\ref{lemuc} this reduces to $\ell_{\gamma}\ast g=\ell_{\gamma}\circ g$ for such $g$. This is taken care of by the recursion and transfinite induction. 
\end{proof}

\section{Taylor Expansion and Compositional Inversion}\label{teci}

\noindent
As before, $\alpha = \omega^\lambda$, where $\lambda$ is an infinite limit ordinal. In some proofs below we use the notation
$\supp S:=\bigcup_{f\in S} \supp f$ for $S\subseteq \LL_{<\alpha}$.

\subsection*{Taylor expansion}
In this subsection we fix $g,h\in \LL_{<\alpha}$ with $g>\R$ and $h\prec g$. Our goal here is the following Taylor identity:

\begin{prop}\label{tayexp}
If $f\in \LL_{<\alpha}$, then the sum $\sum_{n=0}^\infty\frac{f^{(n)}\circ g}{n!}h^n$ exists and \[f\circ (g+h)\ =\ \sum_{n=0}^\infty\frac{f^{(n)}\circ g}{n!}h^n.\]
\end{prop}

\noindent
Note that the assumption on $h$ is  weaker than $h\prec 1$; this weaker assumption works here because $\supp \der_{\alpha}\preceq x^{-1}$. We need three lemmas:

\begin{lemma}\label{Taylorexists}
 If $f\in \LL_{<\alpha}$, then $\sum_n\frac{f^{(n)}\circ g}{n!}h^n$ exists. The map $T: \LL_{<\alpha}\to \LL_{<\alpha}$
given by $T(f):= \sum_n\frac{f^{(n)}\circ g}{n!}h^n$ is an $\LL_{<\alpha}$-composition with $g+h$.
\end{lemma}
\begin{proof} First note that 
$\mathfrak{S}:=\supp\big((\supp\der_{\alpha})\circ g\big)$ is well-based:
this is because $\supp \der_{\alpha}=\{\ell_{\beta}^\dagger:\ \beta< \alpha\}$ is well-based, and the map $f\mapsto f\circ g: \LL_{<\alpha}\to \LL_{<\alpha}$ is strongly additive. Thus $\mathfrak{S}\cdot \supp h$ is well-based, and we claim that $\mathfrak{S}\cdot \supp h\prec 1$: this is because for
$\fm\in \supp(\der_{\alpha})$ we have $\fm\preceq x^{-1}$, so
$\fm\circ g\preceq x^{-1}\circ g=g^{-1}$, so for $\fn\in \mathfrak{S}$ we have $\fn\preceq g^{-1}$, and thus $\fn h\preceq g^{-1}h\prec 1$. 

Next, let $f\in \LL_{<\alpha}$ and $\fm\in \supp(f^{(n)}\circ g)$.
Then $\fm\in \supp(\fn\circ g)$ with $\fn\in \supp f^{(n)}$, so
$\fn=\fn_1\cdots \fn_n\fv$ with $\fn_1,\dots, \fn_n\in \supp(\der_{\alpha})$
and $\fv\in \supp f$, hence 
$$\fn\circ g\ =\ (\fn_1\circ g)\cdots(\fn_n\circ g)\cdot(\fv\circ g),$$
which gives $\fm\in \fS^n\cdot \fS_f$, where $\fS_f:=\supp\big((\supp f)\circ g\big)$. Thus we have shown:
$$\supp\big((f^{(n)}\circ g)h^n\big)\ \subseteq\ 
\big(\fS\cdot\supp h\big)^n\cdot \fS_f.$$
Now $\fS_f$ is well-based, so by Neumann's Lemma and what
we proved about $\fS\cdot \supp h$ we may conclude that $\sum_n\frac{f^{(n)}\circ g}{n!}h^n$ does exist. The map $T$ is clearly $\R$-linear with $T(1)=1$ and $T(x)=g+h$. 
Let $(f_i)_{i\in I}$ in $\LL_{<\alpha}$ be a summable family. Then $\bigcup_i\fS_{f_i}$ is well-based and the set $\{i \in I:\fm \in \fS_{f_i}\}$ is finite for every $\fm\in \mfL_{<\alpha}$. It follows that $\sum_{n,i}\frac{f_i^{(n)}\circ g}{n!}h^n$ exists, and so $\sum_iT(f_i)$ exists as well, and both sums equal $T(\sum_if_i)$. Thus $T$ is strongly $\R$-linear. A routine computation using Lemma~\ref{prfam} also gives
$T(f_1f_2)=T(f_1)T(f_2)$ for $f_1,f_2\in \LL_{<\alpha}$. 
\end{proof}

\begin{lemma}\label{Taylorlog}
 For $f\in \LL_{<\alpha}^{>}$ we have $T(f)\sim f\circ g$
and $\log T(f)=T(\log f)$. 
\end{lemma}
\begin{proof} For nonzero $f\in \LL_{<\alpha}$ and $n\ge 1$ we have
$$\supp (f^{(n)}\circ g)h^n\ \prec\ \max \fS_f\ =\ \max \supp(f\circ g)$$ with notations from the
proof of Lemma~\ref{Taylorexists}, and thus $T(f)\sim f\circ g$.
In view of Proposition~\ref{l:logassoc}, the rest now follows as in the proof of Lemma~\ref{l:EqualLogs}, with $f\mapsto f\circ g$
in the role of $\Phi$ and $h$ instead of $\epsilon$.  
\end{proof}

\begin{lemma}\label{l:smalldif}
Let $\gamma\geq \omega$ and $\gamma < \alpha$. Then $T(\ell_\gamma) - \big(\ell_\gamma\circ (g+h)\big) \prec 1$. 
\end{lemma}
\begin{proof} We have $\lambda_{g+h} = \lambda_g$, so by  Lemma \ref{l:infcomp},
$$\ell_\gamma \circ g -\ell_\gamma \circ \ell_{\lambda_g}\prec 1,\qquad  \ell_\gamma \circ (g+h) -\ell_\gamma \circ \ell_{\lambda_g}\prec 1,$$
and hence $\ell_\gamma \circ g -\ell_\gamma \circ (g+h)\prec 1$.  Since $T(\ell_\gamma)=\ell_{\gamma}\circ g + \sum_{n=1}^\infty \frac{1}{n!}(\ell_\gamma^{(n)}\circ g)h^n$, it therefore suffices to show that  $(\ell_\gamma^{(n)}\circ g)h^n \prec 1$ for all $n \geq 1$.  Let $n\ge 1$ and let $\fS$ be as in the proof of Lemma \ref{Taylorexists}. That proof for $f=\ell_{\gamma}'$ gives
$$\supp\big((\ell_\gamma^{(n)}\circ g)h^n\big)\ \subseteq\ 
\big(\fS\cdot\supp h\big)^{n-1}\cdot  \supp\big((\supp \ell_{\gamma}') \circ g\big)\cdot (\supp h).$$
Now $\supp \ell_{\gamma}'=\{\ell_\gamma' \} \prec x^{-1}$, so $(\supp \ell_\gamma' )\circ g \prec g^{-1}$. In view of $h \prec g$, this yields
 $\supp\big((\supp \ell_{\gamma}') \circ g\big)\cdot (\supp h) \prec 1$. Also $\fS\cdot\supp h \prec 1$, and thus  $(\ell_\gamma^{(n)}\circ g)h^n \prec 1$. 
\end{proof}

\begin{proof}[Proof of Proposition~\ref{tayexp}] Our job is to show that the above maps $f\mapsto f\circ (g+h)$
and $T$ agree. By the strong linearity of these maps this reduces to $\fm\circ (g+h)=T(\fm)$ for $\fm\in \mathfrak{L}_{<\alpha}$.
Taking logarithms and using that these maps commute with taking logarithms (Lemma~\ref{Taylorlog} and Proposition~\ref{l:logassoc}) this reduces further to $\ell_{\gamma}\circ (g+h)=T(\ell_{\gamma})$ for $\gamma<\alpha$. We prove this by induction on
$\gamma$. The case $\gamma=0$ is obvious, and for the step from $\gamma$ to
$\gamma+1$ we use again that the two maps commute with
taking logarithms. Let now $\gamma<\alpha$ be an infinite limit ordinal.
A routine computation using Proposition \ref{l:FullChain} gives $T(\ell_\gamma)' = T(\ell_\gamma')(g+h)'$.
Moreover, by Lemma~\ref{Taylorlog},
\[ \log T(\ell_\gamma')\ =\ T(\log \ell_\gamma')\ =\ T(-\sum_{\beta<\gamma} \ell_{\beta+1})\ =\  -\sum_{\beta<\gamma}T(\ell_{\beta+1}),\]
and likewise $\log \big(\ell_{\gamma}'\circ(g+h)\big)=-\sum_{\beta<\gamma}\ell_{\beta+1}\circ(g+h)$, so 
$ \log T(\ell_\gamma')=\log(\ell_\gamma'\circ(g+h))$ by the natural inductive assumption, hence $T(\ell_{\gamma}')=\ell_{\gamma}'\circ (g+h)$, and
thus
\[
T(\ell_\gamma)'\ =\ T(\ell_\gamma')(g+h)'\ =\ \left(\ell_\gamma' \circ (g+h)\right)(g+h)'\ =\ \big(\ell_\gamma \circ (g+h)\big)'
\]
by the chain rule. It remains to check that $T(\ell_\gamma)$ and $\ell_\gamma \circ (g+h)$ have the same constant term. This follows from Lemma \ref{l:smalldif}.
\end{proof}

\noindent
This gives the existence part of Theorem~\ref{uc}: we just showed that our composition $\circ$ admits Taylor expansion as stated in that theorem, and the other three items 
are respectively contained in Corollary~\ref{circadd}, Lemma~\ref{comgam}, and Lemma~\ref{bg}.

\subsection*{Compositional Inversion} For $f,g\in \LL_{<\alpha}^{>\R}$ we have $f\circ g\in \LL_{<\alpha}^{>\R}$. Thus
$\LL_{<\alpha}^{>\R}$ is a monoid with respect to composition and with $x$ as its identity element. Let us say that $g\in \LL_{<\alpha}^{>\R}$ is {\em $($compositionally$)$ invertible\/} if $f\circ g=x$
for some $f\in\LL_{<\alpha}$; note that such $f$ is unique and
satisfies $f\in \LL_{<\alpha}^{>\R}$ and $g\circ f=x$, since
$(g\circ f)\circ g=g\circ(f\circ g)=g$; we denote this unique $f$ by
$g^{\inv}$. Note that if $f,g\in\LL_{<\alpha}^{>\R}$ are invertible, then so are $f^{\inv}$ and $f\circ g$, with $(f\circ g)^{\inv}=g^{\inv}\circ f^{\inv}$.
Thus  the invertible elements of $\LL_{<\alpha}^{>\R}$ are exactly the elements of a group $G_{\alpha}$ with the group operation
given by composition. Our goal here is to identify $G_{\alpha}$
as a subset of $\LL_{<\alpha}^{>\R}$. By $\lambda_x=0$ and Lemma~\ref{logofcomp}, $\lambda_f=0$ is necessary for $f\in \LL_{<\alpha}^{>\R}$ to belong to $G_{\alpha}$. It is also sufficient: 

\begin{prop}\label{inversion} For $f\in \LL_{<\alpha}^{>\R}$ we have: $f\in G_{\alpha}\Leftrightarrow \lambda_f=0$. 
\end{prop}

\noindent
We begin by considering the series {\em tangent to the identity}. These are the
$x+h$ with $h\in \LL_{<\alpha}^{\prec x}$. Fix such $h$ and note that
then  
$$f\circ (x+h)\ =\ \sum_{n=0}^\infty \frac{1}{n!}f^{(n)}h^n\qquad (f\in \LL_{<\alpha})$$
and that the map $f\mapsto f\circ (x+h): \LL_{<\alpha} \to \LL_{<\alpha}$
equals $I+D$ where $I$ is the identity map on $\LL_{<\alpha}$ and the strongly $\R$-linear map
$D: \LL_{<\alpha} \to \LL_{<\alpha}$ is given by
$D(f) =\ \sum_{n=1}^{\infty}\frac{1}{n!}f^{(n)}h^n$.  Now $\supp \der_{\alpha}=\{\ell_{\beta}^\dagger:\ \beta<\alpha\}\preceq x^{-1}$, so $\supp \der_{\alpha}$
is well-based. Moreover, $\supp h\prec x$, so  $D$ has well-based support 
$$\supp D\ \subseteq\ \bigcup_{n=1}^\infty(\supp \der_{\alpha})^n\cdot (\supp h)^{n}\ \prec\ 1.$$ 
Thus by Lemma~\ref{invsupp} the map $I+D$ on $\LL_{<\alpha}$ is 
bijective with inverse $I+E$ where the strongly $\R$-linear map $E:\LL_{<\alpha} \to \LL_{<\alpha}$ is given by
$E(f)= \sum_{n=1}^{\infty} (-1)^nD^n(f)$ and has well-based support
$\supp E \subseteq \bigcup_{n=1}^\infty (\supp D)^n\prec 1$.

\begin{lemma}\label{tangid}  Let $h\in \LL_{<\alpha}^{\prec x}$. Then the operator $f\mapsto f\circ (x+h): \LL_{<\alpha}\to \LL_{<\alpha}$ maps $x+\LL_{<\alpha}^{\prec x}$
bijectively onto itself. In particular, $x+h$ is invertible.
\end{lemma}
\begin{proof} For $f=x+g$ with $g\in \LL_{<\alpha}^{\prec x}$ we have
$$f\circ (x+h)\ =\ (x+g)\circ(x+h)\ =\ x+h+g^*$$
with $g^*=g\circ (x+h)\prec x\circ (x+h)=x+h\asymp x$, so the above operator does map $x+\LL_{<\alpha}^{\prec x}$
injectively into itself. Conversely, with $g\in \LL_{<\alpha}^{\prec x}$ 
we use the above inverse $I+E$ of $I+D$ to get
$f:=(I+E)(x+g)$ with $f\circ (x+h)=x+g$. It remains to note that
$f=x+g + E(x+g)$ and $\supp E(x+g)\subseteq (\supp E)\supp (x+g)\prec x$, so
$E(x+g)\prec x$. 
\end{proof}

\noindent
Thus if $f,g\in \LL_{<\alpha}^{>\R}$ are tangent to the identity, then so
are $f\circ g$ and $g^{\inv}$: 
$$G_{\alpha,1}\ :=\ \{f\in \LL_{<\alpha}^{>\R}:\ f \text{ is tangent to the identity}\} \text{ is a subgroup of }G_{\alpha}.$$
Below we improve this by showing that $G_{\alpha,1}$ is a {\em normal\/} subgroup of $G_{\alpha}$.

\begin{lemma}\label{easyreduc} Let $f\in G_{\alpha}$ and $r\in \R^{>}$. Then $rf, f^r\in G_{\alpha}$.
\end{lemma}
\begin{proof} From $f\circ f^{\inv}=x$ we get
$rf\circ f^{\inv}=rx$. In view of $rx\circ r^{-1}x=x$, this gives $rf\circ (f^{\inv}\circ r^{-1}x)=x$. Likewise $f^r\circ (f^{\inv}\circ x^{1/r})=x$, using
$x^r\circ x^{1/r}=x$.
\end{proof}

\noindent
For $f,g,h\in \LL$ we use the notation $f=g+o(h)$ to mean $f-g\prec h$. So far we defined $\lambda_g$ only for $g\in \LL^{>\R}$. We now extend this to all $g\in \LL^\times$ in the obvious way:
$$\lambda_g\ := \min \sigma(\fd g)\ \text{ if }\ \fd g\ne 1, \qquad \lambda_g\ :=\ \infty\ \text{ if }\ \fd g=1, \text{ with }\infty>\alpha \text{ for all }\alpha.$$

\begin{lemma}\label{reduc} Let $\fm=\ell_1^{r_1}\ell_2^{r_2}\cdots=\prod_{1\le \beta<\alpha}\ell_{\beta}^{r_{\beta}}\in \mathfrak{L}_{<\alpha}$, and let $g\in \LL_{<\alpha}^{>}$, $\lambda_g>0$. Then 
$\fm\circ(xg)= \fm+o(\fm)$, and thus
$x\fm\circ (xg)=x\fm g+o(x\fm g)$.
\end{lemma} 
\begin{proof} We have $\fm\circ(xg)=\prod_{1\le \beta<\alpha}\big(\ell_{\beta}\circ (xg)\big)^{r_{\beta}}$, so it suffices to show:
$$\ell_{\beta}\circ (xg)\ =\ \ell_{\beta}+o(\ell_{\beta})\ \text{ for }1\le \beta < \alpha.$$
For $\omega\le \beta<\alpha$ this holds by Corollary~\ref{domconstant}. It holds for
$\beta=1$ by observing $\ell_1\circ(xg)=\log(xg)=\ell_1+\log g$ and
$\log g\preceq \ell_2\prec \ell_1$. An easy induction then
gives $\ell_n\circ(xg)=\ell_n+o(\ell_n)$ for all $n\ge 1$.
\end{proof}

\begin{cor}\label{g1a} Let $G_{\alpha}^1:=\{x\fm(1+\epsilon):\ \fm\in \mathfrak{L}_{<\alpha},\ \lambda_{\fm}>0,\ \epsilon\in \LL_{<\alpha}^{\prec 1}\}$.
Then $G_{\alpha}^1$ is a subgroup of $G_{\alpha}$. 
\end{cor}
\begin{proof} Note that $G_{\alpha}^1\supseteq G_{\alpha,1}$. Let $f,g\in G_{\alpha}^1$, so we have $\fm,\fn\in \mathfrak{L}_{<\alpha}$ with $\lambda_{\fm},\lambda_{\fn}>0$ and $\epsilon_1,\epsilon_2\in \LL_{<\alpha}^{\prec 1}$, such that
$$f\ =\ x\fm(1+\epsilon_1), \quad g\ =\ x\fn(1+\epsilon_2).$$ 
Now $x\fm\circ g=x\fm\circ x\fn(1+\epsilon_2)=x\fm\fn(1+\epsilon_2)+o(x\fm\fn)=x\fm\fn+o(x\fm\fn)$ by Lemma~\ref{reduc}, so $f\circ g=x\fm\fn+o(x\fm\fn)\in G_{\alpha}^1$. Given $f\in G_{\alpha}^1$ as before,
and taking $\fn:=\fm^{-1}$, the above shows that $f\circ x\fn=x+o(x)\in G_{\alpha,1}$, so $f\circ x\fn\circ h=x$ where $h\in G_{\alpha,1}$. Thus $f\in G_{\alpha}$ and $f^{\inv}=x\fn\circ h=x\fn+o(x\fn)\in G_{\alpha}^1$.  
\end{proof}

\noindent
Here is a  useful  way to summarize the proof of Corollary~\ref{g1a}: let $f,g\in \LL_{<\alpha}$ and $f=x\fm+o(x\fm)$,
$g=x\fn+o(x\fn)$ with $\fm,\fn\in \mathfrak{L}_{<\alpha},\ \lambda_{\fm},\lambda_{\fn}>0$. Then
$$ f\circ g\ =\ x\fm\fn+o(x\fm\fn), \qquad f^{\inv}\ =\ x/\fm + o(x/\fm).$$

\begin{proof}[Proof of Proposition~\ref{inversion}] Let $f\in \LL_{<\alpha}^{>\R}$.
The direction
$f\in G_{\alpha}\Rightarrow \lambda_f=0$ was already explained. For the converse, assume $\lambda_f=0$; our job is to derive 
$f\in G_{\alpha}$. By Lemma~\ref{easyreduc} we can arrange that
$f$ has leading coefficient $1$ with
$\fd f=x\fm$ and $\fm=\prod_{1\le \beta<\alpha} \ell_{\beta}^{r_{\beta}}$, so $f=x\fm+o(x\fm)\in G_{\alpha}^1\subseteq G_{\alpha}$.  
\end{proof}

\noindent
The proof shows that $G_{\alpha}= \{ax^b\circ g:\ a,b\in \R^{>},\ g\in G_{\alpha}^1\}$.
Note in this connection that $G_{\R,x}:=\{ax^b:\ a,b\in \R^{>}\}$ is a subgroup of 
$G_{\alpha}$: 
$$ax^b\circ sx^t\ =\ as^bx^{bt}\qquad (a,b,s,t\in \R^{>}).$$
It is easy to see that $G_{\alpha}^1$ is {\em not\/} a normal subgroup of $G_{\alpha}$. On the other hand:

\begin{cor} $G_{\alpha,1}$ is a normal subgroup of $G_{\alpha}$.
\end{cor}
\begin{proof} Let $f\in G_{\alpha,1}$. By the above description of
$G_{\alpha}$ it suffices to show that $g\circ f\circ g^{\inv}\in G_{\alpha,1}$, for all $g=sx^t$ with $s,t\in \R^{>}$, and for all $g\in G_{\alpha}^1$. 
For $s,t\in \R^{>}$ and $g=sx^t$ we have $g^{\inv}=ax^b$ with $a=s^{-1/t}$ and $b=1/t$, so with $f=x+o(x)$ we get $f\circ g^{\inv}=g^{\inv}+o(g^{\inv})=ax^b(1+\epsilon)$ with $\epsilon\prec 1$, and thus 
$$g\circ f\circ g^{\inv}\ =\ s(ax^b)^t(1+\epsilon)^t\ =\ x(1+\epsilon)^t\ =\ x+o(x)\in G_{\alpha,1}.$$
Next, let $g\in G_{\alpha}^1$, so $g=x\fm+o(x\fm)$ with $\fm\in \mathfrak{L}_{<\alpha},\ \lambda_{\fm}>0$. Then $g^{\inv}=x/\fm + o(x/\fm)$, so $f\circ g^{\inv}= (x+o(x))\circ (x/\fm +o(x/\fm))=x/\fm + o(x/\fm)$,
and thus $g\circ f\circ g^{\inv}=(x\fm+o(x\fm))\circ(x/\fm + o(x/\fm))=x+o(x)\in G_{\alpha,1}$. \end{proof}

\section{Uniqueness, Embedding $\LL$ into $\No$, and Final Remarks}\label{fs}

\noindent
We continue to let $\circ$ denote the composition on $\LL$ constructed in Sections~\ref{sec:comp1} and \ref{sec:comp2}. 
Corollary~\ref{coru} characterizes  this composition uniquely, but in the first subsection below we establish the
more elegant characterization given by Theorem~\ref{uc} from the introduction. Note that in Section~\ref{teci} (end of first subsection) we already observed that $\circ$ witnesses the existence part of Theorem~\ref{uc}.

 In the second subsection we indicate the natural embedding
of $\LL$ into $\No$, and in the last subsection we finish with some remarks.

\subsection*{Uniqueness} Let $\ast$  denote any composition on
$\LL$ and let $f,g,h$ range over $\LL$. 

\begin{lemma}\label{astcirc} Let $f\in \LL_{<\omega}$ and $g>\R$. Then $f\ast g=f\circ g$.
\end{lemma} 
\begin{proof} By induction on $n$ and using (CL2), (CL3) we obtain 
$$\ell_n\ast g\ =\log_n(g)\ =\ \ell_n\circ g.$$ Hence for
$\fm=\prod_n\ell_n^{r_n}\in \mathfrak{L}_{<\omega}$ we have
$\fm\ast g=\prod_n \log_n(g)^{r_n}=\fm \circ g$ by Lemma~\ref{lemuc}.
The rest is an application of (CL4) (strong linearity).
\end{proof}

\noindent
We say that {\em $\ast$ obeys the Chain Rule\/} if $(f\ast g)'=(f'\ast g)\cdot g'$ for all $f,g$ with $g>\R$.
We say that {\em $\ast$ admits Taylor expansion\/} if for all $f,g,h$ with $g>\R$ and $h\prec g$ the sum
$\sum_n \frac{f^{(n)}\circ g}{n!}h^n$ exists and equals $f\ast(g+h)$. Note that  if $\ast$ admits Taylor expansion, then, with
$\epsilon$ ranging over (sufficiently small) nonzero elements of $\LL$,
$$ f'\ =\  \lim_{\epsilon\to 0} \frac{f\ast(x+\epsilon)-f}{\epsilon}.$$

\begin{lemma}\label{teocr} If $\ast$ admits Taylor expansion, then it obeys the chain rule.
\end{lemma}
\begin{proof} Assume $\ast$ admits Taylor expansion, and let $g>\R$. Then 
$$f'\ast g\ =\ \lim_{\epsilon\to 0} \frac{f\ast(g+\epsilon)-f\ast g}{\epsilon}$$ as is easily verified. The usual argument shows that then $\ast$ obeys the chain rule: for all sufficiently small $\epsilon\ne 0$ we have $g\ast (x+\epsilon)\ne g$ and
$$\frac{(f\ast g)\ast(x+\epsilon)-f\ast g}{\epsilon}\ =\ \frac{f\ast (g\ast(x+\epsilon))-f\ast g}{g\ast (x+\epsilon)-g}\ \cdot\ \frac{g\ast (x+\epsilon)-g}{\epsilon}.$$
Now $g\ast(x+\epsilon)=g + \epsilon_g$ with $\epsilon_g\to 0$ as 
$\epsilon\to 0$; thus letting $\epsilon$ go to $0$ in the above displayed equality
yields $(f\ast g)'=(f'\ast g)\cdot g'$.  
\end{proof}

\noindent
{\em In the rest of this subsection we assume that $\ast$ admits Taylor expansion
and has the following property: for all $\beta,\gamma$, 
\begin{itemize}
\item $\ell_\gamma \ast \ell_{\omega^\beta}\ =\ \ell_{\omega^\beta+\gamma}$\ if $\gamma < \omega^{\beta+1}$;
\item $\ell_{\omega^{\beta+1}} \ast \ell_{\omega^\beta}\ =\ \ell_{\omega^{\beta+1}}-1$;
\item $\ell_{\omega^\gamma} \ast \ell_{\omega^\beta}$ has constant term $0$
if $\gamma>\beta$ is a limit ordinal.
\end{itemize}}

\noindent
We have to derive that then $f\ast g=f\circ g$, where $g>\R$.
Here is the main lemma:

\begin{lemma}\label{tecr} If $\rho>\omega^{\beta+1}$, then $\ell_{\rho}\ast \ell_{\omega^\beta}=\ell_{\rho} +\epsilon_{\rho}$ with $\epsilon_{\rho}\prec 1$.
\end{lemma}
\begin{proof} Set $\mu=\omega^{\beta +1}$.  For $\rho=\mu+1$ we have
\begin{align*} \ell_{\rho}* \ell_{\omega^\beta}\  &=\ \log(\ell_{\mu}* \ell_{\omega^\beta})\ =\ \log(\ell_{\mu}-1)\\  
&=\ \log\big(\ell_{\mu}(1-\ell_{\mu}^{-1})\big)\ =\ 
\log(\ell_{\mu}) +\log(1-\ell_{\mu}^{-1})\\
&=\ \ell_{\rho} + \epsilon_{\rho}, \qquad \epsilon_{\rho}\asymp \ell_{\mu}^{-1}\prec 1.
\end{align*}
Next, let $\rho>\mu+1$, and assume inductively that for every ordinal $\nu$ with $\mu < \nu< \rho$ we have 
$\ell_{\nu}*\ell_{\omega^\beta}= \ell_{\nu}+\epsilon_{\nu}$ with
$\epsilon_{\nu}\prec 1$, so $\ell_{\nu}*\ell_{\omega^\beta}= \ell_{\nu}(1+h_{\nu})$ with $h_{\nu}\prec \ell_{\nu}^{-1}$. Take $\gamma\ge \beta+1$ such that $\omega^\gamma\le \rho< \omega^{\gamma+1}$. We distinguish three cases; only in the second case do we use the full inductive assumption.

\medskip\noindent
{\bf Case $\rho=\omega^\gamma$ and $\gamma$ is a successor ordinal.} Then $\gamma=\xi+1$, $\xi\ge \beta+1$ and from $\ell_{\rho}= (\ell_{\rho}\ast \ell_{\omega^{\xi}})+1$ and $\ell_{\omega^\xi} \ast \ell_{\omega^\beta} = \ell_{\omega^\xi} +\epsilon$ with 
$\epsilon \preceq 1$,
 we obtain
\begin{align*}
\ell_{\rho} \ast \ell_{\omega^\beta}\ &=\ \big((\ell_{\rho}\ast \ell_{\omega^{\xi}})+1\big)\ast \ell_{\omega^\beta}\ =\ \big(\ell_{\rho} \ast(\ell_{\omega^{\xi}}\ast \ell_{\omega^\beta})\big)+1\\
&=\ \big((\ell_{\rho}\ast(\ell_{\omega^{\xi}}+\epsilon)\big) +1\
=\ \left(\sum_{n=0}^\infty \frac{\ell_{\rho} ^{(n)}\ast \ell_{\omega^{\xi}}}{n!}\epsilon^n \right)+1\\ 
&=\ \ell_{\rho} + \epsilon_{\rho}, \quad \epsilon_{\rho}\prec 1,
\end{align*}
since $\ell_{\rho}^{(n)}\prec 1$ for $n\ge 1$.

\medskip\noindent
{\bf Case $\rho=\omega^\gamma$ and $\gamma$ is a limit ordinal.}  By Lemma~\ref{teocr} the composition $\ast$ obeys the Chain Rule, so by our assumption that the constant term of $\ell_{\rho}\ast \ell_{\omega^\beta}$ is $0$:
$$\ell_{\rho}\ast \ell_{\omega^\beta}\ =\ \int[(\ell_{\rho}'\ast \ell_{\omega^\beta})\cdot \ell_{\omega^\beta}'].$$ 
Now $\ell_{\rho}'=\prod_{\nu<\rho} \ell_{\nu}^{-1}$ and likewise for $\ell_{\omega^\beta}'$, so
\begin{align*} (\ell_{\rho}'\ast \ell_{\omega^\beta})\cdot \ell_{\omega^\beta}'\ &=\ \big((\prod_{\nu<\rho} \ell_{\nu}^{-1})\ast \ell_{\omega^\beta}\big)\cdot \prod_{\nu< \omega^\beta}\ell_{\nu}^{-1}\\
&=\ \prod_{\nu<\mu}(\ell_{\nu}^{-1}\ast \ell_{\omega^\beta})\cdot \prod_{\mu\le \nu < \rho}(\ell_{\nu}^{-1}\ast \ell_{\omega^\beta})\cdot \prod_{\nu<\omega^\beta}\ell_{\nu}^{-1}\\
&=\ \prod_{\nu<\mu} \ell_{\omega^\beta +\nu}^{-1}\cdot \prod_{\mu\le \nu < \rho}(\ell_{\nu}^{-1}\ast \ell_{\omega^\beta})\cdot \prod_{\nu<\omega^\beta}\ell_{\nu}^{-1}\\
&=\ \prod_{\nu<\mu} \ell_{\nu}^{-1}\cdot \prod_{\mu\le \nu < \rho}(\ell_{\nu}^{-1}\ast \ell_{\omega^\beta})\\
&=\  \big(\prod_{\nu<\mu} \ell_{\nu}^{-1}\big)\cdot (\ell_{\mu}^{-1}\ast \ell_{\omega^\beta})\cdot \prod_{\mu< \nu < \rho}(\ell_{\nu}^{-1}\ast \ell_{\omega^\beta}).
\end{align*} 
Now  $\ell_{\mu} \ast \ell_{\omega^\beta}\ =\ \ell_{\mu}-1$ gives $\ell_{\mu}^{-1}\ast \ell_{\omega^\beta}=\ell_{\mu}^{-1}(1-\ell_{\mu}^{-1})^{-1}$. Together with the equality derived above and the inductive assumption this yields
\[ (\ell_{\rho}'* \ell_{\omega^\beta})\cdot \ell_{\omega^\beta}'\ =\ \ell_{\rho}'(1-\ell_{\mu}^{-1})^{-1}\prod_{\mu<\nu<\rho}(1+h_{\nu})^{-1}\ =\ \ell_{\rho}'(1+h)\]
with $h\prec \ell_{\rho}^{-2}$. 
Hence $\ell_{\rho}\ast \ell_{\omega^\beta}=\ell_{\rho} + \int \ell_{\rho}'h$. Now $\int \ell_{\rho}'\ell_{\rho}^{-2}=-\ell_{\rho}^{-1}$, so
$\int \ell_{\rho}'h =\epsilon_{\rho}$ with $\epsilon_{\rho}\prec \ell_{\rho}^{-1}\prec 1$.

\medskip\noindent
{\bf Case $\rho >\omega^\gamma$.} Then
$\rho = \omega^\gamma + \nu$ where $0<\nu < \omega^{\gamma+1}$, so $\ell_\rho = \ell_\nu\ast \ell_{\omega^\gamma}$. Now $\ell_{\omega^\gamma} \ast \ell_{\omega^\beta} = \ell_{\omega^\gamma} +\epsilon$ with $\epsilon=-1$ if $\gamma=\beta+1$ and $\epsilon \prec 1$
if $\gamma>\beta+1$. Thus
\begin{align*}
\ell_\rho \ast \ell_{\omega^\beta}\ &=\ \ell_{\nu}\ast (\ell_{\omega^\gamma} \ast \ell_{\omega^\beta})\ =\ \ell_{\nu} \ast( \ell_{\omega^\gamma} +\epsilon)\ =\ \sum_{n=0}^\infty \frac{\ell_{\nu}^{(n)}\ast \ell_{\omega^\gamma}}{n!}\epsilon^n\\
&=\  (\ell_{\nu}\ast \ell_{\omega^\gamma})+ \epsilon_{\rho}\ =\ 
\ell_\rho + \epsilon_{\rho}\quad \text{ where }\ \epsilon_{\rho} \prec 1. \qedhere
\end{align*}
\end{proof}

\begin{lemma}
$f \ast \ell_{\omega^\beta} = f \circ \ell_{\omega^\beta}$. 
\end{lemma}
\begin{proof} By the usual reductions
it suffices to verify the identity for hyperlogarithms $f=\ell_\rho$. For $\rho \leq \omega^{\beta+1}$ our assumptions on $\ast$ take care of this. Let $\rho > \omega^{\beta+1}$ and assume inductively that $\ell_{\nu}\ast \ell_{\omega^\beta}=\ell_{\nu}\circ \ell_{\omega^\beta}$ for all $\nu<\rho$.  Then by the chain rule, $(\ell_\rho\ast \ell_{\omega^\beta} )'= (\ell_\rho \circ \ell_{\omega^\beta})'$. By Lemma~\ref{tecr}, $\ell_\rho \ast \ell_{\omega^\beta}$ and $\ell_\rho \circ \ell_{\omega^\beta}$ both have the constant term $0$, so they are equal.
\end{proof}

\noindent
We now finish the proof that $f\ast g=f\circ g$ for all $f,g$ with $g>\R$.
First, for nonzero $\gamma$ we
have $\gamma=\omega^{\beta_1} + \cdots + \omega^{\beta_k}$
with $\beta_1\ge \beta_2\ge \dots \ge \beta_k$, $k\ge 1$.
For $k=1$ we have $f\ast \ell_{\gamma}=f\circ \ell_{\gamma}$ by the last lemma. For $k>1$ we have $\gamma=\omega^{\beta_1}+\nu$ with 
$\nu=\omega^{\beta_2}+\cdots +\omega^{\beta_k}<\omega^{{\beta_1}+1}$, and thus 
$$f\ast \ell_{\gamma}\ =\ f\ast(\ell_{\nu}\ast \ell_{\omega^{\beta_1}})\ =\ (f\ast \ell_{\nu})\ast \ell_{\omega^{\beta_1}}\ =\ (f\circ \ell_{\nu})\circ \ell_{\omega^{\beta_1}}\ =\ f\circ \ell_{\gamma},$$
where for the third equality we use an obvious induction assumption on $k$.
We have now shown that $f\ast \ell_{\gamma}=f\circ \ell_{\gamma}$ for all
$f$ and $\gamma$. 

Next, let $f\in \LL_{\ge \omega}$ and $g>\R$. In Section~\ref{sec:comp2} we defined
$f^{\uparrow 3}\in \LL_{\ge \omega}$ and observed that 
$\log_3(g)=\ell_{\gamma}+\epsilon$ with
$\gamma=\lambda_g+3$ and $\epsilon\prec 1$. Then 
$f=f^{\uparrow 3}\circ \ell_3=f^{\uparrow 3}\ast \ell_3$. Using also 
$\ell_3\circ g=\log_3(g)=\ell_3\ast g$ we obtain
\begin{align*} f\circ g\ &=\ (f^{\uparrow 3}\circ \ell_3)\circ g\ =\ f^{\uparrow 3}\circ(\ell_3\circ g)\ =\ f^{\uparrow 3}\circ (\ell_\gamma+\epsilon),\\  
f\ast g\ &=\ (f^{\uparrow 3}\ast \ell_3)\ast g\ =\ f^{\uparrow 3}\ast(\ell_3\ast g)\ =\ f^{\uparrow 3}\ast(\ell_{\gamma}+\epsilon),
\end{align*}
which by Taylor expansion yields $f\circ g=f\ast g$.

Finally, for arbitrary $f,g$ with $g>\R$ we have $f=\sum_{\fm\in \mathfrak{L}_{<\omega}}f_{[\fm]}\fm$ where all $f_{[\fm]}$ lie in $\LL_{\ge \omega}$. In view of Lemma~\ref{astcirc} this gives 
$$f\circ g\ =\ 
\sum_{\fm\in \mathfrak{L}_{<\omega}}(f_{[\fm]}\circ g)(\fm\circ g)\
=\ \sum_{\fm\in \mathfrak{L}_{<\omega}}(f_{[\fm]}\ast g)(\fm\ast g)\
=\ f\ast g.$$ 

\medskip\noindent
This concludes the proof of Theorem~\ref{uc}. Note that for the above proof of $\ast=\circ$ we only needed the Taylor  identity
for $f\ast(g+h)$ with $g>\R$ and $h\preceq 1$.

\subsection*{Embedding $\LL$ into $\No$}
We use here \cite{ADH1+} and its notations, viewing $\No$ as a logarithmic-exponential field extension of $\R$ in the usual way. In that paper we defined 
$\log_{\alpha}\omega:=\omega^{\omega^{-\alpha}}\in \No$, thinking of it as the ``$\alpha$ times iterated function $\log $ evaluated at $\omega$'' in view of 
$\log (\log_{\alpha}\omega)=\log_{\alpha+1}\omega$. Berarducci and Mantova \cite{BM} constructed a derivation on $\No$ which in \cite{ADH1+} and here
we denote by $\der_{\operatorname{BM}}$. 

\begin{prop}\label{embllno} There is a unique ordered field embedding $\iota: \LL \to \No$ such that: \begin{enumerate}
\item[(i)] $\iota$ is the identity on $\R$ and $\iota(\ell_{\alpha})=\log_{\alpha}\omega$ for all $\alpha$;
\item [(ii)] for every summable family $(f_i)_{i\in I}$ in $\LL$ the family $\iota(f_i)$ is summable in $\No$ and $\iota(\sum_i f_i)=\sum_i\iota(f_i)$;
\item[(iii)] $\iota(\log f)=\log \iota(f)$ for all $f\in \LL^{>}$.
\end{enumerate}
This embedding also preserves the derivation:   $\iota(f')=\der_{\operatorname{BM}}(\iota(f))$ for all $f\in \LL$.
\end{prop}
\begin{proof} In \cite[Section 2]{ADH1+} we defined for any summable family $(a_i)$ in $\No$ the product
$\prod_i\omega^{a_i}:=\omega^{\sum_i a_i}$.  In \cite[remarks preceding lemma 3.3] {ADH1+} we  defined $a^r:= \exp(r\log a)$ for $a\in \No^{>}$ and
$r\in \R$, and recorded the fact that $(\log_{\alpha}\omega)^r=\omega^{r\omega^{-\alpha}}$ for $r\in \R$.
Thus for any logarithmic hypermonomial $\prod_{\beta<\alpha}\ell_{\beta}^{r_{\beta}}$ of $\LL$
we have a  product 
$$\prod_{\beta<\alpha} (\log_{\beta}\omega)^{r_{\beta}}\ =\ \prod_{\beta<\alpha} \omega^{r_{\beta}\omega^{-\beta}}\ =\ \omega^{\sum_{\beta<\alpha}r_{\beta}\omega^{-\beta}}.$$
It is routine to check that this yields a unique $\R$-linear map 
$\iota: \LL\to \No$ such that for every logarithmic hypermonomial $\prod_{\beta<\alpha}\ell_{\beta}^{r_{\beta}}$ we have
$$\iota(\prod_{\beta<\alpha}\ell_{\beta}^{r_{\beta}})\ =\ \prod_{\beta<\alpha} (\log_{\beta}\omega)^{r_{\beta}}$$ and
for every summable family $(f_i)_{i\in I}$ in $\LL$ the family $\iota(f_i)$ is summable in $\No$ and $\iota(\sum_i f_i)=\sum_i\iota(f_i)$. It is easy to verify that this map $\iota$ is an ordered field embedding satisfying (i) and (ii). It also satisfies (iii) in view of \cite[Lemma 2.3]{BM}. As to uniqueness, let $i$ also be an ordered field embedding
satisfying (i), (ii), (iii) with $i$ instead of $\iota$. Then for 
$\fm=\prod_{\beta<\alpha}\ell_{\beta}^{r_{\beta}}$ we have $\log \fm=\sum_{\beta<\alpha}r_{\beta}\ell_{\beta+1}$. Therefore, using again \cite[Lemma 2.3]{BM},
$$\log i(\fm)\ =\ i(\log \fm)\ =\ \sum_{\beta<\alpha}r_{\beta}\log_{\beta+1}\omega\ =\ \log \prod_{\beta<\alpha}( \log_{\beta}\omega)^{r_{\beta}},$$
so $i(\fm)=\prod_{\beta<\alpha} (\log_{\beta}\omega)^{r_{\beta}}=\iota(\fm)$. Thus
$i=\iota$.
That $\iota$ is also an embedding of differential fields with the derivation
$\der_{\operatorname{BM}}$ on $\No$ uses the fact that 
$$\der_{\operatorname{BM}}(\log_{\alpha}\omega)\ =\ 1/\prod_{\beta<\alpha}\log_{\beta}\omega,$$ 
for which we refer to \cite[two lines before Lemma 2.10]{ADH1+}. 
\end{proof}

\noindent
In \cite{ADH1+} we also defined a canonical embedding of the differential
field $\T$ into $\No$, and we observe here that on $\T_{\log}=\LL^{\cup}_{<\omega}$
this embedding agrees with the embedding of Proposition~\ref{embllno}.

\subsection*{Final Remarks} One  issue we didn't touch is the monotonicity of composition on the right: for $f,g,h\in \LL^{>\R}$ with $g < h$, do we have
$f\circ g < f\circ h$? We believe this to be true, and it would reflect how composition behaves for germs of functions in Hardy fields. But we only have proofs
for special cases. It might be better to deal with this in the wider setting of the (conjectural) field $\H$ of all hyperseries where 
every  positive infinite element
should have a compositional inverse.  

\medskip\noindent
It would be interesting to represent right composition with various $g\in \LL^{>\R}$  on certain
subfields of $\LL$ in the form $\ex^{\phi\der}$ for suitable $\phi\in \LL$, as we did for $g=\ell_{\omega^\beta}$ in the remark following
the proof of Lemma~\ref{car}. 

\medskip\noindent
The identity $\ell_{\omega^{\beta+1}}\circ \ell_{\omega^{\beta}}=\ell_{\omega^{\beta+1}}-1$ reflects a choice of integration constant $-1$. It is surely the
most natural choice, but for any family $(c_{\beta})$ of real numbers there is a composition $\ast$ on $\LL$ such that instead for all
$\beta$,
$$\ell_{\omega^{\beta+1}}\ast\ell_{\omega^{\beta}}\ =\ \ell_{\omega^{\beta+1}}+ c_{\beta}. $$
Such a composition $\ast$ is obtained by replacing (\ref{e:PartialTaylor}) with 
$$f \ast \ell_{\omega^\beta}\ :=\ \sum_{n=0}^\infty \frac{c_{\beta}^n}{n!}\derdelta^n f \qquad (f\in \LL_{[\mu,\alpha)})$$ 
and following otherwise the definitions in Sections~\ref{sec:comp1} and \ref{sec:comp2}.  
Theorem~\ref{uc} goes through for $\ast$ in the role of $\circ$, except that the above identity involving the constants $c_{\beta}$ replaces 
``$\ell_{\omega^{\beta+1}}\circ\ell_{\omega^{\beta}}\ =\ \ell_{\omega^{\beta+1}}-1$''. The proofs for $\circ$ are easily adapted to $\ast$.  Note that any $c_{\beta}\ge 0$ would give a failure of monotonicity of $\ast$ on the right. 

\medskip\noindent
Another topic is the connection to Hardy fields. Kneser~\cite{K}  yields a real analytic function $\ell_{\text{K}}: \R \to \R^{>}$ with  
$ \ell_{\text{K}}(\log t ) = \ell_{\text{K}}(t)-1$ for $t>0$; its germ at $+\infty$,
also denoted by $\ell_{\text{K}}$ below, generates a Hardy field extension of $\R(x, \log x, \log_2 x,\dots)$ such that $\R < \ell_{\text{K}} < \log_n(x)$ for all $n$, with
$x$ here the germ of the identity function on $\R$. Clearly, $\ell_{\text{K}}$ has $\ell_{\omega}$ as a kind of formal counter part.
In the appendix to~\cite{Schm01}, Schmeling constructs likewise for all $n>1$ a real analytic function with $\ell_{\omega^n}$
as a formal counter part. Much remains to be done to strengthen this connection.  There is ongoing work along these lines
with partial results announced in \cite{ADH2} . 



\end{document}